\documentclass[10pt]{amsart}
\usepackage{amssymb,amsmath,amsthm,bbm,amsfonts,amsbsy}
\usepackage{stmaryrd}
\usepackage{cancel}
\usepackage[all]{xy}
\usepackage[shortlabels]{enumitem}
\usepackage{soul}
\usepackage[colorlinks,linkcolor=blue,citecolor=blue,urlcolor=blue]{hyperref}
\usepackage[textwidth=1.8cm]{todonotes}  % option '[disable]' to remove the notes, [textwidth=2.5cm]
\usepackage{marginnote}
\usepackage{zref-savepos}

\newcount\todocount

\newcommand{\checkxpos}[3][]{%
	\ifdim \zposx{#2}sp < 20000000sp%
	\mynote[#1]{#3}%
	\else%
	\note[#1]{#3}%
	\fi%
}

\newcommand{\mytodo}[2][]{%
	\zsaveposx{todo\the\todocount}%
	\checkxpos[#1]{todo\the\todocount}{#2}%
	\global\advance\todocount1\relax
}

\newcommand{\mynote}[2][]{{%
		\let\marginpar\marginnote
		\reversemarginpar
		\renewcommand{\baselinestretch}{0.8}%
		\todo[#1]{#2}}}

\newcommand{\note}[2][]{\renewcommand{\baselinestretch}{0.8}\todo[#1]{#2}}

\numberwithin{equation}{section}

\newtheorem{thm}{Theorem}[section]
\newtheorem{cor}[thm]{Corollary}
\newtheorem{prop}[thm]{Proposition}
\newtheorem{lem}[thm]{Lemma}

\theoremstyle{definition}
\newtheorem{definition}[thm]{Definition}
\newtheorem{rem}[thm]{Remark}

\theoremstyle{remark}

\allowdisplaybreaks[1]

\newcommand{\im}{\mathbf{i}}
\newcommand{\dd}{\text{d}}

\newcommand{\ad}{\text{\textnormal{ad}}}
\newcommand{\Ad}{\text{\textnormal{Ad}}}

\newcommand{\g}{\mathfrak{g}}
\newcommand{\m}{\mathfrak{m}}
\newcommand{\E}{\mathcal{E}}
\newcommand{\Aut}{\text{{\sc Aut}}}
\newcommand{\Hom}{\text{{\sc Hom}}}

\begin{document}

\title{Absolute parallelism for $2$-nondegenerate CR structures via bigraded Tanaka prolongation}
\author
{Curtis Porter}
\address{Curtis Porter\\
Department of Mathematics\\
Duke University\\
Durham\\
NC \ 27710\\
USA}
\email{cwp19@math.duke.edu}

\author{Igor Zelenko}
\address{Igor Zelenko\\
		Department of Mathematics\\
	Texas A\&M University\\
	College Station\\
	Texas \ 77843\\
	USA}
	
\email{zelenko@math.tamu.edu}
\urladdr{\url{http://www.math.tamu.edu/~zelenko}} 
 
\thanks{ I.\ Zelenko was partly supported by NSF grant DMS-1406193 and Simons Foundation Collaboration Grant for Mathematicians 524213.}

\begin{abstract}
This article is devoted to the local geometry of everywhere $2$-nondegenerate CR manifolds $M$ of hypersurface type. An absolute parallelism for such structures was recently constructed independently by Isaev-Zaitsev, Medori-Spiro, and Pocchiola in the minimal possible dimension ($\dim M=5$), and for $\dim M=7$ in certain cases by the first author. In this paper, we develop a bigraded (i.e. $\mathbb Z\times \mathbb Z$-graded) analog of Tanaka's prolongation procedure to construct an absolute parallelism for these CR structures in arbitrary (odd) dimension with Levi kernel of arbitrary admissible dimension.  We introduce the notion of a bigraded Tanaka symbol -- a complex bigraded vector space -- containing all essential information about the CR structure. Under the additional regularity assumption that the symbol is a Lie algebra, we define a bigraded analog of the Tanaka universal algebraic prolongation, endowed with an anti-linear involution, and prove that for any CR structure with a given regular symbol there exists a canonical absolute parallelism on a bundle whose dimension is that of the bigraded universal algebraic prolongation. Moreover, we show that for each regular symbol there is a unique (up to local equivalence) such CR structure whose algebra of infinitesimal symmetries has maximal possible dimension, and the latter algebra is isomorphic to the real part of the bigraded universal algebraic prolongation of the symbol. In the case of $1$-dimensional Levi kernel we classify all regular symbols and calculate their bigraded universal algebraic prolongations.
In this case, the regular symbols can be subdivided into nilpotent, strongly non-nilpotent, and weakly non-nilpotent. The bigraded universal algebraic prolongation of strongly non-nilpotent regular symbols is isomorphic to the complex orthogonal algebra $\mathfrak{so}\left(m,\mathbb C\right)$, where $m=\tfrac{1}{2}(\dim M+5)$. Any real form of this algebra -- except   $\mathfrak{so}\left(m\right)$ and $\mathfrak{so}\left(m-1,1\right)$ -- corresponds to the real part of the bigraded universal algebraic prolongation of exactly one strongly non-nilpotent regular CR symbol. However, for a fixed $\dim M\geq 7$ the dimension of the bigraded universal algebraic prolongations of all possible regular CR symbols achieves its maximum on one of the nilpotent regular symbols, and this maximal dimension is $\tfrac{1}{4}(\dim M-1)^2+7$.
 \end{abstract}
\maketitle\markboth{Curtis Porter and Igor Zelenko}{On absolute parallelism for $2$-nondegenerate CR structures via bigraded Tanaka prolongation}

\tableofcontents

%%%%%%%%%%%%%%%%%%%%%%%%%%%%%%%%%%%%%%%%%%%%%%%%%%%%%%%%%%%%%%%%%%%%%%%%%%%%%%%%%%%%%%%%%%%

\section{Introduction}

%%%%%%%%%%%%%%%%%%%%%%%%%%%%%%%%%%%%%%%%%%%%%%%%%%%%%%%%%%%%%%%%%%%%%%%%%%%%%%%%%%%%%%%%%%%
\subsection{Preliminaries}
A CR structure on a manifold $M$ is a distribution $D\subset TM$ of even rank, together with an isomorphism field $J:D\rightarrow D$ which satisfies $J^2=-\mathrm{Id}$, such that the following integrability condition holds: if the complexification $\mathbb{C}D$ is split into the $\im$-eigenspace $H$ of $J$ and the $(-\im)$-eigenspace $\overline H$ ($\im=\sqrt{-1}$), then the distribution $H$ -- and therefore $\overline H$ -- is involutive. At their generic points, real, codimension $c$ submanifolds of $\mathbb C^{n+c}$ are endowed with the natural CR structure determined by rank $n$, complex subbundles of their tangent bundles. When $D$ has corank $1$, the CR structure is said to be of \emph{hypersurface type}. The Levi form is a $\mathbb{C}TM/\mathbb{C}D$-valued Hermitian form $\mathfrak{L}$ defined on $H$ by
$$\mathfrak {L}(X, Y)=\im[X,\overline{Y}]\mod \mathbb{C}D\quad  X, Y \in \Gamma(H).$$
The \emph{Levi kernel} $K\subset H$  consists of all vectors in $H$ which are degenerate for the Levi form. CR-structures with $K=0$ are called Levi-nondegenerate.

The equivalence problem for Levi-nondegenerate CR structures of hypersurface type is classical: E. Cartan solved it for hypersurfaces in $\mathbb C^2$ \cite{cartanCR}, then Tanaka \cite{tanakaCR} and Chern and Moser \cite{chernmoserCR} generalized the solution to hypersurfaces in $\mathbb{C}^{n+1}$ for $n\geq1$. This case is well understood in the general framework of parabolic geometries \cite{tanaka2, capshichl, capslovak}.
If the Levi form has signature $(p,q)$ for $p+q=n$, then the maximally symmetric model is obtained as a complex projectivization of the cone of nonzero vectors in $\mathbb C^{n+2}$ which are isotropic with respect to a Hermitian form of signature $(p+1, q+1)$, and the algebra of infinitesimal symmetries of this model is isomorphic to $\mathfrak {su}(p+1, q+1)$.
%The maximally symmetric model here is a sphere in $\mathbb C^n$.

The notion of $k$-nondegeneracy is a natural generalization of Levi-nondegeneracy, so that Levi-non-degenerate CR structures are exactly $1$-nondegenerate in this terminology. Higher order nondegeneracy requires Levi-degeneracy, so we assume the fiber $K_x$ over each $x\in M$ is nontrivial. Moreover, we assume $\dim K_x>0$ is constant so that $K$ is a distribution. This allows for a convenient description of $2$-nondegeneracy as follows (see the monograph \cite[chapter 11]{BER99} for the more general definition). For $v\in K_x$ and $y\in\overline{H}_x/\overline{K}_x$, take local sections $V\in\Gamma(K)$ and $Y\in\Gamma(\overline{H})$ such that $V(x)=v$ and $y\equiv Y(x)\mod\mathbb{C}D$, and define a linear map\footnote{Though $v$ is not yet identified with an element of a Lie algebra, the operator $\ad_v$ is naturally induced by the adjoint action of the vector field $V$ on the Lie algebra of vector fields on $M$. The notation $\ad_v$ converges to its standard usage in \S\ref{symbolsec}.}
\begin{equation}\label{adv}
\begin{aligned}
 \ad_v&:\overline{H}_x/\overline{K}_x\to H_x/K_x,\\
 &\hspace{7mm}y\hspace{5mm}\mapsto [V,Y]|_x \,\, \mathrm{mod}\,\, K_x\oplus \overline {H}_x.
\end{aligned}
\end{equation}
One can similarly define a linear map $\ad_{v}:H_x/K_x\to \overline{H}_x/\overline{K}_x$ for $v\in \overline K_x$ (or simply take complex conjugates). A Levi-degenerate CR structure is \emph{$2$-nondegenerate at $x$} if $\ad_v\neq 0$ for $v\in K_x$ (or $v\in \overline{K}_x$) unless $v=0$. It is shown in \cite[Appendix]{kaupzaitsev} that the definition of $k$-nondegeneracy in \cite{BER99} is equivalent to ours when $K$ has constant rank (and for arbitrary $k$ under the analogous constant rank assumption for the fibers of a special filtration in $K$ called the \emph{Freeman sequence} \cite{freeman}).

Among hypersurface-type CR manifolds, the lowest dimension in which $2$-nondegeneracy can occur is $\dim M=5$. A natural candidate for the maximally symmetric model is given by a tube hypersurface in $\mathbb C^3$ over the future light cone in $\mathbb R^3$,
\begin{equation}
\label{tube}
M_0=\{(z_1, z_2, z_3)\in\mathbb C^3: (\mathrm{Re}\, z_1)^2+(\mathrm{Re}\, z_2)^2-(\mathrm{Re}\, z_3)^2=0, \mathrm{Re}\, z_3>0\}
\end{equation}
 and this model has been extensively studied in \cite{felskaup1, felskaup2, kaupzaitsev, merker}. In particular, its algebra of infinitesimal symmetries is isomorphic to $\mathfrak {so}(3,2)$. However, the structure of absolute parallelism in this situation was constructed only recently  and independently in the following three papers (preceded by the work \cite{ebenfelt} for a more restricted class of structures): by Isaev and Zaitsev \cite{isaev},  Medori and Spiro \cite{medori} and  Merker-Pocchiola \cite{pocchiola}. The question of existence of a functorially constructed absolute parallelism which is a Cartan connection in this case is rather subtle and is addressed in more detail in \cite{medori2},\cite[section 4.3]{Greg}, and \cite{Greg2}.

The only results about an absolute parallelism for $2$-nondegenerate, hypersurface-type CR structures of dimension higher than $5$ were obtained by the first author in \cite{porter} in the case of $7$-dimensional CR manifolds with $\text{rank}_\mathbb{C}K=1$, under certain additional algebraic assumptions that are automatic in $5$-dimensional case, and
in more recent work \cite{Greg} about structures modeled on a homogeneous space of a simple group, there called the ``maximally symmetric model." 

In this article we develop a unified framework for the construction of an absolute parallelism for $2$-nondegenerate, hypersurface-type CR structures on manifolds of arbitrary odd dimension $\geq5$ and with Levi kernel of arbitrary admissible dimension, under certain additional algebraic assumptions. These algebraic assumptions are automatic in dimension $5$, and in dimension $7$ they include all cases treated in \cite{porter}, along with an additional case only mentioned therein. The algebraic assumption in dimension $7$ also matches the homogeneous models with simple symmetry groups discussed by A. Santi \cite{santi}. In contrast to \cite{Greg}, we do not assume that our structures are modeled on ``maximally symmetric" homogeneous spaces of simple groups, and in fact many of our structures -- including the most symmetric one in any given class under consideration -- are modeled on homogeneous spaces whose symmetry groups are not even semisimple.

Our method is a modification of Tanaka's algebraic and geometric procedures from his 1970 paper \cite{tanaka1}.
Note that the method of Medori and Spiro is also inspired by Tanaka, modifying the techniques of his 1979 paper \cite{tanaka2} devoted to parabolic geometries in order to apply them to the non-parabolic one in  \cite{medori}. However, their construction seems specific to the concrete case under consideration, using the properties of the algebra $\mathfrak{so}(3,2)$ in particular, and the method does not appear to be easily extended to $2$-nondegenerate CR structures  in higher dimensions.

\subsection{Tanaka's algebraic and geometric prolongation procedures}
Our initial objective was to follow the scheme of Tanaka's 1970 paper \cite{tanaka1}, in which he developed a deep generalization of the theory of $G$-structures that is ideally adapted to nonholonomic structures. Let us briefly describe the main steps of the Tanaka constructions relevant to us here (for details see the original paper \cite{tanaka1} and  also \cite{ mori,aleks, zeltan}). Here and in the sequel, ``graded vector space" means $\mathbb Z$-graded.

 First, a distribution $D\subset TM$ generates a filtration in the tangent space at every point $x\in M$ by taking iterative Lie brackets of its sections. Under the regularity assumption that the dimensions of the corresponding subspaces in this filtration are independent of $x$, Tanaka observed  that after passing from this filtered structure to the corresponding graded object, one can  assign to the distribution at $x$ a negatively ($\mathbb Z_-$) graded nilpotent Lie algebra
$\displaystyle{\g_{-}(x)=\oplus_{i<0} \g_i(x)}$
called the \emph{Tanaka symbol of $D$ at $x$}, which contains the information about the principal parts of all commutators of vector fields taking values in $D$. Tanaka considered distributions with constant symbols or ``of constant type $\g_{-}$", when the graded Lie algebras $\g_{-}(x)$ are isomorphic for every $x\in M$ to a fixed, graded, nilpotent Lie algebra $\g_{-}$, generated by $\g_{-1}$. The \emph{flat distribution  $D (\g_-)$ of constant type $\g_-$} is the left-invariant distribution on the simply connected Lie group $G_-$ with Lie algebra $\g_-$ such that the fiber of $D$ at the identity is $\g_{-1}$.

Now let $\mathrm{Aut}(\g_-)$ be the group of automorphisms of the graded Lie algebra $\g_{-}$ and $\mathfrak{der}(\g_{-})$ the Lie algebra of $\mathrm{Aut}(\g_-)$, so that $\mathfrak{der}(\g_{-})$ is the algebra of all derivations of $\g_{-}$ which preserve the grading. Assigning weight $0$ to the space $\mathfrak{der}(\g_{-})$, the vector space $\g_{-}\oplus \mathfrak{der}(\g_{-})$ is naturally endowed with the structure of a graded Lie algebra. For a distribution of type $\g_{-}$ one can construct a principal $\mathrm{Aut}(\g_-)$-bundle $P^0(\g_{-})$ over $M$ whose fiber over $x$ consists of all graded Lie algebra isomorphisms $\g_{-}\to\g_{-}(x)$.

Additional structures on $D$ can be encoded in the choice of a subgroup $G_0\subset\mathrm{Aut}(\g_{-})$ with Lie algebra $\g_0$, leading to a $G_0$-reduction $P^0$ of the bundle $P^0(\g_-)$. The bundle $P^0$ is called the structure of type $\g_{-}\oplus \g_0$, or of Tanaka symbol $\g_{-}\oplus \g_0 $. One can define the \emph{flat structure of type $\g_{-}\oplus \g_0 $} as the $G_0$-reduction $P^0$ of the bundle $P^0(\g_-)\to G_-$ such that the fiber of $P^0$ over $x\in G_-$ consists of pullbacks of isomorphisms in the fiber over the identity of $G_-$ under the left translation mapping $x$ to the identity. Note that if $G^0$ denotes a Lie group with Lie algebra $\g_-\oplus\g_0$ such that $G_0\subset G^0$, then the flat structure of type $\g_-\oplus\g$ is at least locally equivalent to the structure of type $\g_-\oplus\g_0$ given by the bundle $G^0\rightarrow G^0/G_0$.

Next, Tanaka defines the \emph{universal algebraic prolongation}
\begin{equation}
\label{UTan0}
{\mathfrak U} (\g_{-}\oplus \g_0)=\g_{-}\oplus \g_0\oplus \bigoplus_{i>0}\g_i
\end{equation}
of $\g_{-}\oplus \g_0$,
which is the maximal (nondegenerate) graded Lie algebra containing $\g_{-}\oplus \g_0$ as its nonpositive part.  Nondegeneracy here means the adjoint action $\ad( y)|_{\g_{-}}$ is nontrivial for any nonzero, nonnegatively graded $ y\in {\mathfrak U} (\g_{-}\oplus \g_0)$. The prolongation procedure to construct canonical frames for structures of type $\g_{-}\oplus \g_0 $ can be described uniformly by the following

\begin{thm} [Tanaka, \cite{tanaka1}]
\label{Tanthm}
Assume that $\dim \,\mathfrak U(\g_-\oplus\g_0)<\infty$.
Then the following holds:
\begin{enumerate}
\item
To any Tanaka structure of type $\g_-\oplus \g_0$ one can assign
the canonical structure of absolute parallelism
on a bundle over $M$ of dimension $\dim \mathfrak U(\g_-\oplus\g_0)$;

\item
The algebra of infinitesimal symmetries of the flat structure of type $\g_-\oplus \g_0$ is isomorphic to $\mathfrak U(\g_-\oplus \g_0)$;

\item
The dimension of the algebra of infinitesimal symmetries of a Tanaka structure of type $\g_-\oplus \g_0$ is not greater than $\dim\,\mathfrak U(\g_-\oplus\g_0)$,
and any Tanaka structure of  type $\g_-\oplus \g_0$ whose algebra of infinitesimal symmetries has dimension $\dim\,\mathfrak U(\g_-\oplus\g_0)$ is locally equivalent to the flat structure of type $\g_-\oplus \g_0$.
\end{enumerate}
\end{thm}

\begin{rem}
\label{Chainrem}
A few words about item (1) of the previous theorem:  If $\g_-$ has $\mu$ nonzero graded components
and the positively graded part of $\mathfrak U(\g_-\oplus\g_0)$ consists of $l$ nonzero graded components, then  Tanaka recursively constructs a sequence of bundles $\{P^i\}_{1\leq i\leq l+\mu}$,
 \begin{equation}
  \label{bundles}
  M\leftarrow P^0\leftarrow P^1\leftarrow P^2\leftarrow\cdots,
 \end{equation}
where for $i>0$, $P^i$ is a bundle over $P^{i-1}$ whose fibers are affine spaces with modeling vector space $\g_i$ from \eqref{UTan0}. We refer to the recursive procedure of construction of $P^i\to P^{i-1}$ as the \emph{$i$th geometric prolongation}.   Observe that all $P^i$ with $i\geq l$ are identified with each other by the bundle projections, which are diffeomorphisms in those cases. The bundle $P^{l+\mu}$ is an $e$-structure over $P^{l+\mu-1}$; i.e., $P^{l+\mu-1}$ is endowed with a canonical frame -- a structure of absolute parallelism. It is important to note that for any $0<i\leq l+\mu-1$ the recursive construction of the bundle $P^{i+1}$ over $P^i$ depends on a choice of normalization conditions. Algebraically, ``normalization condition" refers to a choice of vector space complement to the image of a certain Lie algebra cohomology differential. The word ``canonical" in item (1) of Theorem \ref{Tanthm} means that one can assign to any Tanaka structure of type $\g_-\oplus\g_0$ a unique absolute parallelism (also called a frame or $e$-structure) on the bundle $P^{l+\mu-1}$ by applying the same fixed normalization conditions at each step of the construction of \eqref{bundles}, so that two Tanaka structures of type $\g_-\oplus\g_0$ are equivalent up to a diffeomorphism of $M$ if and only if the absolute parallelisms assigned to them are equivalent up to a diffeomorphism of the corresponding bundles. We emphasize that this assignment depends on a choice of normalization conditions at each step of the geometric prolongation procedure, but once this choice is fixed, the equivalence problem for Tanaka structures  is reduced to the equivalence problem for $e$-structures. This assignment is in fact a functor from the category of Tanaka structures (whose morphisms are diffeomorphisms of the underlying manifold $M$) to the category of $e$-structures (whose morphisms are $e$-structure-preserving diffeomorphisms of the bundle  $P^{l+\mu}$).
\end{rem}

The main advantage of this approach compared to Cartan's original method of equivalence is that, independently of the choice of normalization conditions, the basic features of the prolongation procedure such as the dimension of the resulting bundles, at which step the canonical frame is achieved, and the algebra of infinitesimal symmetries of the maximally symmetric ``flat" model can be described purely algebraically in terms of ${\mathfrak U} (\g_{-}\oplus \g_0)$. Questions concerning the most natural normalization conditions -- or whether normalization conditions exist such that the resulting absolute parallelism is a $\mathfrak U(\g_-\oplus\g_0)$-valued Cartan connection -- are more subtle and are best understood in the framework of the parabolic geometries; i.e., when the nonnegatively graded part of $\mathfrak U(\g_-\oplus\g_0)$ is a parabolic subalgebra (\cite{tanaka2, capslovak}).

\subsection{Bigraded Tanaka prolongation and description of main results}
It is clear that the standard Tanaka approach will not work for Levi degenerate CR structures, because describing the $k$-nondegeneracy of a CR structure on the graded level requires the assignment of nonnegative degree to vectors in the
Levi kernel, while in the standard Tanaka theory the nonnegatively graded components of the universal prolongation algebra correspond to vertical vectors fields on the appropriate bundle. Thus, the analog of a Tanaka symbol for such structure is not immediately obvious.

An attempt to define an analog of the Tanaka symbol for $k$-nondegenerate CR-structure, called an \emph{abstract core}, was made by  A. Santi in \cite{santi}. This notion is very natural and was used there toward the description of homogeneous models, but neither a functorial notion analogous to the universal prolongation of a Tanaka symbol nor a relation of the abstract core to a construction of an absolute parallelism was given there.

In the present paper, in the case of $2$-nondegenerate, hypersurface-type CR structures, we first propose the analog of the Tanaka symbol (section \ref{symbolsec}). In contrast to the standard Tanaka theory, this symbol is not a Lie algebra in general. It is a graded and even bigraded complex vector space endowed with an anti-linear involution, and with bigrading-compatible Lie brackets defined for most pairs of bigraded components, except the pair corresponding to $K$ and $\overline{K}$. Here and in the sequel, ``bigraded vector spaces" mean $(\mathbb Z\times\mathbb Z)$-graded ones.

Then we restrict ourselves to the class of symbols which are Lie algebras, calling them regular symbols (see Remark \ref{natrem} below discussing the naturality of this restriction).
Regular symbols have the structure of bigraded Lie algebras, so that the bigrading-compatible Lie brackets are defined on the whole symbol. We also define the notion of the flat CR structure with given regular symbol.

Our symbols and Santi's abstract cores \cite{santi} in the considered situation are in fact equivalent: one notion is uniquely recoverable from the other one -- see Remark \ref{Santirem} below for more detail. However, the main novelty here is that for our regular symbols we define the analog of the Tanaka universal algebraic prolongation in a functorial way. This analog is the maximal, nondegenerate, complex, \emph{bigraded} Lie algebra, such that it contains the symbol as its part with nonnegative first weight and, in addition, the only possible non-vanishing bigraded components with first weight equal to $1$ have biweights $(1,-1)$  and $(1,1)$ (Definition \ref{universal}). This bigraded algebra is endowed  with an anti-linear involution (i.e., with a real form).
 %and it is always finite-dimensional (section \ref{algebraic_prolongation}).
 The naturality of this notion is justified by our main theorem -- Theorem \ref{maintheor} -- on the existence of a canonical absolute parallelism for all CR structures with given regular $2$-nondegenerate CR symbol, which shows that \emph{the real part of the bigraded universal algebraic prolongation plays exactly the same role for our structure as Tanaka's universal prolongation for standard Tanaka structures}. In other words, for every CR structure with a given regular, $2$-nondegenerate CR symbol $\g^0$
%having a finite-dimensional bigraded universal algebraic prolongation,
a canonical absolute parallelism exists on a bundle of dimension equal to the (complex) dimension of the bigraded universal algebraic prolongation. Moreover, among such structures the flat structure is the only one -- up to local equivalence -- whose algebra of infinitesimal symmetries has the latter dimension, and the algebra of infinitesimal symmetries of the flat structure of type $\g^0$ is isomorphic to the real part of the bigraded universal Tanaka prolongation.

We emphasize that Theorem \ref{maintheor} establishes the existence of a canonical absolute parallelism following a choice of normalization condition in each step of the prolongation -- cf., Remark \ref{Chainrem} above -- but we do not specify which normalization conditions are preferable, and in particular we do not investigate if such conditions exist to determine a canonical Cartan connection. However, we believe that the framework developed here will be useful to address this question.

Further, we classify all regular $2$-nondegenerate symbols with $1$-dimensional Levi kernel in section \ref{classificationsec} (Theorems \ref{regclass} and \ref{regclass0}) and calculate their bigraded universal prolongations in section \ref{algebraic_prolongation} (Theorems \ref{regbuap} and \ref{regbuap0}).
It turns out in this case that if the linear operator $\ad_v$ is as in \eqref{adv}, then the CR symbol at a point $x\in M$ is regular if and only if an antilinear map $A: H_x/K_x \to H_x/K_x$ defined by
 \begin{equation}
 \label{antilinA}
 A (y)= \ad_v(\bar y)
 \end{equation}
satisfies
\begin{equation}
\label{alpha}
A^3=\alpha A, \quad \alpha\in \mathbb R.
\end{equation}
Note that, although $\alpha$ depends on the choice of a generator $v$ of $K_x$, its sign is independent of this choice. We subdivide the set of regular $2$-nondegenerate symbols with $1$-dimensional Levi kernel into
\emph{nilpotent regular} if $\alpha=0$ -- equivalently, $A^3=0$ -- and
\emph {non-nilpotent regular} otherwise; i.e., when $\alpha\neq 0$. The latter  type is subdivided into two subtypes:
\emph{strongly non-nilpotent regular} if $A$ is a bijection, and \emph{weakly non-nilpotent regular}
otherwise.

The bigraded universal algebraic prolongation of any strongly non-nilpotent regular symbol  is isomorphic to the complex orthogonal algebra $\mathfrak{so}\left(m,\mathbb C\right)$ with $m=\tfrac{1}{2}(\dim M+5)$. To describe the real parts of this prolongation  we say that the Hermitian form induced on $H/K$ by the Levi form is the \emph{reduced Levi form}. If  $\alpha$ in \eqref{alpha} is positive, there exists exactly one strongly non-nilpotent regular symbol whose signature of the reduced Levi form equal to $(p, q)$ -- $\dim M=2(p+q)+3$ -- and the real part of its bigraded universal prolongation is isomorphic to $\mathfrak{so}(p+2, q+2)$.
Such a symbol is said to be of type ${\mathrm I}_{p, q}$.
The case when $\alpha<0$ may occur if and only if $p=q$, and the real parts of their universal prolongations are  isomorphic to $\mathfrak{so}^*(2p+4)$. Such symbols are said to be of type $\mathrm{II}_p$.

To describe the unique (up to local diffeomorphism) maximally symmetric model among all $2$-non-degenerate, hypersurface-type CR structures with  strongly non-nilpotent regular CR symbol of type $\mathrm{I}_{p,q}$, we follow a procedure similar to that in \cite[Appendix, \S 6]{satake}. Consider a real-symmetric form $S$ of signature $(p+2, q+2)$ on $\mathbb R^{p+q+4}$. Complexifying, we extend $S$ to a form $S^\mathbb C$ on $\mathbb C^{p+q+4}$. To a complex line generated by $z=u+iv\in\mathbb C^{p+q+4}$  ($u,v\in \mathbb R^{p+q+4}$), one can associate a real plane $\text{span}\{u, v\}\subset\mathbb R^{p+q+4}$. This defines a map $F: \mathbb P^{p+q+3}(\mathbb C)\rightarrow \text{Gr}_2(\mathbb{R}^{p+q+4})$. The maximally symmetric model is given by the preimage under $F$ of $S$-isotropic planes in $\mathbb R^{p+q+4}$, which is a real hypersurface in the complex hypersurface of $S^\mathbb C$-isotropic lines in $\mathbb P^{p+q+3}(\mathbb C)$.  Using the same arguments as in \cite[section 2]{isaev} based on \cite{satake}, it can be shown that this model is locally equivalent to the tube hypersurface in $\mathbb C^{p+q+2}$ over the future light cone in $\mathbb R^{p+q+2}$, as in the model \eqref{tube}, but with signature $(p+1, q+1)$. 
%This can be shown using our main Theorem \ref{maintheor} and the same arguments as in \cite[section 2]{isaev} based on \cite{satake}.  

To describe the unique (up to local diffeomorphism) maximally symmetric model among all $2$-non-degenerate, hypersurface-type CR structures with  strongly non-nilpotent regular CR symbol of type $\mathrm{II}_{p}$ take $\mathbb C^{2p+4}$ endowed with split-signature hermitian form $h$ and a symmetric form $S$ such that in some basis, the forms $h$ and $S$ are given by matrices
$$\begin{pmatrix} \mathbbm{1}_{p+2} & 0\\0 &-\mathbbm{1}_{p+2}\end{pmatrix} \textrm{ and } \begin{pmatrix} 0 & \mathbbm{1}_{p+2}\\\mathbbm{1}_{p+2}&0\end{pmatrix},$$
respectively. The maximally symmetric  model is given by the real hypersurface of $h$-isotropic lines in the complex hypersurface of $S$-isotropic lines in the projective space $\mathbb P^{2p+3}(\mathbb C)$ of $\mathbb C^{2p+4}$.

The bigraded universal algebraic prolongation of weakly non-nilpotent regular symbol is equal to the symbol itself. Finally, the bigraded universal algebraic prolongations of nilpotent regular symbols are not semisimple. An explicit description of these algebras will be given in a future work \cite{DPZ}, but we are obliged to mention an important strain within them. For $\dim M\geq 7$, the maximal dimension among the bigraded universal algebraic prolongations of regular symbols with $1$-dimensional Levi kernel is achieved by the prolongation of the nilpotent regular symbol such that, in some basis (Theorem \ref{regclass0}), the antilinear map $A$ defined by \eqref{antilinA} has a matrix with only one nonzero Jordan block of size $2$; this dimension is $\tfrac{1}{4}(\dim M-1)^2+7$.

For completeness, we offer without proof the (local) hypersurface realizations of the maximally symmetric models for these nilpotent symbols. If $n=\frac{1}{2}(\dim M-1)$ and the signature of the reduced Levi form is $(p, q)$, $p+q=n-1$, then in coordinates $(z_1,\ldots, z_n, w)$ for $\mathbb C^{n+1}$ these are the hypersurfaces given by the equation
\begin{equation}
\label{maxmod}
\mathrm{Im}(w+z_1^2 \bar z_n)=z_1 \bar z_2 +\bar z_1 z_2 +\sum_{i=3}^{n-1} \varepsilon_i z_i\bar z_i,
\end{equation}
where $\varepsilon_i \in\{-1, 1\}$ and $\{\varepsilon _i\}_{i=3}^{n-1}$ consists of $p-1$ terms equal to $1$ and $q-1$ terms equal to $-1$ (note that for $\dim M=7$ last sum in the right-hand side of \eqref{maxmod} disappears -- see \eqref{maxmod7} below). The model \eqref{maxmod} was kindly communicated to us by Boris Doubrov, and will feature in future joint work \cite{DPZ} along with realizations of maximally symmetric models for other nilpotent symbols. These realizations can be obtained using the relation between CR-structures and systems of partial differential equations via Segre varieties; see, for example, \cite{sukhov, merker}.

For some small dimensions, we compare in the following table the dimension of the bigraded universal prolongations of such nilpotent regular symbols with the corresponding dimensions in the case of strongly non-nilpotent regular symbols:
\vskip .2in

\centerline{\begin{tabular}{|c| c| c|}
	\hline
	$\dim M$ & $\dim \mathfrak U_{\textrm{bigrad}}(\g^0)$ & \textrm{maximal} $\dim \mathfrak U_{\textrm{bigrad}}(\g^0)$ \\
	~&for strongly non-nilpotent symbols & for nilpotent symbols\\
	\hline
	7&15&16\\
	\hline
	9&21&23\\
	\hline
	11&28&32\\
	\hline
	%\medskip
	%\pause
\end{tabular}}
\vskip .2in

Moreover, we compare our results with the previously known results for $\dim M=5$ and $7$. For $\dim M=5$ there exists exactly one $2$-nondegenerate CR symbol. It automatically has a $1$-dimensional kernel, is strongly non-nilpotent regular, and it corresponds to $p=1$ and $q=0$, so that the real part of the bigraded universal algebraic prolongation is equal to $\mathfrak{so}(3,2)$ as expected from \cite{isaev,medori2, pocchiola}.

For $\dim M=7$  the bigraded universal algebraic prolongation of the strongly non-nilpotent regular symbol is isomorphic to $\mathfrak {so}(6, \mathbb C)\cong \mathfrak {sl}_4(\mathbb C)$. Here there are three different strongly non-nilpotent regular symbols: one in the case of reduced Levi form of signature $(2, 0)$, and two for signature $(1, 1)$.  In the first case  the real part of the bigraded universal algebraic prolongation is isomorphic to $\mathfrak {so}(4,2)\cong  \mathfrak {su}(2,2)$. In the second case it is isomorphic to either $\mathfrak {so}(3,3)\cong  \mathfrak {sl}_4(\mathbb R)$ or $\mathfrak {so}^*(6)\cong \mathfrak {su}(3,1)$. Note that in \cite{porter}, an absolute parallelism was constructed for two of these three cases, corresponding to $\mathfrak {so}(4,2)\cong  \mathfrak {su}(2,2)$ and $\mathfrak {so}^*(6)\cong \mathfrak {su}(3,1)$. Also in $\dim M=7$ there are two weakly non-nilpotent regular symbols with $1$-dimensional kernel -- one for each signature -- and exactly one regular nilpotent symbol. As noted in the table above, the bigraded universal algebraic prolongation of the latter has dimension $16$, compared to $15$ in the strongly non-nilpotent regular case, and the local hypersurface realization of the maximally symmetric model is given in coordinates $(z_1,z_2, z_3, w)$ of $\mathbb C^{4}$ by the equation
\begin{equation}
\label{maxmod7}
\mathrm{Im}(w+z_1^2 \bar z_3)=z_1 \bar z_2 +\bar z_1 z_2, 
\end{equation}
a particular case of \eqref{maxmod} with $n=3$. Finally, for $\dim M=7$ there are no $2$-nondegenerate CR structures with $2$-dimensional Levi kernel.

\begin{rem}
\label{natrem}
 Several words about the importance of CR structures with regular symbols: starting with $\dim M=7$, the space of non-equivalent symbols of $2$-nondegenerate, hypersurface-type CR structures with $1$-dimensional Levi kernel is large and contains continuous parameters (moduli). In fact, the classification of such symbols is equivalent to the classification of pairs consisting of a real line $\ell$ of nondegenerate Hermitian forms of arbitrary signature and a complex line $A$ of self-adjoint, anti-linear operators in a complex vector space of dimension $\frac{1}{2}(\dim M-3)$, up to a change of a basis in this space (see Proposition \ref{reduction_of_equivalence}).
 %Although classification of pairs of bilinear/ sesquilinear forms and linear/anti-linear operators satisfying certain symmetry properties were extensively studied in the literature in the frame of the classical Kronecker theory of pencils (\cite{gantmacher, thompson}),
This classification was recently carried out in \cite {SZ20} in terms of a special canonical form containing continuous parameters. Hence, the regular CR symbols classified in the present paper represent a very small subset of all symbols. In fact, they form only a finite subset in the space of all CR symbols with given $\dim M\geq 7$, which itself depends on continuous parameters. However, \cite{SZ1} shows that for \emph{generic} nonregular symbols of $2$-nondegenerate, hypersurface-type, pseudo-convex CR structures with $1$-dimensional Levi kernel, there is no CR structure with this symbol admitting a transitive group of symmetries, and the same is true (\cite{Sykes}, in preparation) in dimensions 7 and 9 without assuming pseudoconvexity. Moreover (\cite {SZ2}), the hypersurfaces given by \eqref{maxmod} would locally -- up to equivalence -- describe the maximally symmetric homogeneous models among 2-nondegenerate, hypersurface-type CR structures of a given dimension $\geq 7$ whose Levi kernel is $1$-dimensional (without assuming regularity of the symbol).
\end{rem}

\noindent{\bf Acknowledgements:} We are grateful to Andrea Santi for pointing out several gaps in the earlier version of the article, to Boris Doubrov for valuable discussions, especially concerning the construction of hypersurface realizations of maximally symmetric models, and to David Sykes for carefully reading the paper and offering helpful comments. 

\section{CR symbol for $2$-nondegenerate CR structures of hypersurface type}
\label{symbolsec}
As mentioned in the Introduction, we consider a $2$-nondegenerate, hypersurface-type CR structure $(D, J)$ on $M$.
For any real vector bundle $E\to M$, $\mathbb{C}E$ denotes the complexified vector bundle whose fiber over $x\in M$ is $\mathbb{C}E_x=E_x\otimes_\mathbb{R}\mathbb{C}=E_x\oplus_{\mathbb{R}}\im E_x$. Recall that $\mathbb{C}D_x=H_x\oplus\overline{H}_x$ splits into $\pm\im$-eigenspaces for $J_x$ and that the Levi kernel is a complex
%, rank-1
subbundle $K\subset H$ with conjugate $\overline{K}\subset \overline H$. Define
\begin{align}\label{g-}
&\g_{-1}(x)=\mathbb{C}D_x/(K_x\oplus \overline{K}_x),
&\g_{-2}(x)=\mathbb{C}T_xM/ \mathbb{C}D_x.
\end{align}
Similar to the Levi form on the holomorphic subbundle $H\subset\mathbb{C}TM$, one can define a $\mathbb{C}T_xM/\mathbb{C}D_x$-valued alternating form $\omega$ on $\g_{-1}(x)$: for $y_1,y_2 \in \g_{-1}$, let $Y_i\in\Gamma(\mathbb C D)$ be such that $Y_i(x)=y_i\, \mathrm{mod}\, K\oplus \overline K$, ($i=1,2$), and set
$$\omega(y_1, y_2)=[Y_1, Y_2](x) \quad \mathrm{mod}\, \g_{-1}(x).$$

Henceforth in this section we suppress the subscript $x$ or argument $(x)$ when referring to pointwise structures like the fibers of bundles or components of the symbol algebra. This is no loss of specificity, as we will restrict our attention to CR structures of constant type. The complex vector space
\begin{equation}
\label{grading}
\mathfrak \g_{-}=\g_{-1}\oplus\g_{-2}
\end{equation}
is endowed with the natural structure of a graded Lie algebra with the only nontrivial brackets coming from the form $\omega$. This algebra is isomorphic to the Heisenberg algebra. Its algebra of derivations
$\mathfrak {der}(\g_-)$ is isomorphic to the complex conformal symplectic algebra $\mathfrak{csp}(\g_{-1},\omega)$. Let
\begin{align*}
&\g_{-1, 1}=H/K,
&\g_{-1, -1}=\overline H/\overline K,
&&\g_{-2,0}=\g_{-2}.
\end{align*}
We say that a bigrading is compatible with the grading of a vector space if the $i^{\text{th}}$ component of the grading is the direct some of all bigraded components whose first weight is $i$. Bearing in mind the involutivity of $H$ and $\overline{H}$, we see that $\g_{-}$ is a bigraded Lie algebra with respect to the bigraded splitting
\begin{equation}
\label{bigrading}
\g_-=\g_{-1, -1}\oplus\g_{-1, 1}\oplus \g_{-2, 0},
\end{equation}
which is compatible with the grading \eqref{grading}. Also recall that $\g_-$ is endowed with the anti-linear involution given by the complex conjugation in $\mathbb{C}TM$. By construction, it satisfies
\begin{equation}
\label{conj}
\overline {\g_{i, j}}=\g_{i,-j}.
\end{equation}

For $v\in K$, take $\ad_{v}:\overline{H}/\overline{K}\to H/K$ to be as in \eqref{adv}. Extending  $\ad_v$  trivially to $H/K$, we obtain a derivation $\ad_v\in\mathfrak {der}(\g_-)$. Here, as before, $\mathfrak {der}(\g_-)$ is the Lie algebra of derivations of $\g_-$ preserving the grading \eqref{grading}, but not necessarily the bigrading \eqref{bigrading}.
Referring to the collection of such $\ad_v$ operators as $\ad_K$, we identify $\ad_K$ with a complex subspace in $\mathfrak {der}(\g_-)$, which we denote by $\g_{0,2}$. Taking the complex conjugate, we define $\ad_{\overline K}$ and similarly identify it with a subspace in $\mathfrak{der}(\g_-)$, denoted $\g_{0,-2}$. Recall that the vector space  $\g_-\oplus \mathfrak{der}(\g_-)$ is naturally endowed  with the structure of Lie algebra, namely,  the \emph{semi-direct sum} of $\mathfrak g_-$ and $\mathfrak{der}(\g_-)$. 
	%and we extend the brackets of $\g_-$ to this sum. 
By construction,
\begin{equation}
\label{g0pm2}
\begin{aligned}
&[\g_{0,2}, \g_{-1,-1}]\subset \g_{-1,1}, &[\g_{0,2}, \g_{-1,1}]= 0, \\
&[\g_{0,-2}, \g_{-1,1}]\subset \g_{-1,-1}, &[\g_{0,-2}, \g_{-1,-1}]=0.
\end{aligned}
\end{equation}
Finally, let $\g_{0,0}$ be the maximal  subalgebra of $\mathfrak {der}(\g_-)$ such that
\begin{eqnarray}
~&\label{g001} [\g_{0,0}, \g_{-1, \pm 1}]\subset \g_{-1, \pm 1}\\
~&\label{g002}[\g_{0,0}, \g_{0, \pm 2}]\subset \g_{0, \pm 2}.
\end{eqnarray}
Equivalently, $\g_{0,0}\subset\mathfrak {der}(\g_-)$ is the maximal subalgebra of derivations preserving the bigrading \eqref{bigrading} and satisfying \eqref{g002}. Collecting all bracket relations between the bigraded components defined so far, we have
\begin{equation}
\label{bigradedcond}
[\g_{i_1, j_1},\g_{i_2, j_2}]\subset \g_{i_1+i_2, j_1+j_2}, \quad i_1, i_2\leq 0,\,\, \{(i_1, j_1), (i_2, j_2)\}\neq \{(0,2), (0, -2)\}.
\end{equation}

Complex conjugation of $y\in\g_-$ induces conjugation on $f\in\mathfrak{der}(\g_-)$ via
\begin{equation}
\label{conjext}
\overline f (y)=\overline{ f(\overline y)}.
\end{equation}
In this way, \eqref{conj} extends to $\g_{0, j}$, with $j=0,\pm 2$. Thus, the anti-linear involution given by
the operator of complex conjugation is defined on the space
\begin{equation}
\label{CRsymbol}
\g^0=\g_{-2,0}\oplus\g_{-1, -1}\oplus \g_{-1,1}\oplus\g_{0,-2}\oplus\g_{0,0}\oplus \g_{0,2}.
%\quad \dim \g_{0,\pm2}=1.
\end{equation}

\begin{definition}
\label{CRsymboldef}
The vector subspace $\g^0$  of  $\g_-\oplus \mathfrak{der}(\g_{-})$, having bigrading as in \eqref{CRsymbol} and endowed with the anti-linear involution $\bar{~}$
is called the \emph{symbol of the CR structure $(D, J)$ at $x$}, or more briefly, the \emph{CR symbol}.
\end{definition}
Collecting all properties that we used in the previous definition, we get the following, natural
\begin{definition}
\label{absCRsymboldef}
Let $\g_-=\g_{-1}\oplus\g_{-2}$ be the complex graded Heisenberg algebra.
A vector space $\g^0=\g_-\oplus \g_0$ with  $\g_0\subset
 \mathfrak{der}(\g_{-})$ that has a bigrading like \eqref{CRsymbol} compatible with the grading $\g_-\oplus \g_0$ is called an \emph{abstract symbol} of $2$-nondegenerate, hypersurface-type CR structure if it satisfies \eqref{bigradedcond}, $\g_{0,0}$ is the maximal subalgebra of $\mathfrak{der}(\g_{-})$ satisfying \eqref{g001} and \eqref{g002}, and it is endowed with the anti-linear involution $\bar{~}$ satisfying \eqref{conj}.
\end{definition}

\begin{definition}
\label{constant}
Let $\g^0$ be an abstract $2$-nondegenerate CR symbol.
We say that a $2$-nondegenerate CR structure of hypersurface type has \emph{constant CR symbol $\g^0$} if its CR symbols at all points are isomorphic to $\g^0$.
\end{definition}

In contrast to the standard Tanaka symbol, an  abstract $2$-nondegenerate symbol is not a Lie algebra in general, because the brackets $[\g_{0, -2}, \g_{0,2}]$ may not belong to it. Indeed, for $\dim \g_-\geq 5$ (corresponding to $\dim M\geq 7$) the regular, abstract $2$-nondegenerate CR symbols  comprise a finite subset in the space of all $2$-nondegenerate CR symbols with given $\dim \g_-$, which itself depends on continuous parameters; see Remark \ref{natrem} above, section \ref{classificationsec} below for the classification of regular, abstract $2$-nondegenerate CR symbols, and \cite{SZ1} for the classification of all abstract  $2$-nondegenerate CR symbols.

\begin{definition}
\label{regsymb}
An abstract $2$-nondegenerate CR symbol $\g^0$ is called \emph{regular} if it is a Lie subalgebra of $\g_{-}\oplus\mathfrak{der}(\g_{-})$, which is equivalent to the condition
\begin{equation}
\label{regcond}
[\g_{0,-2}, \g_{0,2}]\subset \g_{0,0}.
\end{equation}
\end{definition}

By construction of $\g_{0,\pm 2}$ it follows that $ [[\g_{0,-2}, \g_{0,2}], \g_{-1, \pm 1}]\subset \g_{-1, \pm 1}$. Hence, by \eqref{g001} -\eqref{g002} the abstract $2$-nondegenerate CR symbol $\g^0$ is \emph{regular} if and only if
\begin{equation}
\label{regcond1} [[\g_{0,-2}, \g_{0,2}], \g_{0, \pm 2}]\subset \g_{0, \pm 2}.
\end{equation}
From this it  is easy to see that the \emph{symbol of the CR structure $(D, J)$ at $x$}  is regular if at only if  the following  relation:
\begin{equation}
\label{vvv}
\ad_{v_1}\circ \ad_{\bar v_2}\circ\ad_{v_3}+\ad_{v_3}\circ \ad_{\bar v_2}\circ\ad_{v_1} \in \mathrm {span}\{\ad_v|v\in K_x\}, \quad \forall v_1,v_2, v_3\in K_x
\end{equation}
Condition \eqref{regcond} endows $\g^0$ with the structure of  the \emph{bigraded Lie algebra},
because the rest of the bigrading conditions follow by construction.

\begin{rem}
\label{g00rem}
The component $\g_{0,0}$ in the definition of $2$-nondegenerate CR symbols can be recovered from the other components. However, we prefer to include it in the definition for brevity, and because it simplifies the notion of algebraic prolongation in the sequel.
\end{rem}

\begin{rem}
\label{Santirem}
In the situation considered here, the abstract core defined by A. Santi in \cite{santi} is in our notation the real part of the vector space
\begin{equation}
\label{m}
\mathfrak m=\g_{-2,0}\oplus\g_{-1, -1}\oplus \g_{-1,1}\oplus\g_{0,-2}\oplus \g_{0,2}
\end{equation}
(i.e. consisting of all components of $\g^0$ except $\g_{0,0}$),
together with:
\begin{itemize}
\item the operator $J$ defined on real parts of $\g_{-1}$ and $\g_{0, -2}\oplus\g_{0,2}$ and satisfying $J^2=-\mathrm{Id}$,
\item the structure of the Heisenberg algebra on $\g_-$, and
\item the identification of $\g_{0, -2}$ and $\g_{0,2}$ with elements of $\mathfrak{der}(\g_-)$ satisfying \eqref{g0pm2},
\end{itemize}
although Santi does not  explicitly use the structure of a bigrading. Clearly this data is equivalent to that encoded in our symbol. Note that the real part of the symbol does not have any natural bigrading.
\end{rem}

Finally, for a regular $2$-nondegenerate CR symbol $\g^0$ -- by analogy with the standard Tanaka theory -- one can define the notion of the flat structure of type $\g^0$. Let $\Re G^0$ and $\Re G_{0,0}$ be the simply connected Lie groups corresponding to the real parts of the Lie algebras $\g^0$ and $\g_{0,0}$. Let $M_0^\mathbb C= G^0/ G_{0,0}$ and $M_0= \Re G^0/\Re G_{0,0}$, both considered left cosets. For every pair $(i,j)$ used so far, let $\widehat D_{i,j}^0$  be the left-invariant distribution on $G^0$ that coincides with $\g_{i,j}$ at the identity. Since all $\g_{i,j}$ are invariant under the adjoint action of $G^{0,0}$, the push-forward of each $\widehat D_{i,j}^0$ to $M_0^\mathbb C$ is a well-defined distribution denoted by $D^0_{i,j}$. Let $D_{-1}^0$ be the distribution which is the sum of $D_{i,j}^0$ with $i=-1$. Restrict all of these distributions to $M_0$, considering them as subbundles of the complexified tangent bundle of $M_0$. The real part of $D_{-1}^0\oplus D_{0,-2}^0\oplus D_{0,2}^0$ defines the corank-1 
distribution $\Re D^0$ on $M_0$, as well as the operator $J$ on the fibers of $\Re D^0$ such that $D_{-1,1}^0\oplus D_{0,2}^0$ and 
$D_{-1, -1}^0\oplus D_{0, -2}^0$ are the $\im$ and $(-\im)$-eigenspaces of $J$, respectively. From the properties of the bigraded components of $\g^0$ it follows that this structure is a $2$-nondegenerate CR structure of hypersurface type with CR symbol $\g^0$.

\begin{definition}
\label{flatdef}
The CR structure on $\Re G^0/\Re G_{0,0}$ constructed above is called the \emph{flat CR structure of type $\g^0$}.
\end{definition}

\section{Universal algebraic prolongation of regular symbols and existence of absolute parallelism}
\label{universalsec}

In this section we fix a regular, abstract, $2$-nondegenerate CR symbol $\g^0$.
The bigraded structure of $\g^0$ is crucial for defining the correct universal algebraic prolongation so that it is suitable to the geometric prolongation procedure for the constructing the absolute parallelism.

\begin{definition}
\label{universal}
The \emph{bigraded universal algebraic prolongation of the symbol $\g^0$} is the maximal (nondegenerate) bigraded Lie algebra
$\mathfrak U_{\text{bigrad}}(\g^0)$ of the form
\begin{equation}
\label{bguniv}
\mathfrak U_{\text{bigrad}}(\g^0)=
%\g_{-2,0}\oplus\g_{-1, -1}\oplus \g_{-1,1}\oplus\g_{0,-2}\oplus\g_{0,0}\oplus \g_{0,2}
\g^0\oplus\g_{1,-1}\oplus\g_{1,1}\oplus\bigoplus_{i\geq 2, j\in\mathbb Z}\g_{i,j},
\end{equation}
where, as in the standard Tanaka theory, nondegeneracy means that $\ad(y)|_{\g_{-}}\neq 0$ for any nonzero $y\in\mathfrak{U} _{\text{bigraded}}$ with a nonnegative first weight. The space  $\g_i:=\bigoplus_{j\in\mathbb Z}\g_{i,j}$ is called the \emph{$i^{\text{th}}$ bigraded prolongation  of the symbol $\g^0$}.
\end{definition}

Let us compare this prolongation to the standard Tanaka universal algebraic prolongation ${\mathfrak U} (\g^0)$ of $\g^0=\g_{-}\oplus\g_{0}$ with respect to the first weight (where $\g_0=\bigoplus_{j\in\{0,\pm 2\}} \g_{0,j}$). In the process, we will demonstrate the uniqueness of the bigraded prolongation, which is not evident \textit{a priori}. Let

\begin{equation}
\label{UTan}
{\mathfrak U} (\g^0)=\g_{-}\oplus \g_0\oplus \bigoplus_{i>0}\widetilde{\g}_i
\end{equation}
We use $\widetilde{~}$ for the positively graded components to distinguish from the notation that will be used below for the bigraded Tanaka prolongation.

The spaces $\widetilde \g_i$ have an explicit realization, which we now describe for $\widetilde \g_1$. Recall that an endomorphism $f$ of a graded vector space has degree $k\in \mathbb Z$ if the image of a vector of weight $i$ is a vector of weight $i+k$; we write $\deg f=k$ in this case.
Similarly, we say that an endomorphism $f$ of a bigraded space has bidegree $(k, l)$ if the image of a vector of biweight $(i, j)$ has biweight $(i+k, j+l)$. Then
\begin{equation}
\label{firstalgTanaka}
\widetilde \g_1=\{f\in \mathrm{End}(\g^0)\ |\ \deg f=1,\ f ([v_1,v_2])=[f (v_1),v_2]+[v_1, f( v_2)]\,\,\forall\, v_1, v_2 \in \g_{-}\}.
\end{equation}
Observe that $f|_{\g_0}=0$ for $f\in \g_1$ as $\deg f=1$, so $f$ can be considered an element of $\mathrm{Hom} (\mathfrak g_-, \g^0)$ and one can define brackets between $f_1\in\widetilde \g_0$ and $f_2\in \mathfrak \g_1$ in the following natural way,
\begin{equation}
\label{posbr}
[f_1,f_2](v) = [f_1(v), f_2]+[f_1,f_2(v)]\quad \forall v\in\mathfrak \g_{-}.
\end{equation}

The bigrading \eqref{CRsymbol} on $\g^0$ induces a bigrading on $\widetilde \g_k$ with components of biweight $(k,l)$ consisting of the endomorphisms of $\g^0$ of  bidegree $(k, l)$. Based on the biweights of $\g^0$, the only nonzero bigraded components of $\widetilde \g_1$ have biweights $(1, \pm 1)$, $(1, \pm 3)$; i.e.,
\begin{equation}
\label{Tang1bigrade}
\widetilde \g_1=\widetilde \g_{1, -3}\oplus \widetilde \g_{1, -1}\oplus\widetilde \g_{1, 1}\oplus\widetilde \g_{1, 3}.
\end{equation}
From the maximality condition for the Tanaka universal  prolongation it follows that $\g_{1, -1}$\ and $\g_{1, -1}$ can be realized as subspaces if $\widetilde \g_1$. The bigrading condition on the prolongation \eqref{bguniv} requires $[\g_{1, -1}, \g_{0, -2}]\subset \g_{1, -3}=0$, as well as $[\g_{1, 1}, \g_{0, 2}]=0$. This implies that $\g_{1,\pm 1}$ must satisfy
\begin{eqnarray}
\label{firstalgbigrad1}
 ~& \g_{1,-1}=\{f\in \widetilde \g_{1, -1} \ |\  [f, \g_{0, -2}]=0, \bigl[[f, \g_{0,2}], \g_{0,2}\bigr]=0\}, \\
 ~&
 \label{firstalgbigrad2}
 \g_{1,1}=\{f\in \widetilde \g_{1, 1}\ |\  [f, \g_{0, 2}]=0, \bigl[[f, \g_{0,-2}], \g_{0,-2}\bigr]=0\}.
\end{eqnarray}
Since these form systems of linear equations, the first bigraded prolongation $\g_1$ is uniquely determined.

The first bigraded prolongation $\g_1=\g_{1,-1}\oplus\g_{1,1}$ can be significantly smaller than  $\widetilde \g_1$. For example, in the case $\dim M=5$, $\g_-$ is a free truncated Lie algebra and $\g_0=\mathfrak{der}(\g_-)\cong\mathfrak {gl}(\g_{-1})$, so $\widetilde \g_1$ can be identified with $\mathrm{Hom}(\g_{-1},\g_{0})$. Because $\dim \g_{-1}=2$, $\dim  \widetilde \g_1=8$. On the other hand, from section \ref{TypeI} it will follow that $\dim \g_1=2$ in this case.

In contrast to the first bigraded prolongation $\g_1$, the $i^{\text{th}}$ bigraded prolongation $\g_i$ of $\mathfrak U_{\text{bigrad}}(\g^0)$ with $i\geq 2$ is exactly the same as that of the standard Tanaka universal prolongation of the space $\g_-\oplus \g_0 \oplus \g_1$ (by the latter we mean the maximal nondegenerate graded Lie algebra such that its part corresponding to weights not greater than $1$ is equal to  $\g_-\oplus \g_0 \oplus \g_1$). Recursively, one has
\begin{equation}
\label{arbalgTanaka}
\g_i=\left\{\left.f\in \mathrm{End}\left(\mathfrak \bigoplus _{k\leq i-1}\g_k\right)\  \right|\ \deg f=i, f ([v_1,v_2])=[f (v_1),v_2]+[v_1, f( v_2)]\,\,\forall\, v_1, v_2 \in \g_{-}\right\},
\end{equation}
and the second weight is assigned in the obvious way. The fact that we prolong $\g_-\oplus \g_0 \oplus \g_1$ and not $\g_-\oplus \g_0 \oplus \widetilde \g_1$ implies that $\g_i$ is in general a subspace of $\widetilde \g_i$. Also, \eqref{arbalgTanaka} is a system of linear equations, so the $i^{\text{th}}$ bigraded Tanaka prolongation is uniquely defined.

Since $\g_1$ is a subspace of endomorphisms acting on a space endowed with an anti-linear involution satisfying \eqref{conj}, an anti-linear involution satisfying \eqref{conj} is induced  on $\g_1$ via the formula \eqref{conjext}. This involution can be extended recursively to the whole algebra $\mathfrak U_{\text{bigrad}}(\g^0)$ using the formula \eqref{conjext}. Let $\Re \mathfrak U_{\text{bigrad}}(\g^0)$ be the real part of $\mathfrak U_{\text{bigrad}}(\g^0)$ with respect to this involution.

The following theorem is the main result of this paper:

\begin{thm}
\label{maintheor}
Assume that $\g^0$ is a regular symbol such that $\dim \mathfrak U_{\text{bigrad}}(\g^0)<\infty$.
\begin{enumerate}
\item
To any $2$-nondegenerate, hypersurface-type CR structure with symbol $\g^0$ one can assign
the canonical\footnote{That is, canonical in the sense of Remark \ref{Chainrem}.} structure of absolute parallelism on a bundle over $M$ of (real) dimension equal to $\dim_{\mathbb C } \mathfrak U_{\text{bigrad}}(\g^0)$. In particular,  two  $2$-nondegenerate, hypersurface-type CR structures with symbol $\g^0$  are equivalent up to a diffeomorphism of $M$ if and only if the structures of absolute parallelism assigned to them are equivalent up to a diffeomorphism of the corresponding bundles.

\item
The algebra of infinitesimal symmetries of the flat CR structure of type $\g^0$ is isomorphic to the real part $\Re \mathfrak U_{\text{bigrad}}(\g^0)$ of $\mathfrak U_{\text{bigrad}}(\g^0)$;

\item
The dimension of the algebra of infinitesimal symmetries of a $2$-nondegenerate, hypersurface type CR structure with regular symbol $\g^0$ is not greater than $\dim_{\mathbb C}\,\mathfrak U_{\text{bigrad}}(\g^0)$,
and any CR structure of type $\g^0$ whose algebra of infinitesimal symmetries has dimension $\dim\,\mathfrak U_{\text{bigrad}}(\g^0)$ is locally equivalent to the flat CR structure of type $\g^0$.
\end{enumerate}
\end{thm}

The proof of this theorem constitutes section \ref{geometric_prolongation}. We emphasize again that the construction of the canonical absolute parallelism in item (1) of Theroem \ref{maintheor} depends on a choice of normalization conditions at each step of the geometric prolongation procedure: roughly speaking, a choice of vector space complement to the image of an appropriate Lie algebra cohomology differential. The exact definition of a normalization condition for the first geometric prolongation is given in Definition \ref{N1_P1_def} (formula \eqref{firsdirect}) and for the higher geometric prolongations in formula \eqref{kdirect} below.

We classify all regular, abstract $2$-nondegenerate CR symbols with $1$-dimensional Levi kernel in section \ref{classificationsec} (Theorems \ref{regclass} and \ref{regclass0}) and calculate the real parts of their bigraded universal algebraic prolongations in section \ref{algebraic_prolongation} (Theorems \ref{regbuap} and \ref{regbuap0}), so that Theorem \ref{maintheor} can be formulated more explicitly in this case. In particular, in this case all bigraded universal algebraic prolongations are finite dimensional.

%%%%%%%%%%%%%%%%%%%%%%%%%%%%%%%%%%%%%%%%%%%%%%%%%%%%%%%%%%%%%%%%%%%%%%%%%%%%%%%%%%%%%%%%%%%%%%%%%%%%%%%%%%%%%%%%%%%%%%%
\section{Classification of Regular $2$-nondegenerate CR Symbols with $\dim_{\mathbb{C}}\g_{0,\pm2}=1$.}
%%%%%%%%%%%%%%%%%%%%%%%%%%%%%%%%%%%%%%%%%%%%%%%%%%%%%%%%%%%%%%%%%%%%%%%%%%%%%%%%%%%%%%%%%%%%%%%%%%%%%%%%%%%%%%%%%%%%%%%
\label{classificationsec}

 In this section we restrict our attention to the case of $1$-dimensional Levi kernel. Let $\g^0$ be a regular, abstract, $2$-nondegenerate CR symbol with $1$-dimensional Levi kernel. Then one can assign to it a pair: a \emph{canonical real line of nondegenerate  Hermitian forms and a canonical complex line of anti-linear operators on the space $\g_{-1,1}$}; see Remark \ref{natrem} above.

A  generator $\ell$ of the canonical real line of Hermitian forms on $\g_{-1,1}$ is defined as follows: First, fix a basis vector $e_0\in\im\Re \g_{-2,0}$. Then $\ell$ is defined via
\begin{equation}
\label{hermit}
[y_1,\overline{y}_2]=\ell(y_1,y_2) e_0, \quad y_1,y_2\in \g_{-1,1}.
\end{equation}
Re-scaling $e_0\in \im\Re\g_{-2,0}$ effects re-scaling of $\ell$.

A generator $A$ of the canonical complex line of anti-linear operators is defined by choosing a basis vector $f\in\g_{0,2}$ and setting
\begin{equation}
\label{anti-linear}
A y=[f,  \overline y], \quad y\in \g_{-1,1}.
\end{equation}
Re-scaling $f\in \g_{0,2}$ effects re-scaling of $A$. Moreover, it is easy to show (\cite[\S 2.3]{porter}) that $A$ is \emph{self-adjoint} with respect to $\ell$, i.e.
\begin{align}\label{A_symm}
\ell(Ay_1,y_2) =\ell( Ay_2,y_1), \quad y_1,y_2\in \g_{-1,1}.
\end{align}

Equivalence relations are naturally defined on the space $\mathfrak C$ of abstract, $2$-nondegenerate CR symbols with $1$-dimensional Levi kernel and the space $\mathfrak P$ of pairs consisting of a real line of nondegenerate Hermitian forms and a complex line of anti-linear, self-adjoint operators on $\g_{-1,1}$. Two CR symbols are equivalent if there exists a linear isomorphism between them preserving all their algebraic structures. Regarding $\mathfrak P$, the standard action of the group $\mathrm{GL}(\g_{-1,1})$ induces an action -- and therefore an equivalence relation -- on it as well. Using Remark \ref{g00rem} and the fact that all nontrivial Lie brackets on the space $\mathfrak m$ defined by \eqref{m} are given via the form $\ell$ and the operator $A$, the following is immediate.

\begin{prop}
\label{reduction_of_equivalence}
Two abstract $2$-nondegenerate CR symbols $\g^0$ and $\widetilde \g^0$ with $1$-dimensional Levi kernels are equivalent if and only if the corresponding pairs $(\ell, A)$ and $(\widetilde \ell, \widetilde A)$ from $\mathfrak P$ are equivalent.
\end{prop}

%Normal forms for the pair $(\langle\ell\rangle_{\mathbb R}, \langle A\rangle_{\mathbb C})$ in $\mathfrak P$ are known in the case of sign-definite $\ell$ only (\cite{vuji}). 
Without the assumption of sign-definiteness of $\ell$, normal forms are still known for pairs $(\langle\ell\rangle_{\mathbb R}, \langle A^2\rangle_{\mathbb R})$
(\cite[Theorem 5.1.1]{gohberg}). Let us call the latter pair the \emph{reduced  pair} of the original one. Note that the equivalence of pairs in $\mathfrak P$ obviously implies the equivalence of the corresponding reduced pairs, but not vice versa. 
Normal forms from the recent paper 
 \cite {SZ20} show that
the space of equivalence classes of $2$-nondegenerate CR symbols is very rich and contains moduli. 
In this paper we are interested in regular symbols only, and this space is small and discrete (see Remark \ref{natrem} above discussing why regular symbols are particularly important). The following lemma is important toward the classification of all regular  $2$-nondegenerate CR symbols:

\begin{lem}
\label{regclasslem}
Assume that an abstract $2$-nondegenerate CR symbol $\g^0$  with $1$-dimensional Levi kernel is  described by a pair  $(\langle\ell\rangle_{\mathbb R}, \langle A\rangle_{\mathbb C})\in\mathfrak P$  and
\begin{equation}
\label{kerim}
Z=\ker\, A, \quad \quad W= \mathrm{Im}\,A
\end{equation}

If  the CR symbol $\g^0$ is regular, then
\begin{equation}
\label{kerim0}
%A|_Z=0, \quad
 A^3=\alpha A.
\end{equation}
for $\alpha\in\mathbb{R}$.
Moreover, if $\alpha\neq 0$, then
%in this case
\begin{eqnarray}
&~& \g_{-1, 1}=Z\oplus W, \label{split1-1}\\
&~& A^2|_W=\alpha I|_W, \label{kerim00}
\end{eqnarray}
and the restrictions of the Hermitian form $\ell$ to $W$ and $Z$ are nondegenerate.
\end{lem}

\begin{proof} Relation \eqref{kerim0} is in fact a direct consequence of \eqref{vvv} and \eqref{anti-linear} in the case of one-dimensional Levi kernel. We give here a detailed proof for completeness, as the proof of \eqref{vvv} was omitted. 
Fix $f\in\g_{0, 2}$ so that $\overline  f\in \g_{0,-2}$. Both $f$, $\bar f$ belong to $\mathfrak{der}(\g_-)\subset\mathfrak{gl}(\g_{-1})$, and their Lie brackets are defined by the commutator in $\mathfrak{gl}(\g_{-1})$. Hence, the regularity condition \eqref{regcond} reads $f\circ \overline f-\overline f\circ f\in\g_{0,0}$, which by the definition of $\g_{0,0}$ implies
\begin{align}
\label{comm3}
(f\circ \overline f - \overline f\circ f)\circ f-f\circ(f\circ \overline f-\overline f\circ f)=\alpha_0 f,
\end{align}
for some constant $\alpha_0$. Recall that the construction of $\g_{0,\pm 2}$ requires $f^2=0$. Therefore, the left-hand side of \eqref{comm3} reduces to $2f\circ \overline f\circ f$, whence
 \begin{align}
\label{comm31}
f\circ \overline f\circ f=\alpha f
\end{align}
for $\alpha$ ($=\tfrac{\alpha_0}{2}$).
%, which must be $\mathbb{R}$-valued by \eqref{A_symm}.
Restricting the last identity to $\g_{-1,-1}$ and using \eqref{anti-linear} we get that
\begin{align}
\label{comm32}
A^3=\alpha A.
\end{align}

Further , let $\alpha\neq 0$. and subspaces $Z$ and $W$ are as in \eqref{kerim}. Let us show that $Z\cap W=0$. Indeed, if $y\in Z\cap W$, then $Ay=0$  and there exists $w\in \g_{-1,1}$ such that $y=A w$. Hence, $0=A^2y=A^3 w=\alpha Aw=\alpha y$,
which together with the assumption that $\alpha\neq 0$ implies that $y=0$. Thus, $Z\cap W=0$ and we have \eqref{split1-1}.
Further  $A|_{W}$ is a bijection, so restricting \eqref{comm32} to $W$ we get \eqref{kerim00}. Finally, from \eqref{kerim00} and the fact that $A^2$ is a self-adjoint operator it follows that $\alpha$ is real.  
\end{proof}

\begin{definition}
\label{strongreg}
Assume that a regular abstract $2$-nondegenerate CR symbol $\g^0$  with $1$-dimensional Levi kernel is described by a pair  $(\langle\ell\rangle_{\mathbb R}, \langle A\rangle_{\mathbb C})\in\mathfrak P$. Then $\g^0$ is called \emph{nilpotent regular} if the coefficient $\alpha$ in \eqref{kerim0} is equal to zero or, equivalently, $A^3=0$ and
\emph {non-nilpotent regular} otherwise, i.e. when $\alpha\neq 0$. A non-nilpotent regular symbol is called
\emph{strongly non-nilpotent regular}, if $A$ is a bijection, and \emph{weakly non-nilpotent regular}
otherwise.
\end{definition}

The classification of all regular CR symbols is given by the following two theorems, describing the cases of non-nilpotent and nilpotent regular symbols separately:

\begin{thm}
\label{regclass}
Consider a non-nilpotent regular, abstract $2$-nondegenerate CR symbol $\g^0$ described by a pair $(\langle\ell\rangle_{\mathbb R}, \langle A\rangle_{\mathbb C})$. Assume
that $\ell$ has signature $(p,q)$ with $p\geq q$, and $(p_1, q_1)$ is the signature of the restriction of $\ell$ to $W= \mathrm{Im}\,A$ ($p_1\leq p, q_1\leq q$). Re-scaling if necessary, we can assume that $A^2|_W=\mathbbm 1_W$ or $A^2|_W=-\mathbbm 1_W$, depending on the sign of $\alpha$ in \eqref{kerim0}.  Then the following two statements hold:
\begin{enumerate}
\item If $A^2|_W=\mathbbm 1_W$, then there are bases of $W$ and $Z$ ($=W^\perp$) such that $\ell|_W$, $\ell|_Z$, $A|_W$, and $A|_Z$ are represented by the matrices
  \begin{eqnarray}
    \label{matrixnormalform1}
    ~&\ell|_W=\begin{pmatrix} \mathbbm 1_{p_1} &0\\0& -\mathbbm 1_{q_1}\end{pmatrix},& \ell|_Z=\begin{pmatrix} \mathbbm 1_{p-p_1} &0\\0&-\mathbbm 1_{q-q_1}\end{pmatrix},\\
    ~& A|_W=\mathbbm 1_W,& A|_Z=0.\nonumber
  \end{eqnarray}
  In particular, a regular symbol with $A^2|_W=\mathbbm 1|_W$ is uniquely -- up to isomorphism -- determined  by the signature  $(p,q)$ of $\ell$, and the signature $(p_1, q_1)$ of the restriction of $\ell$ to $W=\mathrm{Im} A$.

\item  If $A^2|_W=-\mathbbm 1|_W$, then $q_1=p_1$  and
there are bases of $W$ and $Z$ ($=W^\perp$) such that $\ell|_W$, $\ell|_Z$, $A|_W$, and $A|_Z$ are represented by the matrices
 \begin{eqnarray}
    \label{matrixnormalform2}
    ~&\ell|_W=\begin{pmatrix} \mathbbm 1_{p_1} &0\\0& -\mathbbm 1_{p_1}\end{pmatrix},& \ell|_Z=\begin{pmatrix} \mathbbm 1_{p-p_1} &0\\0&-\mathbbm 1_{q-p_1}\end{pmatrix},\\
    ~& A|_W=\begin{pmatrix}0&-\mathbbm 1_{p_1}\\ \mathbbm 1_{p_1}&0\end{pmatrix},& A|_Z=0.\nonumber
  \end{eqnarray}
  Thus, a regular symbol with $A^2|_W=-\mathbbm 1|_W$ exists only if $\dim W$ is even, the signature of the restriction of $\ell$ to $W$ is $(p_1, p_1)$, and in this case it is uniquely -- up to isomorphism -- determined by the signature $(p,q)$ of $\ell$, $p\geq q\geq p_1$.
\end{enumerate}
\end{thm}

To formulate the classification theorem in the case of nilpotent regular symbol, denote by $J_k$ the \emph{nilpotent Jordan block of size $k\times k$} and by $P_k$ the $k\times k$-matrix with $1$'s on the anti-diagonal, i.e., the diagonal going from the lower left corner to the upper right corner of the matrix, and $0$'s elsewhere. In terms of the entries the  matrices $J_k$ and $P_k$ are defined by
\begin{equation*}
(J_k)_{ij}=\begin{cases} 1, & j=i+1\\ 0,&\text{otherwise}\end{cases},\quad (P_k)_{ij}=\begin{cases} 1, & i+j=k+1\\ 0,&\text{otherwise}
\end{cases}, \quad 1\leq i,j\leq k.
\end{equation*}
Also, given square matrices $A_1,\ldots A_s$, we shall denote by $A_1\oplus\cdots\oplus A_s$ the corresponding \emph{block-diagonal} matrix.

%$$J_k= \begin{pmatrix}
%    0     & 1     &         &          &   \\
%          & 0     & 1       &          &   \\
%          &       & \ddots  & \ddots   &   \\
%          &        &        & \ddots   &  1 \\
%          &        &        &          &  0  \\
%  \end{pmatrix}$$

\begin{thm}
\label{regclass0}
Consider a nilpotent regular, abstract $2$-nondegenerate CR symbol $\g^0$ described by a pair $(\langle\ell\rangle_{\mathbb R}, \langle A\rangle_{\mathbb C})$. Then there exists a basis of $\g_{-1,1}$, integers $k_1, \dots k_s\in\{1, 2,3\}$, at least one of which is greater than $1$, and numbers $\epsilon_1, \dots, \epsilon_s\in\{-1,1\}$ such that in the chosen basis
\begin{eqnarray}
~&
A=J_{k_1}\oplus\cdots\oplus J_{k_s} \label{case0A},\\
~& \ell=\epsilon_1P_{k_1}\oplus\cdots\oplus \epsilon_s P_{k_s}. \label{case0l}
\end{eqnarray}
\end{thm}
Before proving these theorems, let us distinguish an important subclass of regular symbols and reveal the consequence of Theorem \ref{regclass} for this subclass.

Strongly non-nilpotent regularity is equivalent to the condition that  $\alpha\neq 0$ together with the condition $Z=0$ in the splitting \eqref{split1-1}. The latter condition is equivalent to the condition that  $p_1=p$ and $q_1=q$ in Theorem \ref{regclass}. Theorem \ref{regclass} reads as follows in this particular case.

\begin{cor}
\label{strongclass}
Suppose a strongly non-nilpotent regular, abstract, $2$-nondegenerate CR symbol $\g^0$ is described by a pair $(\langle\ell\rangle_{\mathbb R}, \langle A\rangle_{\mathbb C})$, such that $\ell$ has signature $(p,q)$ with $p\geq q$. Re-scaling if necessary, we may assume $A^2=\mathbbm 1_{\g_{-1,1}}$ or $A^2=-\mathbbm 1_{\g_{-1,1}}$, depending on the sign of $\alpha$ in \eqref{kerim0}. Then the following two statements hold:
\begin{enumerate}
\item If $A^2=\mathbbm 1_{\g_{-1,1}}$, there is a basis of $\g_{-1,1}$ such that the generator $\ell$ of the canonical line of Hermitian forms and a generator $A$ of the canonical line of anti-linear operators are represented by the matrices
    \begin{equation}
    \label{matrixnormalform1s}
    \ell=\begin{pmatrix} \mathbbm 1_p &0\\ 0&-\mathbbm 1 _q\end{pmatrix} \quad \mathrm {  and }\quad  A=\mathbbm 1_{p+q}
    \end{equation}
     in this basis. Thus, in this case strongly non-nilpotent regular symbols are uniquely -- up to isomorphism -- determined  by the signature $(p,q)$, $p\geq q$.

\item  If $A^2=-\mathbbm 1_{\g_{-1,1}}$ , then $q=p$ and there is a basis of $\g_{-1,1}$
such that a generator $\ell$ of the canonical line of Hermitian forms and a generator $ A$ of the canonical line of anti-linear operators are represented by the matrices
\begin{equation}
    \label{matrixnormalform2s}
    \ell=\begin{pmatrix} \mathbbm 1_p &0\\ 0&-\mathbbm 1_p\end{pmatrix}
    \quad \mathrm{ and } \quad A= \begin{pmatrix}0 &-\mathbbm 1_p \\ \mathbbm 1_p & 0 \end{pmatrix},
    \end{equation}
    in this basis. Thus, in this case strongly non-nilpotent regular symbols exist only if $\dim \g_{-1,1}$ is even and the signature of the Hermitian form is $(p,p)$. Such a symbol is uniquely -- up to an isomorphism -- determined by $p$ or $\dim\,\g_{-1,1}=2p$.
\end{enumerate}
\end{cor}

\begin{definition}
\label{typedef}
The $2$-nondegenerate CR symbol from item (1) of Corollary \ref{strongclass} is called the \emph{strongly non-nilpotent regular symbol of type $\mathrm{I}_{p,q}$}, or simply \emph{type} I if the specification of the signature is not essential. The $2$-nondegenerate CR symbol from item (2) of Corollary \ref{strongclass} is called the \emph{strongly non-nilpotent regular symbol of type $\mathrm{II}_{p}$}, or simply \emph{type }II if the specification of the signature is not essential.
\end{definition}

\noindent{\bf Proof of Theorem \ref{regclass}}

\noindent First consider the case when $A^2|_W=\mathbbm{1}_{W}$. Following the proof of \cite[Lemma 3.1]{uhl}, let
 $$W_+=\{x+Ax, x\in W\}, \quad  W_-=\{x-Ax, x\in W\}.$$
 Then $W_\pm$ are real vector subspaces and $A w=\pm w$ for any $w\in W_{\pm}$. Also, from anti-linearity, $W_-=\im W_+$ and $W=W_+\oplus W_-$. By the standard theory for nondegenerate Hermitian forms, we can choose  an $\ell$-orthonormal basis in $W_+$ with respect to which $\ell|_W$ and $A|_W$ are represented by matrices as in \eqref{matrixnormalform1}. The two matrix representations for the restrictions of $\ell$ and $A$ to $Z$ follow from the facts that $Z=W^\perp$ and $Z=\ker A$ by choosing an $\ell$-orthonormal basis in $Z$.

Now consider the case $A^2=-\mathbbm{1}_{W}$. Define positive and negative cones in $W$ by $\mathcal C^+=\{w: \ell(w,w)\geq 0\}$ and $\mathcal C^-=\{w: \ell(w,w)\leq 0\}$. By \eqref{A_symm},
$$\ell(A w, A w)=\ell(A^2 w, w)=-\ell(w,w).$$
Hence, the anti-linear map $A$ sends $\mathcal C^+$ to $\mathcal C^-$ and vice versa, which is possible only if the signature of $\ell|_W$ is $(p_1,p_1)$. Indeed, the positive (negative) index of $\ell|_W$ is the maximal dimension of subspaces belonging to $\mathcal C^+$ ($\mathcal C^-)$. Since $A$ sends subspaces to subspaces and is nonsingular on $W$ (from the fact that $A^2=-\mathbbm{1}_{W}$), the positive and negative indices of $\ell|_W$ must be equal.

Let $e\in W$ with $\ell(e,e)=1$ so that $\ell(Ae,Ae)=-1$ and $\text{span}_\mathbb{C}\{e,Ae\}$ is an $A$-stable subspace of $W$. We can take $e,Ae$ to be orthogonal with respect to $\ell$ as follows. If $\ell(e,Ae)=r\mathbf{e}^{\im\phi}$ for real $\phi,r>0$ (and $\mathbf{e}$ the natural exponential), replace $e$ with $\mathbf{e}^{-\im\frac{\phi}{2}}e$ so that $\ell(e,Ae)=r$.

 Let $\E=\text{span}_\mathbb{R}\{e,Ae\}$. Now set
 $\tilde e =a_{1} e+a_{2} Ae$ for $a_1,a_2\in\mathbb{R}$ and let $z=a_{1}+\im  a_{2}$. Invoking our hypothesis on $A$,  $A\tilde e=-a_{2}e+a_{1} Ae$, and it is easy to see that
 \begin{equation}
 \label{zeq}
 \tilde e-\im A \tilde e= z(e-\im Ae).
 \end{equation}
The restriction of $\ell$ to $\E$ defines a symmetric form that can be extended $\mathbb{C}$-linearly to the symmetric form $(\cdot, \cdot)$ on $\mathbb C\E$. By construction, $(e-\im Ae,e-\im Ae)=2-2r\im$ and orthonormality of $\{\tilde e, A\tilde e\}$ is equivalent to $(\tilde e-\im A\tilde e,\tilde e-\im A\tilde e)=2$. From this and \eqref{zeq} it follows that $z$ must satisfy the equation $z^2=\tfrac{1}{1-\im r}$, which has a solution.

 This proves that an orthonormal basis of $\mathcal{E}$ of the form $\{e,Ae\}$ exists. Self-adjointness \eqref{A_symm} ensures that the orthogonal complement of the nondegenerate, $A$-stable subspace $\mathcal E$ is also $A$-stable. Iterating the above procedure, we determine that there is an orthonormal basis $\{e_{\alpha_1}\}_{\alpha_1=1}^{2p_1}$ of $W$ whose first $p_1$ vectors belong to $\mathcal C^+$ and last $p_1$ vectors belong to $\mathcal C^-$, such that $Ae_a=e_{p+a}$ for $1\leq a\leq p_1$. Therefore, the matrix representations of $\ell|_W$ and $A|_W$ in this basis are exactly \eqref{matrixnormalform2}. The proof is completed in the same manner as in the previous case by noticing that $Z=W^\perp$, $Z=\ker A$, and we can choose an $\ell$-orthonormal basis in $Z$ to get the rest of \eqref{matrixnormalform2} concerning the restrictions to $Z$.
 \hfill $\Box$
\medskip

\noindent{\bf Proof of Theorem \ref{regclass0}}

The proof mimics part of the  proof of Theorem 5.1.1 from \cite{gohberg}; the main difference here is that we consider anti-linear instead of linear maps. By \eqref{comm32} $A^3=0$ in the considered case. Consider the following two subcases:

{\bf Case 1} $A^2\neq 0$. Choose some positive definite Hermitian form $(\cdot, \cdot)$ on $\g_{-1,1}$ and let $H$ be the linear operator relating the forms $\ell$ and $(\cdot, \cdot)$, i.e. $\ell(v, w)=(Hv, w), \,\, v, w\in \g_{-1,1}$. By construction, $HA^2$ is a self-adjoint linear operator with respect to $(\cdot, \cdot)$ and hence it has a non-zero (real) eigenvalue. Let $a_1$ be the corresponding eigenvector. Then $\ell(A^2 a_1,a_1)=(HA^2 a_1, a_1)\neq 0$ is real. Normalize $a_1$ such that
\begin{equation}
\label{eps1}
\ell(A^2 a_1, a_1)=\epsilon_1,
\end{equation}
where $\epsilon_1\in\{-1, 1\}$.

Let $a_j:=A^{j-1} a_1$, $j=2,3$. Directly from \eqref{eps1}  and the fact that $A$ is self-adjoint with respect to $\ell$ it follows that
\begin{equation*}
\ell(a_j, a_k)=\epsilon_1, \quad j+k=4.
\end{equation*}
Moreover, from the identity $A^3 a_1=0$ we get
\begin{equation*}
\ell(a_j, a_k)=0, \quad j+k>4.
\end{equation*}
Now set
\begin{equation}
\label{b}
b_1:=a_1+\alpha_2 a_2 +\alpha_3 a_3, \quad b_j:=A^{j-1} b_1, \quad j=2,3.
\end{equation}
We can choose $\alpha_2$, and $\alpha_3$ such that $\ell(b_1, b_1)=\ell (b_1, b_2)=0$. Indeed, from \eqref{b} it follows that the condition
$\ell(b_1, b_2)=0$ is equivalent to
$$\ell(a_1, a_2)+2\epsilon \alpha_2=0,$$
which fixes $\alpha_2$, and the condition $\ell(b_1, b_1)=0$ is equivalent to
$$\ell(a_1, a_1)+2\epsilon_1 \mathrm{Re}\,\alpha_3+(\text{terms depending on } \alpha_2)=0,$$
which can be achieved by an appropriate choice of $\alpha_3$, taking into account $\alpha_2$ is chosen in the previous step.
If $X_1=\mathrm {span} \{b_1, b_2, b_3\}$, then in the basis $(b_3, b_2, b_1)$ the map $A|_{X_1}$ is represented by the matrix  $J_3$ and
the form $\ell|_{X_1}$ is represented by the matrix $\epsilon_1 P_3$.

{\bf Case 2}. $A^2=0$. By assumption, $A\neq 0$, so
in this case there exists a vector $a_1$ such that $\ell(A a_1, a_1)\neq 0$, otherwise using the analog of the polarization identity we will get that $A=0$. Further, we can repeat the argument from the previous case to build two vectors $b_1$ and $b_2$ and the plane $X_1=\mathrm {span} \{b_1, b_2\}$ such that in the basis $(b_2, b_1)$ the map $A|_{X_1}$ is represented by the matrix  $J_2$ and
the form $\ell|_{X_1}$ is represented by the matrix $\epsilon_1 P_2$ for some $\epsilon_1\in \{-1, 1\}$.

Now consider the restrictions of $A$ and $\ell$ to the orthogonal complement $X_1^\perp$ of $X_1$ with respect to $\ell$.
If $A|_{X_1^\perp}=0$, then $A$ is trivially in the form \eqref{case0A} with $k_2=\ldots k_s=1$ and we can bring $\ell$ to the form \eqref{case0l} by  diagonalization.

If $A|_{X_1^\perp}\neq 0$, then it  satisfies one of the two subcases above and  we can repeat the construction to get $k_2$- dimensional subspace $X_2$ of $X_1^\perp$, $k_2\in \{2,3\}$,  such that in some basis of $X_2$
the map $A|_{X_2}$ is represented by the matrix  $J_{k_2}$ and
the form $\ell|_{X_2}$ is represented by the matrix $\epsilon_2 P_{k_2}$ for some $\epsilon_2\in \{-1, 1\}$. Then we repeat the process for the space $(X_1\oplus X_2)^\perp$ and so on to get the desired normal form \eqref{case0A}-\eqref{case0l}.
$\Box$.

%%%%%%%%%%%%%%%%%%%%%%%%%%%%%%%%%%%%%%%%%%%%%%%%%%%%%%%%%%%%%%%%%%%%%%%%%%%%%%%%%%%%%%%%%%%%%%%%%%%%%%%%%%%%%%%%%%%%%%%
\section{Calculation of Bigraded Universal Algebraic Prolongation of Regular Symbols}\label{algebraic_prolongation}
%%%%%%%%%%%%%%%%%%%%%%%%%%%%%%%%%%%%%%%%%%%%%%%%%%%%%%%%%%%%%%%%%%%%%%%%%%%%%%%%%%%%%%%%%%%%%%%%%%%%%%%%%%%%%%%%%%%%%%%

The purpose of the present section is to prove

\begin{thm}
\label{regbuap}
Let $\g^0$ be a regular non-nilpotent CR symbol with $1$-dimensional Levi kernel and refer to Definition \ref{typedef}:
\begin{enumerate}
\item
If $\g^0$ is strongly non-nilpotent regular of Type $\mathrm {I}_{p,q}$, then $\Re\mathfrak{U}_{\text{bigrad}}(\g^0)$ is isomorphic to $\mathfrak {so}(p+2, q+2)$;
\item
If $\g^0$ is strongly non-nilpotent regular of Type $\mathrm {II}_{p}$, then $\Re\mathfrak{U}_{\text{bigrad}}(\g^0)$ is isomorphic to $\mathfrak {so}^*(2p+4)$;
\item
If $\g^0$ is weakly non-nilpotent regular, then $\Re\mathfrak{U}_{\text{bigrad}}(\g^0)=\Re\g^0$.
\end{enumerate}
\end{thm}

\subsection{Proof of parts (1) and (2) of Theorem \ref{regbuap}}

\subsubsection{ An adapted basis of $\g^0$}
\label{basis}
Let $\g^0$ be strongly non-nilpotent regular symbol with $n=\dim_{\mathbb C}\g_{-1,\pm1}$.
We introduce a convenient basis of $\g_-$ and explicitly calculate bigraded components of $\g^0$ with weight $0$.
 Fix index ranges and constants
\begin{align}\label{g-indices}
&1\leq a,b\leq p,
&&1\leq\alpha,\beta,\gamma\leq n,
&&\epsilon_\alpha=\left\{\begin{smallmatrix}1,&\alpha\leq p;\\-1,&\alpha>p.\end{smallmatrix}\right.
\end{align}
Let $\{e_\alpha\}$ be the basis of $\g_{-1,1}$ from Corollary \eqref{strongclass} so that the Hermitian form $\ell$ and operator $A$ are nicely presented by normal forms \eqref{matrixnormalform1s} or \eqref{matrixnormalform2s}. Including $\overline{e}_\alpha\in\g_{-1,-1}$  and $e_0$ as in \eqref{hermit} completes an adapted basis of $\g_-$:
\begin{align*}
&e_0\in\im\Re\g_{-2,0},
&\{e_\alpha\}\subset\g_{-1,1},
&&\{\overline{e}_{\alpha}\}\subset\g_{-1,-1},
\end{align*}
with bracket relations
\begin{align}\label{g-SE}
\epsilon_\alpha e_0=[ e_\alpha,\overline{e}_{\alpha}],
\end{align}
all others being trivial. The corresponding real basis for $\Re\g_-$ is
\begin{align}\label{reg-}
&\im e_0\in\Re\g_{-2},
&\{(e_\alpha+\overline{e}_\alpha),\ \im(e_\alpha-\overline{e}_\alpha)\}\subset\Re\g_{-1}.
\end{align}
A raised index $e^0,e^\alpha,\overline{e}^\alpha\in(\g_-)^*$ refers to the dual basis vector $e^\alpha(e_\alpha)=1$, etc. which vanishes on all other basis elements. The overline notation for complex conjugation is therefore consistent with \eqref{conjext}.

For the  symbol of Type $\mathrm{I}_{p,q}$ we choose basis elements for $\g_{0,\pm 2}$ by specifying how they act on $\g_{-1,\mp 1}$ according to the normal form for the anti-linear operator $A$ from \eqref{matrixnormalform1s}:
\begin{align}
&e_{n+1}=\ad_v=\sum_{\alpha}e_\alpha\otimes \overline{e}^{\alpha}\in\g_{0,2},
&\overline{e}_{n+1}=\ad_{\overline{v}}=\sum_{\alpha}\overline{e}_{\alpha}\otimes e^{\alpha}\in\g_{0,-2},
\tag{Type I}
\end{align}
whose corresponding bracket relations are
\begin{align}\label{ImSE}
&[ e_{n+1},\overline{e}_{\alpha}]=e_\alpha,
&[ \overline{e}_{n+1},e_{\alpha}]=\overline{e}_{\alpha},
&&[ e_{n+1},\overline{e}_{n+1}]=\sum_\alpha e_\alpha\otimes e^\alpha-\overline{e}_\alpha\otimes \overline{e}^\alpha.
\end{align}

In the case of the symbol of Type $\mathrm {II}_{p}$
\begin{equation}\tag{Type II}
\begin{aligned}
&e_{n+1}=\ad_v=\sum_{a}e_{p+a}\otimes \overline{e}^{a}-e_a\otimes \overline{e}^{p+a}\in\g_{0,2},\\
&\overline{e}_{n+1}=\ad_{\overline{v}}=\sum_{a}\overline{e}_{p+a}\otimes e^{a}-\overline{e}_{a}\otimes e^{p+a}\in\g_{0,-2},
\end{aligned}
\end{equation}
giving brackets
\begin{align}\label{IImSE}
&[ e_{n+1},\overline{e}_{a}] =e_{p+a},
&[ e_{n+1},\overline{e}_{p+a}]=-e_a,
&&[ \overline{e}_{n+1},e_a]=\overline{e}_{p+a},
&&[ \overline{e}_{n+1},e_{p+a}]=-\overline{e}_{a},
\end{align}
\begin{align*}
[ e_{n+1},\overline{e}_{n+1}]=\sum_{\alpha}-e_\alpha\otimes e^\alpha+\overline{e}_\alpha\otimes \overline{e}^\alpha.
\end{align*}
Both symbol types relate to their real forms via
\begin{align}\label{reg0pm2}
(e_{n+1}+\overline{e}_{n+1}),\im(e_{n+1}-\overline{e}_{n+1})\in\Re\g_0.
\end{align}

With $e_0\in\g_{-2,0}$ fixed, one can interpret the Lie bracket as a symplectic form $\sigma\in\Lambda^2(\g_{-1})^*$ on the vector space $\g_{-1,1}\oplus\g_{-1,-1}$ via $[ y_1,y_2]=\sigma(y_1,y_2)e_0$. Each bigraded component $\g_{-1,\pm 1}$ determines a Lagrangian subspace. Identifying $\{e_\alpha,\overline{e}_{\alpha}\}$ with the standard basis of column vectors for $\mathbb{C}^{2n}$, $\sigma(e_\alpha,\overline{e}_{\beta})=e_\alpha^t[\sigma]\overline{e}_{\beta}$ for the matrix $[\sigma]$ divided into $n\times n$ blocks
\begin{align*}
[\sigma]=\left[\begin{array}{cc}0&\epsilon\\-\epsilon&0\end{array}\right],
\end{align*}
with $\epsilon\in\text{Mat}_{n\times n}\mathbb{R}$ as the diagonal matrix $\epsilon_{\alpha,\beta}=\epsilon_\alpha\delta_{\alpha,\beta}$ (Kronecker delta). The symplectic group $Sp_{2n}\mathbb{C}$ consists of $2n\times2n$ matrices $S$ satisfying $S^t[\sigma] S=[\sigma]$, so the group $\hat{G}_{0,0}$ of all bigraded algebra automorphisms of $\g_-$ is represented by the subgroup of $Sp_{2n}\mathbb{C}$ comprised of $n\times n$ block-diagonal matrices, along with the \emph{conformal scaling} operation
\begin{align}\label{conf_scale}
&\{e_0,e_\alpha,\overline{e}_{\alpha}\}\mapsto\{r^2e_0,re_\alpha,r\overline{e}_{\alpha}\},
&r\in\mathbb{R}\setminus\{0\},
\end{align}
corresponding to a different choice of $e_0\in\im\Re\g_{-2}$.

The Lie algebra of $Sp_{2n}\mathbb{C}$ is denoted $\mathfrak{sp}_{2n}\mathbb{C}$, and its standard representation is discussed in \cite{fulton}. The subalgebra $\mathfrak{sp}_{2n}\mathbb{C}\cap\hat{\g}_{0,0}$ is represented by matrices of the form
\begin{align}\label{hatg00_rep}
&\left[\begin{array}{cc}\epsilon B&0\\0&-\epsilon B^t\end{array}\right],
&B\in\text{Mat}_{n\times n}\mathbb{C}.
\end{align}
We can also represent the restrictions to $\g_{-1}$ of $\ad_{e_{n+1}}$ and $\ad_{\overline{e}_{n+1}}$ as off-diagonal matrices in $\mathfrak{sp}_{2n}\mathbb{C}$. For CR symbols of Type I,
\begin{align}\label{ad_typeI}
&\ad_{e_{n+1}}=\left[\begin{array}{cc}0&\mathbbm{1}_n\\0&0\end{array}\right],
&\ad_{\overline{e}_{n+1}}=\left[\begin{array}{cc}0&0\\\mathbbm{1}_n&0\end{array}\right],
\end{align}
where $\mathbbm{1}_n$ is the $n\times n$ identity matrix. For Type II when $n=2p$, we subdivide into $p\times p$ blocks
\begin{align*}
&\ad_{e_{n+1}}=\left[\begin{array}{cccc}0&0&0&-\mathbbm{1}_p\\0&0&\mathbbm{1}_p&0\\0&0&0&0\\0&0&0&0\end{array}\right],
&\ad_{\overline{e}_{n+1}}=\left[\begin{array}{cccc}0&0&0&0\\0&0&0&0\\0&-\mathbbm{1}_p&0&0\\\mathbbm{1}_p&0&0&0\end{array}\right].
\end{align*}

Writing $[\epsilon B]$ to abbreviate the matrix \eqref{hatg00_rep}, $\g_{0,0}\subset\mathfrak{der}(\g_-)$ defined by\eqref{g001}-\eqref{g002}  contains all such matrices which satisfy
\begin{align}\label{g00_rep}
&[\epsilon B]\ad_{ v}-\ad_{ v}[\epsilon B]=c\ad_{ v},
& v=e_{n+1},\overline{e}_{n+1},
&&c\in\mathbb{C}.
\end{align}
In particular, $[\mathbbm{1}_n]\in\g_{0,0}$, and we name
\begin{align}\label{Id}
E=\sum_{\alpha}e_\alpha\otimes e^\alpha-\overline{e}_\alpha\otimes \overline{e}^\alpha,
\end{align}
which is imaginary with respect to conjugation in $\g^0$. In addition to \eqref{g00_rep}, $\g_{0,0}$ includes the infinitesimal generator of the conformal scaling operation \eqref{conf_scale},
\begin{align}\label{conf_alg}
\widehat{E}=2e_0\otimes e^0+\sum_{\alpha}e_\alpha\otimes e^\alpha+\overline{e}_\alpha\otimes \overline{e}^\alpha,
\end{align}
which is real and clearly commutes with $e_{n+1}$ and $\overline{e}_{n+1}$. The real form $\Re\g^0$ will feature
\begin{align}\label{reEhatE}
\im E,\widehat{E}\in\Re\g_{0,0}.
\end{align}

We augment our bracket relations
\begin{equation}\label{ESE}
\begin{aligned}
&[ E,e_\alpha]=e_\alpha=[ \widehat{E},e_\alpha],
&[ E,\overline{e}_\alpha]=-\overline{e}_\alpha=-[ \widehat{E},\overline{e}_\alpha],
&&[\widehat{E},e_0] =2e_0,\\
&[ E,e_{n+1}] =2e_{n+1},
&[ E,\overline{e}_{n+1}]=-2\overline{e}_{n+1},
&&[ e_{n+1},\overline{e}_{n+1}]=\pm E,
\end{aligned}
\end{equation}
where the last bracket depends on the CR symbol Type. We will consider the two symbol Types separately to fill out a basis of $\g_{0,0}$ and perform the algebraic prolongation procedure.
%\medskip
%-\widehat{E} is the grading element, E is the bigrading element

%%%%%%%%%%%%%%%%%%%%%%%%%%%%%%%%%%%%%%%%%%%%%%%%%%%%%%%%%%%%%%%%%%%%%%%%%%%%%%%%%%%%%%%%%%%%%%%%%%%%%%%%%%%%%%%%%%%%%%%
\subsubsection{Type I CR Symbols and the Algebra $\mathfrak{so}(p+2,q+2)$.}
\label{TypeI}
%%%%%%%%%%%%%%%%%%%%%%%%%%%%%%%%%%%%%%%%%%%%%%%%%%%%%%%%%%%%%%%%%%%%%%%%%%%%%%%%%%%%%%%%%%%%%%%%%%%%%%%%%%%%%%%%%%%%%%%

In light of \eqref{ad_typeI}, condition \eqref{g00_rep} is equivalent to $\epsilon(B+B^t)=c\mathbbm{1}_n$. For diagonal $B$, we have already seen in \eqref{Id} that $B=\epsilon\mathbbm{1}_n$ corresponds to $E\in\g_{0,0}$, which leaves the case $c=0$ when $B$ is skew-symmetric. For each pair $\alpha,\beta$ of indices \eqref{g-indices} with $\alpha<\beta$, define
\begin{align*}
e_{\alpha\beta}=\epsilon_\alpha e_\alpha\otimes e^\beta-\epsilon_\beta e_\beta\otimes e^\alpha+\epsilon_\alpha \overline{e}_\alpha\otimes \overline{e}^\beta-\epsilon_\beta \overline{e}_\beta\otimes \overline{e}^\alpha\in\g_{0,0},
\end{align*}
and denote
\begin{align*}
&e_{\beta\alpha}=-e_{\alpha\beta},
&e_{\alpha\alpha}=0.
\end{align*}
Notice that by \eqref{conjext}, $\overline{e_{\alpha\beta}}=e_{\alpha\beta}$, whereby
\begin{align}\label{reIg00}
e_{\alpha\beta}\in\Re\g_{0,0}.
\end{align}
Along with $E$ and $\widehat{E}$ as in \eqref{Id} and \eqref{conf_alg}, these $\binom{n}{2}$ derivations complete a basis for $\g_{0,0}$. They also have structure equations
\begin{align}\label{g00skewSE}
&[ e_{\alpha\beta},e_\gamma]=\delta^\beta_{\gamma}\epsilon_\alpha e_\alpha-\delta^\alpha_{\gamma}\epsilon_\beta e_\beta,
&[ e_{\alpha\beta},\overline{e}_\gamma]=\delta^\beta_{\gamma}\epsilon_\alpha \overline{e}_\alpha-\delta^\alpha_{\gamma}\epsilon_\beta \overline{e}_\beta,
&&[ e_{\alpha\beta},e_{\beta\gamma}]=\epsilon_\beta e_{\alpha\gamma},
\end{align}
all others being trivial.

The following $n$ derivations of bigraded degree (1,1) span $\g_{1,1}$,
\begin{align*}
E_{\alpha}=2e_\alpha\otimes e^0+2\epsilon_\alpha e_{n+1}\otimes e^\alpha+\epsilon_\alpha(\widehat{E}-E)\otimes \overline{e}^\alpha-2\epsilon_\alpha\sum_{\beta}\epsilon_\beta e_{\alpha\beta}\otimes\overline{e}^\beta.
\end{align*}
Similarly, $\g_{1,-1}$ is spanned by
\begin{align*}
\overline{E}_{\alpha}=-2\overline{e}_\alpha\otimes e^0 +2\epsilon_\alpha \overline{e}_{n+1}\otimes \overline{e}^\alpha+\epsilon_\alpha(\widehat{E}+E)\otimes e^\alpha-2\epsilon_\alpha\sum_{\beta}\epsilon_\beta e_{\alpha\beta}\otimes e^\beta.
\end{align*}
The second bigraded prolongation is spanned by the imaginary derivation of bigraded degree $(2,0)$
\begin{align*}
E_0=2\widehat{E}\otimes e^0+\sum_\alpha (E_\alpha\otimes e^\alpha-\overline{E}_\alpha\otimes\overline{e}^\alpha),
\end{align*}
and all higher bigraded prolongations are trivial. For the real form, we have
\begin{align}\label{reg+}
&\{(E_\alpha+\overline{E}_\alpha),\ \im(E_\alpha-\overline{E}_\alpha)\}\subset\Re\g_{1},
&\im E_0\in\Re\g_{2}.
\end{align}

The final algebra $\g=\g^1\oplus\g_{2}$ has dimension $4n+6+\binom{n}{2}=\binom{n+4}{2}$. We will see that $\Re\g$ is isomorphic to a familiar matrix Lie algebra, so fix the constant and index ranges
\begin{align*}
&N=n+4,
&1\leq i,j,k,l\leq N,
\end{align*}
and write $E^i_j$ for the $N\times N$ matrix which has entry 1 in the $i^{\text{th}}$ column and $j^{\text{th}}$ row and zeros elsewhere. Multiplication in $\text{Mat}_{N\times N}\mathbb{C}$ is written in terms of the basis $\{E^i_j\}$ using the Kronecker delta,
\begin{align*}
E^i_jE^k_l=\delta^i_lE^k_j,
\end{align*}
and $\mathfrak{gl}_N\mathbb{R}=\text{Mat}_{N\times N}\mathbb{C}$ is equipped with the Lie bracket defined by the matrix commutator
\begin{align}\label{mat_comm}
[ E^i_j,E^k_l] =\delta^i_lE^k_j-\delta^k_jE^i_l.
\end{align}

$\mathbb{R}^N$ admits a symmetric bilinear form $Q$ of signature $(p+2,q+2)$ defined by
\begin{align*}
&Q=E_1^3+E_3^1+E_2^4+E_4^2+\sum_\alpha\epsilon_\alpha E^{4+\alpha}_{4+\alpha};
&\text{i.e.,}
&&Q=\left[\begin{array}{cc}\begin{smallmatrix}0&0&1&0\\0&0&0&1\\1&0&0&0\\0&1&0&0\end{smallmatrix}&0\\0&\epsilon\end{array}\right].
\end{align*}
Let $\mathfrak{so}(p+2,q+2)$ denote the Lie subalgebra of $\mathfrak{gl}_N\mathbb{R}$ comprised of matrices $C$ which satisfy
\begin{align*}
C^tQ+QC=0.
\end{align*}
To show that $\Re\g$ is isomorphic to $\mathfrak{so}(p+2,q+2)$, we will exhibit a basis of the latter with the same names as that of the former such that the structure equations of each coincide. To begin, set
\begin{align}\label{so42-}
&e_0=E_4^1-E_3^2,
&e_\alpha=\epsilon_\alpha E_3^{4+\alpha}-E_{4+\alpha}^1,
&&\overline{e}_\alpha=\epsilon_\alpha E_4^{4+\alpha}-E_{4+\alpha}^2,
\end{align}
so that these replicate \eqref{g-SE}. Corresponding to $\g_{0,\pm2}$ we have
\begin{align}\label{so42ad}
&e_{n+1}=E_3^4-E^1_2,
&\overline{e}_{n+1}=E_4^3-E^2_1,
\end{align}
brackets for which agree with \eqref{ImSE}. Analogous to basis elements of $\g_{0,0}$ we name
\begin{align*}
&E=-E^1_1+E^2_2+E^3_3-E^4_4,
&\widehat{E}=-E^1_1-E^2_2+E^3_3+E^4_4,
&&e_{\alpha\beta}=\epsilon_\alpha E^{4+\beta}_{4+\alpha}-\epsilon_\beta E^{4+\alpha}_{4+\beta},
\end{align*}
whose brackets with \eqref{so42-} and \eqref{so42ad} in $\mathfrak{so}(p+2,q+2)$ reproduce \eqref{ESE} and \eqref{g00skewSE}. Playing the same role as the first prolongation in $\g$ are the matrices
\begin{align*}
&E_\alpha=2(\epsilon_\alpha E_2^{4+\alpha}-E^4_{4+\alpha}),
&\overline{E}_\alpha=2(\epsilon_\alpha E_1^{4+\alpha}-E^3_{4+\alpha}),
\end{align*}
and a basis of $\mathfrak{so}(p+2,q+2)$ is completed by the matrix
\begin{align*}
E_0=2(E_2^3-E_1^4).
\end{align*}

A change of (real) basis according to \eqref{reg-}, \eqref{reg0pm2}, \eqref{reEhatE}, \eqref{reIg00}, and \eqref{reg+} now shows that $\Re\g$ is isomorphic to $\mathfrak{so}(p+2,q+2)$.

%%%%%%%%%%%%%%%%%%%%%%%%%%%%%%%%%%%%%%%%%%%%%%%%%%%%%%%%%%%%%%%%%%%%%%%%%%%%%%%%%%%%%%%%%%%%%%%%%%%%%%%%%%%%%%%%%%%%%%%
\subsubsection{Type II CR Symbols and the Algebra $\mathfrak{so}^*(2p+4)$}
\label{TypeII}
%%%%%%%%%%%%%%%%%%%%%%%%%%%%%%%%%%%%%%%%%%%%%%%%%%%%%%%%%%%%%%%%%%%%%%%%%%%%%%%%%%%%%%%%%%%%%%%%%%%%%%%%%%%%%%%%%%%%%%%
%\medskip

Recall that Type II is only possible for $p=q$ when $n=2p$. In this case,
\begin{align*}
\epsilon=\left[\begin{array}{cc}\mathbbm{1}_p&0\\0&-\mathbbm{1}_p\end{array}\right],
\end{align*}
and \eqref{hatg00_rep} is clarified by splitting $B\in\text{Mat}_{n\times n}\mathbb{C}$ into $p\times p$ blocks $B_1,B_2,B_3,B_4\in \text{Mat}_{p\times p}\mathbb{C}$ to yield
\begin{align*}
\left[\begin{array}{rrrr}B_1&B_2&0&0\\-B_3&-B_4&0&0\\0&0&-B_1^t&-B_3^t\\0&0&B_2^t&B_4^t\end{array}\right],
\end{align*}
at which point \eqref{g00_rep} becomes
\begin{align*}
&B_2+B_2^t=0=B_3+B_3^t,
&B_1-B_4^t=\pm c\mathbbm{1}_p.
\end{align*}
We have already seen that $c\neq0$ corresponds to $E\in\g_{0,0}$ when $B_1=\mathbbm{1}_p=-B_4^t$, which leaves matrices satisfying $B_4=B_1^t$ while $B_2$ and $B_3$ are skew-symmetric.

For each pair $a,b$ of indices \eqref{g-indices}, define
\begin{align*}
E_{ab}=e_a\otimes e^b-e_{p+b}\otimes e^{p+a}-\overline{e}_b\otimes\overline{e}^a+\overline{e}_{p+a}\otimes\overline{e}^{p+b},
\end{align*}
so that $E_{ab}-E_{ba}$ is real and $E_{ab}+E_{ba}$ is imaginary with respect to conjugation \eqref{conjext} in $\g^0$. For each pair of indices with $a<b$,
\begin{align*}
e_{ab}&=e_a\otimes e^{p+b}-e_{b}\otimes e^{p+a}+\overline{e}_{p+b}\otimes\overline{e}^a-\overline{e}_{p+a}\otimes\overline{e}^b,
&e_{ba}=-e_{ab},
&&e_{aa}=0,\\
\overline{e}_{ab}&=\overline{e}_a\otimes \overline{e}^{p+b}-\overline{e}_{b}\otimes \overline{e}^{p+a}+e_{p+b}\otimes e^a-e_{p+a}\otimes e^b,
&\overline{e}_{ba}=-\overline{e}_{ab},
&&\overline{e}_{aa}=0.
\end{align*}
These $p^2+2\binom{p}{2}=\binom{2p}{2}$ derivations complete $E,\widehat{E}$ to a basis of $\g_{0,0}$. A corresponding basis for the real form is furnished by
\begin{align}\label{reIIg00}
(E_{ab}-E_{ba}), \im(E_{ab}+E_{ba}),(e_{ab}+\overline{e}_{ab}),\im(e_{ab}-\overline{e}_{ab})\in\Re\g_{0,0}.
\end{align}

For the first bigraded prolongation, $\g_{1,1}$ is spanned by $n$ derivations
\begin{align*}
E_a&=2e_a\otimes e^0+2e_{n+1}\otimes e^{p+a}+(\widehat{E}-E)\otimes\overline{e}^a-2\sum_bE_{ab}\otimes\overline{e}^b+2\sum_be_{ab}\otimes\overline{e}^{p+b},\\
E_{p+a}&=2e_{p+a}\otimes e^0+2e_{n+1}\otimes e^a+(E-\widehat{E})\otimes \overline{e}^{p+a}-2\sum_b E_{ba}\otimes\overline{e}^{p+b}+2\sum_b\overline{e}_{ab}\otimes\overline{e}^b,
\end{align*}
while $\g_{1,-1}$ is the span of
\begin{align*}
\overline{E}_a&=-2\overline{e}_a\otimes e^0+2\overline{e}_{n+1}\otimes\overline{e}^{p+a}+(\widehat{E}+E)\otimes e^a+2\sum_bE_{ba}\otimes e^b+2\sum_b\overline{e}_{ab}\otimes e^{p+b},\\
\overline{E}_{p+a}&=-2\overline{e}_{p+a}\otimes e^0+2\overline{e}_{n+1}\otimes\overline{e}^{a}-(\widehat{E}+E)\otimes e^{p+a}+2\sum_bE_{ab}\otimes e^{p+b}+2\sum_be_{ab}\otimes e^{b}.
\end{align*}
As in the case of Type I, the second bigraded prolongation is spanned by the imaginary derivation of bigraded degree $(2,0)$
\begin{align*}
E_0=2\widehat{E}\otimes e^0+\sum_\alpha (E_\alpha\otimes e^\alpha-\overline{E}_\alpha\otimes\overline{e}^\alpha),
\end{align*}
and all higher bigraded prolongations are trivial.

Again, the real form of the final algebra $\g=\g^1\oplus\g_2$ is isomorphic to a real form of the complex, special orthogonal algebra $\mathfrak{so}_{n+4}\mathbb{C}$. Using the same representation as in \cite{fulton}, $\mathfrak{so}_{2(p+2)}\mathbb{C}$ is realized as complex matrices of the form
\begin{align}\label{soC}
&\left[\begin{array}{cr}A&B\\C&-A^t\end{array}\right],
&A,B,C\in\text{Mat}_{(p+2)\times(p+2)}\mathbb{C};
&&B+B^t=0=C+C^t.
\end{align}
The real form $\mathfrak{so}^*(2p+4)\subset \mathfrak{so}_{2(p+2)}\mathbb{C}$ is the subalgebra of matrices \eqref{soC} which additionally satisfy
\begin{align*}
\left[\begin{array}{cr}\overline{A}^t&\overline{C}^t\\\overline{B}^t&-\overline{A}\end{array}\right]
\left[\begin{array}{rc}0&\mathbbm{1}_{p+2}\\-\mathbbm{1}_{p+2}&0\end{array}\right]+
\left[\begin{array}{rc}0&\mathbbm{1}_{p+2}\\-\mathbbm{1}_{p+2}&0\end{array}\right]
\left[\begin{array}{cr}A&B\\C&-A^t\end{array}\right]=0,
\end{align*}
so that $\overline{A}=A$, $\overline{B}=-B$, and $\overline{C}=-C$. This is a real Lie algebra with the matrix commutator bracket \eqref{mat_comm}, and we use the same matrices $E^i_j$ as in \S\ref{TypeI} to construct a basis of $\mathfrak{so}^*(2p+4)$ whose structure equations are the same as those of $\g$. To begin, set
\begin{align*}
&e_0=\im E^{p+2}_{n+3}-\im E_{n+4}^{p+1},
&&e_a=\im E^{p+1}_{p+2+a}-\im E^a_{n+3},
&&e_{p+a}=E^{p+1}_a-E^{p+2+a}_{n+3},\\
&
&&\overline{e}_a=E^{p+2+a}_{n+4}-E^{p+2}_a,
&&\overline{e}_{p+a}=\im E^{p+2}_{p+2+a}-\im E_{n+4}^a,
\end{align*}
so that brackets of these matrices agree with \eqref{g-SE}. Adding
\begin{align*}
&e_{n+1}=E^{p+1}_{p+2}-E^{n+4}_{n+3},
&\overline{e}_{n+1}=E^{n+3}_{n+4}-E^{p+2}_{p+1},
\end{align*}
yields equations \eqref{IImSE}. To replicate \eqref{ESE}, set
\begin{align*}
&E=-E_{p+1}^{p+1}+E^{p+2}_{p+2}+E_{n+3}^{n+3}-E_{n+4}^{n+4},
&\widehat{E}=-E_{p+1}^{p+1}-E^{p+2}_{p+2}+E_{n+3}^{n+3}+E_{n+4}^{n+4},
\end{align*}
along with
\begin{align*}
&E_{ab}=-E_b^a+E_{p+2+a}^{p+2+b},
&e_{ab}=\im E^b_{p+2+a}-\im E^a_{p+2+b},
&&\overline{e}_{ab}=\im E_a^{p+2+b}-\im E^{p+2+a}_{b},
\end{align*}
completing a basis corresponding to that of $\g_{0,0}$. Playing the role of the first bigraded prolongation, we have
\begin{align*}
&E_a=2E^a_{p+2}-2E^{n+4}_{p+2+a},
&&E_{p+a}=2\im E^{n+4}_a-2\im E^{p+2+a}_{p+2},\\
&\overline{E}_a=2\im E_{p+1}^{p+2+a}-2\im E^{n+3}_a,
&&\overline{E}_{p+a}=2E^a_{p+1}-2E^{n+3}_{p+2+a},
\end{align*}
and a basis for $\mathfrak{so}^*(2p+4)$ is completed by
\begin{align*}
E_0=2\im E^{n+3}_{p+2}-2\im E^{n+4}_{p+1}.
\end{align*}

A change of (real) basis according to \eqref{reg-}, \eqref{reg0pm2}, \eqref{reEhatE}, \eqref{reIIg00}, and \eqref{reg+} now shows that $\Re\g$ is isomorphic to $\mathfrak{so}^*(2p+4)$.

This concludes the proof of parts (1) and (2) of Theorem \ref{regbuap}.

%%%%%%%%%%%%%%%%%%%%%%%%%%%%%%%%%%%%%%%%%%%%%%%%%%%%%%%%%%%%%%%%%%%%%%%%%%%%%%%%%%%%%%%%%%%%%%%%%%%%%%
\subsection{Proof of part (3) of Theorem \ref{regbuap}} \label{anot0}
%%%%%%%%%%%%%%%%%%%%%%%%%%%%%%%%%%%%%%%%%%%%%%%%%%%%%%%%%%%%%%%%%%%%%%%%%%%%%%%%%%%%%%%%%%%%%%%%%%%%%%

Recalling Lemma \ref{regclasslem}, we have a splitting $\g_{-1,1}=Z\oplus W$ where
\begin{align*}
&Z=\ker\ad_{\overline{K}},
&W=\text{Im}(\ad_K).
\end{align*}
By \eqref{g002}, it is straightforward to confirm that derivations in $\g_{0,0}$ preserve the splitting of $\g_{-1,1}$. We augment \eqref{g-indices} with index ranges
\begin{align*}
&1\leq\alpha_1,\beta_1\leq p_1+q_1=n_1,
&&\epsilon_{\alpha_1}=\left\{\begin{smallmatrix}1,&\alpha_1\leq p_1;\\-1,&\alpha_1>p_1,\end{smallmatrix}\right.
&&\epsilon_{\alpha}=\left\{\begin{smallmatrix}1,&n_1<\alpha\leq n_1+p-p_1;\\-1,&\alpha>n_1+p-p_1.\end{smallmatrix}\right.
\end{align*}

Let $e_0\in\im\Re\g_{-2}$, $\{e_{\alpha_1}\}\subset W$, $\{e_\alpha\}\subset\g_{-1,1}$ be a basis as described in Theorem \ref{regclass}, and identify $\{e_\alpha,\overline{e}_\alpha\}$ with the standard basis of $\mathbb{C}^{2n}$ as in \S\ref{basis} so that $\g_{0}$ may once again be interpreted as a subalgebra of the conformal symplectic algebra. In contrast to \eqref{ad_typeI}, basis vectors $e_{n+1}\in\g_{0,2}$ and $\overline{e}_{n+1}\in\g_{0,-2}$ for Type I are now represented
\begin{align*}
&\ad_{e_{n+1}}=\left[\begin{array}{cc}0& \delta_{n_1}\\0& 0\end{array}\right],
&\ad_{\overline{e}_{n+1}}=\left[\begin{array}{cc}0& 0\\\delta_{n_1}& 0\end{array}\right],
\end{align*}
where $\delta_{n_1}$ is the $n\times n$ diagonal matrix which is 1 in its first $n_1$ diagonal entries and zero elsewhere. For Type II, we subdivide into $p\times p$ blocks
\begin{align*}
&\ad_{e_{n+1}}=\left[\begin{array}{cccc}0&0&0&-\delta_{p_1}\\0&0&\delta_{p_1}&0\\0&0&0&0\\0&0&0&0\end{array}\right],
&\ad_{\overline{e}_{n+1}}=\left[\begin{array}{cccc}0&0&0&0\\0&0&0&0\\0&-\delta_{p_1}&0&0\\\delta_{p_1}&0&0&0\end{array}\right].
\end{align*}

Aside from the conformal scaling generator $\widehat{E}$ as in \eqref{conf_alg}, $\g_{0,0}$ is represented by matrices
\begin{align*}
&\left[\begin{array}{cccc}\epsilon B_W&0&0&0\\0&\epsilon B_Z&0&0\\0&0&-\epsilon B_W^t&0\\0&0&0&-\epsilon B_Z^t\end{array}\right],
&B_W\in\text{Mat}_{n_1\times n_1}\mathbb{C},
&&B_Z\in\text{Mat}_{(n-n_1)\times(n-n_1)}\mathbb{C},
\end{align*}
where the analog of condition \eqref{g00_rep} requires that $B_W$ is as in the strongly non-nilpotent regular case of Type $\mathrm{I}_{p_1,q_1}$ or $\mathrm{II}_{p_1}$. Therefore, the subalgebra of $\g^0$ comprised of $\g_{-2}\oplus W\oplus\overline{
W}\oplus \g_{0,2}\oplus\g_{0,-2}$ along with $\widehat{E}$ and derivations in $\g_{0,0}$ corresponding to matrices $B_W$ satisfying $\eqref{g00_rep}$ determines a strongly non-nilpotent regular symbol of Type $\mathrm{I}_{p_1,q_1}$ or $\mathrm{II}_{p_1}$, which we refer to as the \emph{underlying strongly non-nilpotent regular symbol} of $\g^0$.

Suppose $f\in\g_{1,1}$ and take $z\in Z$, $\overline{y}\in\g_{-1,-1}$ so that
\begin{align*}
f([z,\overline{y}])=[f(z),\overline{y}]+[z,f(\overline{y})].
\end{align*}
If $\overline{y}\in\overline{W}$, we get $0=[f(z),\overline{y}]+[z,f(\overline{y})]$ which shows that both terms are zero, because $f(\overline{y})\in\g_{0,0}\Rightarrow [z,f(\overline{y})]\in Z$ while $f(z)\in\g_{0,2}\Rightarrow [f(z),\overline{y}]\in W$. If $\overline{y}\in\overline{Z}$, $[f(z),\overline{y}]=0$. Thus, $f(z)\in\g_{0,2}$ acts trivially on all of $\g_{-1}$, whence $f(z)=0$.

Now let $y_1,y_2\in\g_{-1}$ such that $0\neq[y_1,y_2]=y\in\g_{-2}$. If $y_1\in Z$ and $y_2\in\overline{Z}$, then the fact that $f(y_1)=0$ implies
\begin{align*}
f(y)=[y_1,f(y_2)],
\end{align*}
which lies in $Z$ since $f(y_2)\in\g_{0,0}$. On the other hand, if $y_1\in W$ and $y_2\in\overline{W}$,
\begin{align*}
f(y)=[f(y_1),y_2]+[y_1,f(y_2)],
\end{align*}
which lies in $W$. Thus $f(y)=0$. The Lie bracket pairs $Z$ nondegenerately with $\overline{Z}$, so the fact that $f$ acts trivially on $Z$ and $\g_{-2}$ implies $f$ vanishes on $\overline{Z}$ as well.

Similar arguments show $f\in\g_{1,-1}$ acts trivially on $Z\oplus\overline{Z}$. In this way we see that $\g_1$ is a subspace of the first bigraded prolongation of the underlying strongly non-nilpotent regular symbol of $\g^0$. However, no nontrivial degree-1 derivations of such a strongly non-nilpotent regular symbol vanish on $\g_{-2}$, so it must be that $\g_1=0$, and Theorem \ref{regbuap} is proved. \hfill$\Box$

%%%%%%%%%%%%%%%%%%%%%%%%%%%%%%%%%%%%%%%%%%%%%%%%%%%%%%%%%%%%%%%%%%%%%%%%%%%%%%%%%%%%%%%%%%%%%%%%%%%%%%
\subsection{Bigraded Prolongations of Nilpotent Regular Symbols.}
%%%%%%%%%%%%%%%%%%%%%%%%%%%%%%%%%%%%%%%%%%%%%%%%%%%%%%%%%%%%%%%%%%%%%%%%%%%%%%%%%%%%%%%%%%%%%%%%%%%%%%

The results of this subsection -- summarized in Theorem \ref{regbuap0} at the end -- rely completely on analysis of the normal form constructed in Theorem \ref{regclass0}. The latter theorem implies that we have an $A$-stable splitting of $\g_{-1,1}$ into $\ell$-orthogonal subspaces of dimension $\leq 3$, where at least one of them has dimension 2 or 3, and provides adapted bases for each subspace in the direct sum. Although such splitting is not canonical in general, we fix one of them. We refer to $k$-dimensional subspaces of the splitting as $(A_k)$ summands and denote them by $\mathfrak{a}_k\subset\g_{-1,1}$, where $k\in\{1,2,3\}$. When working with adapted bases, the restriction of $A$ to an $(A_k)$ summand is represented by a matrix  $J_k$ and the restriction of the form $\ell$ by the matrix $\pm P_k$. It will be understood that $e_0\in\g_{-2,0}$ and $v\in\g_{0,2}$ are fixed so that normal forms of $A$ and $\ell$ are as in Theorem \ref{regclass0} for all of $\g_{-1,1}$.

We first consider $\g_{-1,1}$ which splits into a single $(A_3)$ summand and $n_1\geq0$ $(A_1)$ summands; i.e.,
\begin{align*}
\g_{-1,1}=\mathfrak{a}_3\oplus\bigoplus_{i=1}^{n_1}\mathfrak{a}_1^i,
\end{align*}
with the convention that $\bigoplus_{i=1}^{n_1}\mathfrak{a}_1^i=0$ when $n_1=0$. Take basis vectors
\begin{align*}
&e_0\in\g_{-2,0},
&v\in\g_{0,2},
&&e_1,e_2,e_3\in\mathfrak{a}_3,
&&w_i\in\mathfrak{a}^i_1,
\end{align*}
with nontrivial brackets
\begin{align*}
&[e_1,\overline{e}_3]=[e_3,\overline{e}_1]=[e_2,\overline{e}_2]=e_0,
&[v,\overline{e}_2]=e_1,[v,\overline{e}_3]=e_2,
&&[\overline{v},e_2]=\overline{e}_1,[\overline{v},e_3]=\overline{e}_2,\\
&[w_i,\overline{w}_i]=\epsilon_ie_0,
&\epsilon_i=\left\{\begin{smallmatrix}0&\text{ if }&n_1=0\\\pm1&\text{ if }&n_1>0\end{smallmatrix}\right..
\end{align*}

As above, fixing $e_0\in\g_{-2,0}$ allows us to interpret the Lie bracket on $\g_-$ as a symplectic form, imposing conformal symplectic structure on $\g_{-1}$ such that  $\g_{-1,\pm1}$ are Lagrangian. The algebra $\g_{0,0}$ is the subalgebra of the conformal symplectic algebra satisfying conditions \eqref{g001} and \eqref{g002}. Let $\hat \g_{0,0}$ be the intersection of $\g_{0,0}$ with the symplectic algebra. Denoting the dual basis of $\g_-$ with superscripts, we name the infinitesimal generator of the conformal scaling operator
\begin{align*}
\widehat{E}=2e_0\otimes e^0+\sum_{j=1}^3e_j\otimes e^j+\overline{e}_j\otimes \overline{e}^j+\sum_{i=1}^{n_1}w_i\otimes w^i+\overline{w}_i\otimes \overline{w}^i,
\end{align*}
so that $\g_{0,0}$ is the direct sum of $\hat \g_{0,0}$ and the line generated by $\widehat{E}$. The space of elements of  $\hat\g_{0,0}$ which preserve the subspace $\mathfrak{a}_3\oplus\overline{\mathfrak{a}}_3$ while acting trivially on the $(A_1)$ summands has the following basis:
\begin{align*}
&E_{13}=e_1\otimes e^1-\overline{e}_3\otimes \overline{e}^3,
&&\overline{E}_{13}=\overline{e}_1\otimes\overline{e}^1-e_3\otimes e^3,\\
&e_{13}=e_1\otimes e^3-\overline{e}_1\otimes \overline{e}^3,
&&e_{22}=e_2\otimes e^2-\overline{e}_2\otimes \overline{e}^2.
\end{align*}
This can be shown by direct verification of conditions \eqref{g001} and \eqref{g002}, though it also follows from the proof of Theorem \ref{regclass0}: with vectors $b_i$ in that proof denoted here by $e_{4-i}$, we have that $e_3$ is defined in $\mathfrak{a}_3$ up to scale and modulo $e_1$ (more precisely, modulo $\im e_1$ over $\mathbb R$, but for the prolongation we work over $\mathbb C$). This freedom is generated by $\overline E_{13}$ and $e_{13}$, respectively. Further, $e_2$ and $e_1$ are defined up to scale, with scaling operators generated by $e_{22}$ and $E_{13}$. Similarly, one can show that the elements of $\g_{0,0}$ that interchange vectors in $\mathfrak{a}_3$ and $\mathfrak{a}_1^i$ are spanned by
\begin{align*}
&W_i=\epsilon_ie_1\otimes w^i-\epsilon_i\overline{w}_i\otimes\overline{e}^3,
&\overline{W}_i=\epsilon_i\overline{e}_1\otimes\overline{w}^i-\epsilon_i w_i\otimes e^3.
\end{align*}
To complete our basis of $\g_{0,0}$, we would include those derivations from $\hat\g_{0,0}$ which preserve the subspace $\bigoplus_{i=1}^{n_1}(\mathfrak{a}_1^i\oplus\overline{\mathfrak{a}}_1^i)$ while vanishing on $\mathfrak{a}_3\oplus\overline{\mathfrak{a}}_3$.
The space of these derivations can be identified with $\mathfrak{gl}\Bigl(\bigoplus_{i=1}^{n_1}(\mathfrak{a}_1^i\oplus\overline{\mathfrak{a}}_1^i)\Bigr)$, because any element of $\mathfrak{gl}\Bigl(\bigoplus_{i=1}^{n_1}(\mathfrak{a}_1^i\oplus\overline{\mathfrak{a}}_1^i)\Bigr)$ can be uniquely extended to an element $\mathfrak{sp}(\g_{-1})$ vanishing on $\mathfrak{a}_3\oplus\overline{\mathfrak{a}}_3$. We will not bother to name a basis of these derivations, but only note that in total $\dim\g_{0,0}=5+2n_1+{n_1}^2$.

Now one calculates that $\g_{1,1}$ is spanned by
\begin{align*}
&E_1=2e_1\otimes e^0+2v\otimes e^2-2e_{13}\otimes\overline{e}^1+(\widehat{E}+e_{22}+\overline{E}_{13}-3E_{13})\otimes\overline{e}^3-2\sum_{i=1}^{n_1} W_i\otimes\overline{w}^i,
\end{align*}
so that $\g_{1,-1}$ is spanned by the conjugate
\begin{align*}
&\overline{E}_1=-2\overline{e}_1\otimes e^0+2\overline{v}\otimes \overline{e}^2+2e_{13}\otimes e^1+(\widehat{E}-e_{22}+E_{13}-3\overline{E}_{13})\otimes e^3-2\sum_{i=1}^{n_1} \overline{W}_i\otimes w^i,
\end{align*}
and higher bigraded prolongations are trivial.

Thus we see that in this class of examples, the prolongations are essentially the same as when $n_1=0$, which we refer to as the \emph{basic $(A_3)$ case}. Our next result shows that these are the only $\g_{-}$ with $(A_3)$ summands that admit nontrivial prolongations. Because the proof is somewhat tedious, we divide the statement into two parts.

\begin{prop}
\label{redbasis}
Suppose there is an $A$-stable, $\ell$-orthonormal splitting $\g_{-1,1}=\mathfrak{a}_3\oplus W$.
\begin{enumerate}
\item The restriction of any $f\in\g_{1}$ to $\g_{-2,0}\oplus\mathfrak{a}_3\oplus\overline{\mathfrak{a}}_3$ is an element of the first bigraded prolongation of the symbol in the basic $(A_3)$ case.

\item If $W$ contains an $(A_2)$ or $(A_3)$ summand, then $\g_1=0$.
\end{enumerate}
\end{prop}

\begin{proof}
It suffices to consider $f\in\g_{1,1}$, as $\g_{1,-1}$ follows by complex conjugation. We use the same names of adapted basis vectors as before.

(1) That the splitting is $\ell$-orthogonal implies $[\overline{e}_1,w]=0$ for every $w\in W$, and since $\overline{e}_1$ lies in the kernel of $fw\in\g_{0,2}$, we see that $[f\overline{e}_1,w]=0$; i.e., $f\overline{e}_1\in\g_{0,0}$ vanishes on $W$. To see that it also vanishes on $\g_{-2,0}$, we first observe that $[f\overline{e}_1,e_2]=0$ by the same argument. Note that \eqref{g002} implies that derivations in $\g_{0,0}$ preserve the kernel of those in $\g_{0,\pm2}$, so expanding
\begin{align*}
0=f[\overline{e}_1,\overline{e}_2]=[f\overline{e}_1,\overline{e}_2]+[\overline{e}_1,f\overline{e}_2],
\end{align*}
it must be that $f\overline{e}_2\in\g_{0,0}$ keeps $\overline{e}_1$ in $\ker v$, whence $[f\overline{e}_1,\overline{e}_2]$ is in (the complex conjugate of) the $\ell$-orthogonal complement of $e_2$. Now we have $f\overline{e}_1(e_0)=f\overline{e}_1[e_2,\overline{e}_2]=0$, and since $f\overline{e}_1$ vanishes on both $W$ and $\g_{-2,0}$, it also vanishes on $\overline{W}$. In particular, $0=f\overline{e}_1[e_3,\overline{w}]=[f\overline{e}_1(e_3),\overline{w}]$ for every $\overline{w}\in\overline{W}$, which shows that $e_3$ remains in the $\ell$-orthogonal complement to $W$ under the action of $f\overline{e}_1$. That is, $[f\overline{e}_1,e_3]\in\mathfrak{a}_3$, yielding
\begin{align*}
fe_0=f[e_3,\overline{e}_1]=[e_3,f\overline{e}_1]\in\mathfrak{a}_3.
\end{align*}

$A$-stability of the splitting means $fy\in\g_{0,2}$ maps $\overline{\mathfrak{a}}_3$ into $\mathfrak{a}_3$ for every $y\in\g_{-1,1}$. Therefore, having shown that $f$ maps $\g_{-2,0}$ into $\g_{-2,0}\oplus\mathfrak{a}_3\oplus\overline{\mathfrak{a}}_3$, it remains to show that $f\overline{e}\in\g_{0,0}$ preserves both $\mathfrak{a}_3\subset\g_{-1,1}$ and $\overline{\mathfrak{a}}_3\subset\g_{-1,-1}$ for every $\overline{e}\in\overline{\mathfrak{a}}_3$. The former is almost immediate from $fe_0\in\mathfrak{a}_3$, as
\begin{align*}
\underbrace{f[e_i,\overline{e}_j]}_{\in\mathfrak{a}_3}=\underbrace{[fe_i,\overline{e}_j]}_{\in\mathfrak{a}_3}+[e_i,f\overline{e}_j].
\end{align*}
To see that $f\overline{e}_j$ also fixes $\overline{\mathfrak{a}}_3$, it is equivalent to prove that it vanishes on $W$. Indeed, if $f\overline{e}_j(w)=0$ for every $w\in W$, then $0=f\overline{e}_j[\overline{e}_i,w]=[f\overline{e}_j(\overline{e}_i),w]$ and $f\overline{e}_j$ leaves $\overline{e}_i$ in the $\ell$-orthogonal complement to $W$. Therefore, we can show that $f\overline{e}_j$ annihilates $W$ for $j=1,2,3$. The first case is immediate from the fact that $\overline{e}_1\in\ker fw$ by expanding $0=f[\overline{e}_1,w]$. For the remaining cases, note that
\begin{align*}
&0=f[\overline{e}_j,w]=[f\overline{e}_j,w]+\lambda_je_{j-1},
&&j=2,3,
\end{align*}
for some $\lambda_2,\lambda_3\in\mathbb{C}$, which is a consequence of the fact that $fw\in\g_{0,2}$ is some multiple of $v$. Expanding $0=f[\overline{e}_2,\overline{e}_3]$ will show $f\overline{e}_2(\overline{e}_3)=f\overline{e}_3(\overline{e}_2)$, so that we can compare
\begin{align*}
&0=f\overline{e}_2[w,\overline{e}_3]=-\lambda_2e_0+[w,f\overline{e}_2(\overline{e}_3)],
&0=f\overline{e}_3[w,\overline{e}_2]=-\lambda_3e_0+[w,f\overline{e}_3(\overline{e}_2)],
\end{align*}
and reveal $\lambda_2=\lambda_3$. If $W$ is nontrivial, then it contains some nonzero $w$ in the kernel of $\overline{v}$, and this kernel must be preserved by $f\overline{e}_3\in\g_{0,0}$, but we have seen that $[f\overline{e}_3,w]=-\lambda_3e_2$, so we conclude $\lambda_2=\lambda_3=0$.

(2) This is immediate if $W$ contains another $(A_3)$ component, as the arguments of part (1) applied to both $\mathfrak{a}_3$'s will show that $f\in\g_{1,1}$ must map $e_0$ into two different direct summands, and there is no nontrivial derivation in the basic $(A_3)$ case with $fe_0=0$. It remains to consider $\mathfrak{a}_2\subset W$, for which we take a basis $w_1,w_2$ with nontrivial brackets
\begin{align*}
&[w_1,\overline{w}_2]=\epsilon e_0=[w_2,\overline{w}_1] \quad(\epsilon=\pm1),
&[v,\overline{w}_2]=w_1, [\overline{v},w_2]=\overline{w}_1.
\end{align*}
Let us assume for contradiction that $f\in\g_{1,1}$ is nontrivial, and that it is scaled such that its restriction to $\g_{-2,0}\oplus\mathfrak{a}_3\oplus\overline{\mathfrak{a}}_3$ agrees with $E_1$ from the basic $(A_3)$ case. In particular,
\begin{align}\label{fe3}
&[f\overline{e}_3,e_0]=2e_0,
&[f\overline{e}_3,\overline{v}]=0.
\end{align}
the second bracket relation from \eqref{fe3} gives
\begin{align*}
0=[f\overline{e}_3,\overline{v}](w_2)=f\overline{e}_3(\overline{w}_1)-\overline{v}\circ f\overline{e}_3(w_2)=f\overline{e}_3(\overline{w}_1),
\end{align*}
where we have used the fact from the proof of part (1) that $f\overline{e}_3|_W=0$. However, using the first bracket relation of \eqref{fe3} we compute
$$2\epsilon e_0=f\overline{e}_3[w_2,\overline{w}_1]=[w_2,f\overline{e}_3(\overline{w}_1)]=0,$$
a contradiction. Hence, $f=0$.
\end{proof}

It remains to consider $\g_{-1,1}$ composed of $n_2\geq 1$ $(A_2)$ summands and $n_1\geq 0$ $(A_1)$ summands,
\begin{align*}
\g_{-1,1}=\Big(\bigoplus_{i=1}^{n_1}\mathfrak{a}_1^i\Big)\oplus\Big(\bigoplus_{j=1}^{n_2}\mathfrak{a}_2^j\Big),
\end{align*}
so we take basis vectors
\begin{align*}
&e_0\in\g_{-2,0},
&v\in\g_{0,2},
&&e^j_1,e^j_2\in\mathfrak{a}_2^j,
&&w_i\in\mathfrak{a}^i_1,
\end{align*}
with nontrivial brackets
\begin{align*}
&[e^j_1,\overline{e}^j_2]=[e_2^j,\overline{e}_1^j]=\epsilon_je_0 \quad(\epsilon_j=\pm1),
&[v,\overline{e}^j_2]=e_1^j, &&[\overline{v},e_2^j]=\overline{e}_1^j,\\
&[w_i,\overline{w}_i]=\varepsilon_{i}e_0,
&\varepsilon_{i}=\left\{\begin{smallmatrix}0&\text{ if }&n_1=0\\\pm1&\text{ if }&n_1>0\end{smallmatrix}\right..
\end{align*}

Dual basis vectors to $e_1^j,e_2^j$ will be denoted $(e_1^{j})^*,(e_2^{j})^*$, while the dual to $w_i$ will still be written $w^i$. As usual, we name
\begin{align*}
\widehat{E}=2e_0\otimes e^0+\sum_{j=1}^{n_2}e_1^j\otimes (e_1^{j})^*+\overline{e}_1^j\otimes (\overline{e}_1^{j})^*+e_2^j\otimes (e_2^{j})^*+\overline{e}_2^j\otimes (\overline{e}_2^{j})^*
+\sum_{i=1}^{n_1}w_i\otimes w^i+\overline{w}_i\otimes \overline{w}^i.
\end{align*}
Aside from $\widehat{E}$, derivations in $\g_{0,0}$ which act nontrivially on and preserve every $\mathfrak{a}_2^j$ are spanned by
\begin{align*}
&E=\sum_{j=1}^{n_2}e_1^j\otimes (e^j_1)^*-\overline{e}_2^j\otimes (\overline{e}^j_2)^*,
&\overline{E}=\sum_{j=1}^{n_2}\overline{e}_1^j\otimes (\overline{e}^j_1)^*-e_2^j\otimes (e^j_2)^*,
\end{align*}
and each individual $\mathfrak{a}_2^j$ is additionally preserved by
\begin{align*}
&e^j_{12}=e_1^j\otimes (e^j_2)^*-\overline{e}_1^j\otimes(\overline{e}^j_2)^*.
\end{align*}
If $n_2>1$ and $1\leq j_1< j_2\leq n_2$, derivations which interchange vectors in $\mathfrak{a}_2^{j_1}$ and $\mathfrak{a}_2^{j_2}$ can simply swap bases
\begin{align*}
E_{j_1,j_2}&=\epsilon_{j_1}e^{j_2}_1\otimes(e^{j_1}_1)^*-\epsilon_{j_2}e^{j_1}_1\otimes(e^{j_2}_1)^*
+\epsilon_{j_1}\overline{e}^{j_2}_2\otimes(\overline{e}^{j_1}_2)^*-\epsilon_{j_2}\overline{e}^{j_1}_2\otimes(\overline{e}^{j_2}_2)^*,
&&(E_{j_2,j_1}=-E_{j_1,j_2})\\
\overline{E}_{j_1,j_2}&=\epsilon_{j_1}\overline{e}^{j_2}_1\otimes(\overline{e}^{j_1}_1)^*-\epsilon_{j_2}\overline{e}^{j_1}_1\otimes(\overline{e}^{j_2}_1)^*
+\epsilon_{j_1}e^{j_2}_2\otimes(e^{j_1}_2)^*-\epsilon_{j_2}e^{j_1}_2\otimes(e^{j_2}_2)^*,
&&(\overline{E}_{j_2,j_1}=-\overline{E}_{j_1,j_2})
\end{align*}
or they can ``raise" each basis with $\ad_v$ and $\ad_{\overline{v}}$ as they switch,
\begin{align*}
&E^{j_1,j_2}=\epsilon_{j_1}e^{j_2}_1\otimes(e^{j_1}_2)^*-\epsilon_{j_2}\overline{e}^{j_1}_1\otimes(\overline{e}^{j_2}_2)^*,
&\overline{E}^{j_1,j_2}=\epsilon_{j_1}\overline{e}^{j_2}_1\otimes(\overline{e}^{j_1}_2)^*-\epsilon_{j_2}e^{j_1}_1\otimes(e^{j_2}_2)^*,
\end{align*}
which are also skew in $j_1,j_2$. If $n_1>0$, derivations in $\g_{0,0}$ which interchange vectors in $\mathfrak{a}_2^j$ and $\mathfrak{a}_1^i$ are spanned by
\begin{align*}
&W_i=\varepsilon_ie_1^j\otimes w^i-\epsilon_j\overline{w}_i\otimes(\overline{e}_2^j)^*,
&\overline{W}_i=\varepsilon_i\overline{e}_1^j\otimes \overline{w}^i-\epsilon_jw_i\otimes(e_2^j)^*.
\end{align*}
As before, our basis of $\g_{0,0}$ is completed with $(n_1)^2$ operators in the symplectic algebra of the subspace $\bigoplus_{i=1}^{n_1}(\mathfrak{a}_1^i\oplus\overline{\mathfrak{a}}_1^i)$ that preserve the maximal Lagrangian $n_1$-planes, and we do not write them.

For the first bigraded prolongation, we have $2n_2+n_1$ basis elements of $\g_{1,1}$,
\begin{align*}
E^j_1&=2\epsilon_je^j_1\otimes e^0+2v\otimes (e^j_1)^*-2e^j_{12}\otimes(\overline{e}^j_1)^*+(\widehat{E}-E+\overline{E})\otimes (\overline{e}^j_2)^*-2\sum_{i=1}^{n_1} W_i\otimes\overline{w}^i\\
&+2\epsilon_j\sum_{j'\neq j}\overline{E}^{j,j'}\otimes(\overline{e}^{j'}_1)^*+E_{j,j'}\otimes(\overline{e}^{j'}_2)^*,
\\
E^j_2&=v\otimes (e^j_2)^*+e^j_{12}\otimes(\overline{e}^j_2)^*+\sum_{j'\neq j}E^{j,j'}\otimes(\overline{e}^{j'}_{2})^*,
\\
E_i&=\varepsilon_iv\otimes w^i+\sum_{j=1}^{n_2}W_i\otimes(\overline{e}^j_2)^*
\end{align*}
and their conjugates provide a basis for $\g_{1,-1}$. The final nontrivial prolongation $\g_2$ is spanned by
\begin{align*}
E_0=\sum_{j=1}^{n_2}E^j_2\otimes(e^j_2)^*-\overline{E}^j_2\otimes(\overline{e}^j_2)^*.
\end{align*}

We summarize the results of this section with the following

\begin{thm}
\label{regbuap0}
Let $\g^0$ be a nilpotent regular CR symbol with $1$-dimensional Levi kernel, referring to Lemma \ref{regclasslem}, Definition \ref{strongreg}, and Theorem \ref{regclass0}. Suppose $\g_{-1,1}$ admits a normal form with $n_1$ $(A_1)$ summands, $n_2$ $(A_2)$ summands, and $n_3$ $(A_3)$ summands so that $\dim\g_-=1+2n_1+4n_2+6n_3$.
\begin{enumerate}
\item If $n_3>1$ or $n_3=1$ and $n_2>0$ then $\Re\mathfrak{U}_{\text{bigrad}}(\g^0)=\Re\g^0$.

\item If $n_3=1$ and $n_2=0$ then $\Re\mathfrak{U}_{\text{bigrad}}(\g^0)=\Re\g^1$ where $\dim\g_{0,0}=5+2n_1+(n_1)^2$ and $\dim\g_1=2$.

\item Otherwise, $n_3=0$ and $\Re\mathfrak{U}_{\text{bigrad}}(\g^0)=\Re\g^2$ where $\dim\g_{0,0}=3+n_2+4\binom{n_2}{2}+2n_1n_2+(n_1)^2$, $\dim\g_1=4n_2+2n_1$, and $\dim\g_2=1$.\hfill$\Box$
\end{enumerate}
\end{thm}

Elementary arguments from Theorem \ref{regbuap0} and the calculation in the proof of Theorem \ref{regbuap} will show that among CR manifolds $M$ of a given dimension, the maximal dimension of the bigraded universal prolongation of a regular symbol with $1$-dimensional Levi kernel is $\tfrac{1}{4}(\dim M-1)^2+7$, and it is achieved for the nilpotent regular symbol with $n_2=1$ and $n_3=0$.

%\newpage

%%%%%%%%%%%%%%%%%%%%%%%%%%%%%%%%%%%%%%%%%%%%%%%%%%%%%%%%%%%%%%%%%%%%%%%%%%%%%%%%%%%%%%%%%%%%%%%%%%%%%%%%%%%%%%%%%%%%%%%%%%
\section{Proof of Theorem \ref{maintheor}} \label{geometric_prolongation}
%%%%%%%%%%%%%%%%%%%%%%%%%%%%%%%%%%%%%%%%%%%%%%%%%%%%%%%%%%%%%%%%%%%%%%%%%%%%%%%%%%%%%%%%%%%%%%%%%%%%%%%%%%%%%%%%%%%%%%%%%%

\subsection{Proof of Part (1): Geometric Prolongation}
%%%%%%%%%%%%%%%%%%%%%%%%%%%%%%%%%%%%%%%%%%%%%%%%%%%%%%%%%%%%%%%%%%%%%%%%%%%%%%%%%%%%%%%%%%%%%%%%%%%%%%%%%%%%%%%%%%%%%%%%%%
\label{geom_prolong_part1}

Consider a regular symbol $\g^0$ whose bigraded universal algebraic prolongation $\mathfrak U_{\text{bigrad}}(\g^0)$ has the bigraded splitting \eqref{bguniv}, such that $\g_i=\bigoplus_{j\in\mathbb Z}\g_{i,j}$ consists of all elements with first weight $i$. Let $l$ be the nonnegative integer such that $\g_l\neq 0$ but $\g_{l+1}=0$. For instance, when $\dim_\mathbb{C}\g_{0,\pm2}=1$ Theorem \ref{regbuap} shows $l=2$ for strongly non-nilpotent regular symbols and $l=0$ for weakly non-nilpotent regular symbols.

As in the standard Tanaka theory (see Remark \ref{Chainrem} and set $\mu=2$), we will recursively construct a sequence of bundles
\begin{equation}
\label{bundles1}
P^{-1}=M\leftarrow P^0\leftarrow P^1\leftarrow P^2\leftarrow\cdots
\end{equation}
such that $P^0$ is a principal bundle whose structure group has Lie algebra $\g_{0,0}$ and for $i>0$, $P^i$ is a bundle over $P^{i-1}$ whose fibers are affine spaces with modeling vector space $\g_i$.
In each step of the inductive procedure, the construction of the bundle $P^{i+1}\rightarrow P^{i}$ produces a concomitant bundle $\Re P^{i+1}\rightarrow \Re P^{i}$. Here, $\Re P^0$ is a principal bundle whose structure group has Lie algebra isomorphic to the real part $\Re \g_{0,0}$ of $\g_{0,0}$, and for $i>0$ $\Re P^i$ is a bundle over $\Re P^{i-1}$ whose fibers are affine spaces with modeling vector space equal to the real part $\Re \g_i$ of $\g_i$ (see the discussion in the paragraph before Theorem \ref{maintheor}). The bundle $\Re P^{l+2}$ is endowed with a canonical frame -- a structure of absolute parallelism.

\subsubsection{Bigraded Frame Bundle $P^0$}\label{frame_bundle}
%%%%%%%%%%%%%%%%%%%%%%%%%%%%%%%%%%%%%%%%%%%%%%%%%%%%%%%%%%%%%%%%%%%%%%%%%%%%%%%%%%%%%%%%%%%%%%%%%%%%%%%%%%%%%%%%%%%%%%%%%%

Let $x\in M$. Recalling definitions \eqref{g-}, we name the canonical quotient projections
\begin{align}\label{q-}
&q_{-1}:\mathbb{C}D_x\to \g_{-1}(x),
&q_{-2}:\mathbb{C}T_xM\to\g_{-2}(x).
\end{align}
Regarding the splitting $\mathbb{C}D=H\oplus\overline{H}$, we also name the summand projections
\begin{align}\label{ppm}
&p_+:\mathbb{C}D\to H,
&p_-:\mathbb{C}D\to\overline{H},
\end{align}
furnishing linear projections onto the bigraded components
\begin{align}\label{q-1pm1}
&q_{-1,1}=q_{-1}\circ p_+:\mathbb{C}D_x\to\g_{-1,1}(x),
&q_{-1,-1}=q_{-1}\circ p_-:\mathbb{C}D_x\to\g_{-1,-1}(x).
\end{align}
For the remaining bigraded components of the CR symbol algebra at $x$ (Definition \ref{CRsymboldef}) that also lie in $\m(x)$ (see \eqref{m}), recall
\begin{align*}
&\ad_{K_x}=\g_{0,2}(x),
&\ad_{\overline{K}_x}=\g_{0,-2}(x).
\end{align*}

Fix an abstract symbol (Definition \ref{absCRsymboldef}) $\g^0$, which is isomorphic to $\g^0(x)$ for every $x\in M$. For any vector subspace $\mathfrak{s}\subset\g^0$ defined as a sum of specified bigraded components of $\g^0$, we will denote by $I(\mathfrak{s})$ the indexing set of all biweights of the components defining $\mathfrak{s}$. For example, $I(\m)=\{(i,j)\ |\ (i,j)=(-2,0),(-1,\pm1),(0,\pm2)\}$.

We now define a bundle $\pi: P^0\to M$ whose fiber over $x\in M$ is comprised of all \emph{adapted frames}, or bigraded Lie algebra isomorphisms,
\begin{align}
\label{fiberP0}
  P^0_x=\left\{
\varphi_x:\g_-\longrightarrow\g_-(x) \left|\ \begin{array}{cc}
\varphi_x(\g_{i,j})=\g_{i,j}(x)&  (i,j)\in I(\g_-)\\
\varphi_x^{-1}\circ \g_{0,\pm2}(x)\circ\varphi_x=\g_{0,\pm2}&\\
\varphi_x([ y_1, y_2])=[\varphi_x( y_1),\varphi_x( y_2)] & y_1, y_2\in\g_-
\end{array}\right.\right\}.
\end{align}
The Lie algebra $\g_{0,0}\subset\g^0$ is tangent to the Lie group $G_{0,0}\subset\Aut(\g_-)$ of bigraded algebra isomorphisms of $\g_-$ whose adjoint action on $\mathfrak{der}(\g_-)$ preserves the
%lines
spaces
$\g_{0,\pm2}$. $P^0$ is a principal $ G_{0,0}$-bundle, where the right principal action $R_g:P^0\to P^0$ of $g\in G_{0,0}$ on each fiber is given by
\begin{align}\label{principal-action}
 R_g(\varphi_x)=\varphi_x\circ g:\g_-\to\g_-(x).
\end{align}
$\Re P^0\to M$ will denote the subbundle of $P^0$ whose fiber over $x$ consists of those $\varphi_x$ as in \eqref{fiberP0} which map $\Re \g_-$ to $\Re \g_-(x)$. $\Re P^0$ is a principal $ \Re G_{0,0}$-bundle (see Definition \ref{flatdef} and the comments preceding it).

The CR filtration of $TM$ induces a filtration on $T P^0$ via the inverse image of the pushforward $\pi_*$, so we name
\begin{align*}
&T^{-2}P^0=TP^0,
&T^{-1}P^0=\pi_*^{-1}(D),
&&T^0P^0=\pi_*^{-1}(\Re(K\oplus\overline{K})),
\end{align*}
with $\mathbb{C}T^{i}P^0$ being the inverse image of the corresponding complexification ($i=-2,-1,0$). For later use, we also name the canonical quotient projection
\begin{align}\label{Qmodzero}
Q:\mathbb{C}T^{-2}P^0\to\mathbb{C}T^{-2}P^0/\mathbb{C}T^{0}P^0.
\end{align}  
The complexified filters have subbundles
\begin{align*}
&T^{-1,1} P^0=(\pi_*)^{-1}(H)\subset\mathbb{C}T^{-1} P^0,
&&T^{-1,-1} P^0=(\pi_*)^{-1}(\overline{H})\subset\mathbb{C}T^{-1} P^0,\\
&T^{0,2} P^0=(\pi_*)^{-1}(K)\subset \mathbb{C}T^{0} P^0,
&&T^{0,-2} P^0=(\pi_*)^{-1}(\overline{K})\subset \mathbb{C}T^{0} P^0,
\end{align*}
\begin{align*}
 T^{0,0}  P^0=\ker\pi_*  \subset T^{0,\pm2}  P^0.
\end{align*}

$P^0$ is equipped with intrinsically defined \emph{soldering forms} taking values in $\g^0$. In the lowest graded degree, we have a $\mathbb{C}$-linear one-form defined at the frame $\varphi_x\in P^0$ using the quotient projection \eqref{q-},
\begin{align*}
&\theta_{-2}=\theta_{-2,0}:\mathbb{C}T^{-2}P^0\to\g_{-2};
&&\theta_{-2}|_{\varphi_x}=(\varphi_x)^{-1}\circ q_{-2}\circ \pi_*.
\end{align*}
We will use the same names for the soldering forms when descending to a quotient of their domain by any subbundle of their kernel. By its definition,
\begin{align*}
&\ker\theta_{-2}=\mathbb{C}T^{-1}P^0,
&\theta_{-2}:T^{-2}P^0/T^{-1}P^0\stackrel{\simeq}{\longrightarrow}\Re\g_{-2},
\end{align*}
where $\simeq$ indicates a linear isomorphism which extends by $\mathbb{C}$-linearity to the complexification. The remaining soldering forms are not true one-forms; i.e., they are not defined on all of $\mathbb{C}T P^0$, but only on individual filters. For instance, the quotient projection \eqref{q-} provides
\begin{align*}
&\theta_{-1}:\mathbb{C}T^{-1}P^0\to\g_{-1};
&\theta_{-1}|_{\varphi_x}=(\varphi_x)^{-1}\circ q_{-1}\circ \pi_*.
\end{align*}
In this case,
\begin{align*}
&\ker\theta_{-1}=\mathbb{C}T^0P^0,
&\theta_{-1}:T^{-1}P^0/T^{0}P^0\stackrel{\simeq}{\longrightarrow}\Re\g_{-1}.
\end{align*}
The definitions of the soldering forms, along with that of the bracket and the fact that $\varphi_x$ is an algebra isomorphism, ensure the following bracket commutation relation, known as ``regularity" in the study of parabolic geometries (cf. \cite[Ch.3]{capslovak}),
\begin{align}\label{g-reg}
&\theta_{-2}([Y_1,Y_2])=[\theta_{-1}(Y_1),\theta_{-1}(Y_2)],
&Y_1,Y_2\in\Gamma(\mathbb{C}T^{-1}P^0).
\end{align}

Incorporating the projections \eqref{q-1pm1} splits $\theta_{-1}$ into bigraded components,
\begin{align*}
&\theta_{-1,\pm1}:\mathbb{C}T^{-1}P^0\to\g_{-1,\pm1},
&\theta_{-1,\pm1}|_{\varphi_x}=(\varphi_x)^{-1}\circ q_{-1,\pm1}\circ\pi_*,
\end{align*}
which satisfy
\begin{align*}
&\ker\theta_{-1,\pm1}=T^{-1,\mp1}P^0,
&\theta_{-1,\pm1}:T^{-1,\pm1}P^0/T^{0,\pm2}P^0\stackrel{\simeq}{\longrightarrow}\g_{-1,\pm1},
\end{align*}
the latter being a $\mathbb{C}$-linear isomorphism. Composing the $\ad_K$, $\ad_{\overline{K}}$ operators with the projections \eqref{ppm},
\begin{align*}
&\theta_{0,\pm2}:\mathbb{C}T^0P^0\to\g_{0,\pm2},
&\theta_{0,\pm2}|_{\varphi_x}(v)=(\varphi_x)^{-1}\circ \ad_{p_\pm\circ\pi_*(v)}\circ\varphi_x,
&&v\in\mathbb{C}T_{\varphi_x}^0P^0,
\end{align*}
and we have
\begin{align*}
&\ker\theta_{0,\pm2}=T^{0,\mp2}P^0,
&\theta_{0,\pm 2}:T^{0,\pm2}P^0/T^{0,0}P^0\stackrel{\simeq}{\longrightarrow} \g_{0,\pm2}.
\end{align*}
By their definitions and \eqref{adv}, these satisfy a bigraded version of the regularity condition \eqref{g-reg},
\begin{align}\label{bigreg}
&\theta_{-1,\pm1}([V,Y])=[\theta_{0,2}(V),\theta_{-1,\mp1}(Y)],
&V\in\Gamma(T^{0,\pm2}P^0),
&&Y\in\Gamma(T^{-1,\mp1}P^0).
\end{align}

At this point, we have defined soldering forms taking values in each of the bigraded components of $\m$. They possess an important equivariance property with respect to the principal $G_{0,0}$ action \eqref{principal-action} on the fibers of $P^0$, which determines a diffeomorphism of $P^0$.

\begin{lem}\label{m-forms-equiv}
 Let $g\in G_{0,0}$ and $\varphi_x\in P^0$. For any $(i,j)\in I(\m)$,
\begin{align*}
R_g^*\theta_{i,j}=\Ad_{g^{-1}}\circ\theta_{i,j}
\end{align*}
\end{lem}
\begin{proof}
We restate as follows. For $(i,j)\in I(\g_-)$ and $Y\in\Gamma(\mathbb{C}T^{i}P^0)$,
\begin{align*}
 \Bigl(R_g^*\bigl(\left.\theta_{i,j}\right|_{R_g(\varphi_x)}
 \bigr)\Bigr)(Y)=g^{-1}\left.\theta_{i,j}\right|_{\varphi_x}(Y),
\end{align*}
and for $V\in\Gamma(\mathbb{C}T^{0}P^0)$,
\begin{align*}
 \Bigl(R_g^*\bigr(\left.\theta_{0,\pm2}\right|_{R_g(\varphi_x)}\bigr)\Bigr)(V)=g^{-1}\left.\theta_{0,\pm2}\right|_{\varphi_x}(V)g,
\end{align*}
where $R_g^*$ in the left-hand sides of the last two equations acts from the space of forms at $R_g(\varphi_x)$ to the space of corresponding forms at $\varphi_x$.
First observe that $x=\pi\circ R_g(\varphi_x)=\pi(\varphi_x)\Rightarrow (\pi\circ R_g)_*=\pi_*$. By definition of $\theta_{-2}$ and $R_g(\varphi_x)$,
 \begin{align*}
  \left(R_g^*\left.\theta_{-2}\right|_{R_g(\varphi_x)}\right)(Y(\varphi_x))&=(\varphi_x\circ g)^{-1}\circ q_{-2}\circ\pi_*((R_g)_*Y)\\
  &=g^{-1}\circ(\varphi_x)^{-1}\circ q_{-2}\circ\pi_*(Y)\\
  &=g^{-1}\left.\theta_{-2}\right|_{\varphi_x}(Y).
 \end{align*}
The $(-1)$-graded forms may be treated similarly. The analogous arguments  for $\theta_{0,\pm2}$ are also immediate from their definition, viz.,
\begin{align*}
 \Bigl(R_g^*\bigr(\left.\theta_{0,\pm2}\right|_{R_g(\varphi_x)}\bigr)\Bigr)(V)&=	\bigl(R_g(\varphi_x)\bigr)^{-1}\circ \ad_{p_\pm\circ\pi_*\bigl((R_g)_*V\bigr)}\circ R_g(\varphi_x)\\
 &=g^{-1}\circ (\varphi_x)^{-1}\circ \ad_{p_\pm\circ\pi_*(V)}\circ\varphi_x\circ g\\
 &=g^{-1}\left.\theta_{0,\pm2}\right|_{\varphi_x}(V)g.
\end{align*}
\end{proof}

Recall that the vertical bundle of a principal bundle is trivialized by \emph{fundamental vector fields} which are associated to vectors in $\g_{0,0}$ by way of the principal action \eqref{principal-action} and the Lie algebra exponential map $\exp:\g_{0,0}\to G_{0,0}$. Specifically, to $ v\in\g_{0,0}$ we associate the vertical vector field
\begin{align}\label{G00_fund_vf}
\zeta_ v(\varphi_x)=\left.\frac{d}{dt}\right|_{t=0}R_{\exp(t v)}(\varphi_x).
\end{align}
Thus we obtain another $\g_0$-valued form
\begin{align*}
 \theta_{0,0}:T^{0,0} P^0\to\g_{0,0},
\end{align*}
which acts isomorphically by 
\begin{align*}
\theta_{0,0}(\zeta_ v)= v,\ \forall v\in\g_{0,0}.
\end{align*}

\begin{cor}\label{equiv-reg}
	Let $v\in\g_{0,0}$ and $Y\in\Gamma(T^{i,j}P^0)$ with $\theta_{i,j}(Y)= y\in\g_{i,j}$ where $(i,j)\in I(\g^0)$. Then
	\begin{align*}
	\theta_{i,j}([\zeta_v,Y])=[v,y].
	\end{align*}
\end{cor}

\begin{proof}
When $(i,j)=(0,0)$, $Y$ is also fundamental, and the result follows from the fact that the algebra of fundamental vector fields on a $G_{0,0}$-principal bundle is isomorphic to $\g_{0,0}$. 

Let $\mathcal L_{\zeta_ v}$ denote the Lie derivative along the vector field $\zeta_ v$ and recall Cartan's formula for the Lie derivative of a one-form $\alpha\in\Omega(P^0)$,
\begin{align*}
 (\mathcal{L}_{\zeta_v}\alpha)(Y)=\dd\alpha(\zeta_ v,Y)+\dd(\alpha(\zeta_ v))(Y).
\end{align*}
$\theta_{i,j}$ is not necessarily a true one-form, so the exterior derivative $\dd\theta_{i,j}$ does not make sense in general. However, since $\theta_{i,j}$ vanishes identically on $\zeta_ v\in\Gamma(T^{0,0}P^0)$ and takes constant value $ y\in\g_{i,j}$ along the vector field $Y\in\Gamma(T^{i,j}P^0)$, we can use the definition of the exterior derivative to interpret
\begin{align*}
 (\mathcal{L}_{\zeta_v}\theta_{i,j})(Y)=\dd\theta_{i,j}(\zeta_ v,Y)=-\theta_{i,j}([\zeta_ v,Y]).
\end{align*}

On the other hand, $R_{\exp(t v)}(\varphi_x)$ is the integral curve of $\zeta_ v$ passing through $\varphi_x$ when $t=0$, so we have the definition of the Lie derivative given by
\begin{align*}
 \mathcal{L}_{\zeta_v}\theta_{i,j}=\left.\frac{d}{dt}\right|_{t=0}R_{\exp(t v)}^*\theta_{i,j}.
\end{align*}
Applying equivariance as in Lemma \ref{m-forms-equiv} to the right-hand-side of this, we see
\begin{align*}
-\theta_{i,j}([\zeta_ v,Y])&= \left.\frac{d}{dt}\right|_{t=0}\Ad_{\exp(-t v)}\circ\theta_{i,j}(Y)\\
&=-\ad_ v\circ\theta_{i,j}(Y)\\
&=-[ v, y].\qedhere
\end{align*}
\end{proof}

%%%%%%%%%%%%%%%%%%%%%%%%%%%%%%%%%%%%%%%%%%%%%%%%%%%%%%%%%%%%%%%%%%%%%%%%%%%%%%%%%%%%%%%%%%%%%%%%%%%%%%%%%%%%%%%%%%%%%%%%%%
\subsubsection{First Prolongation}\label{first_prolong}
%%%%%%%%%%%%%%%%%%%%%%%%%%%%%%%%%%%%%%%%%%%%%%%%%%%%%%%%%%%%%%%%%%%%%%%%%%%%%%%%%%%%%%%%%%%%%%%%%%%%%%%%%%%%%%%%%%%%%%%%%%

We now begin the process of extending the soldering forms so that they are true one-forms. This cannot be done canonically, so we must incorporate into our construction the ambiguity of the choice of such extensions.

\begin{definition}\label{hatP1def}
$\hat{\pi}^1:\hat{P}^1\to P^0$ is the bundle whose fiber over $\varphi\in P^0$ is composed of maps
\begin{align*}
\varphi_1\in\Hom(\g_{-2},\mathbb{C}T^{-2}_{\varphi}P^0/\mathbb{C}T^{0}_{\varphi}P^0)\oplus\bigoplus_{(i,j)\in I(\g_{-1}\oplus\g_0)}\Hom(\g_{i,j},T^{i,j}_{\varphi}P^0),
\end{align*}
which satisfy
\begin{align*}
&\theta_{i,j}\circ\varphi_1|_{\g_{i,j}}=\mathbbm{1}_{\g_{i,j}};
&&(i,j)\in I(\g^0),
\end{align*}
where $\mathbbm{1}_{\g_{i,j}}$ is the identity map on $\g_{i,j}$. The subbundle $\Re\hat{P}^1\to P^0$ is determined by those $\varphi_1\in\hat{P}^1$ with
\begin{align*}
&\varphi_1(\overline{ y})=\overline{\varphi_1( y)},
&\forall y\in\g^0.
\end{align*}
\end{definition}

It is straightforward to see that the fiber over $\varphi\in P^0$ is nonempty: for a basis $\mathfrak{b}$ of $\g_{i,j}$, there is a basis of $T^{i,j}_\varphi P^0$ that $\theta_{i,j}$ maps to $\mathfrak{b}$, and one can simply take $\varphi_1|_{\g_{i,j}}$ to invert $\theta_{i,j}$ on $\mathfrak{b}$. In order to characterize the fibers of $\hat{\pi}^1:\hat{P}^1\to P^0$, let $\hat{\g}_1\subset\mathrm{End}(\g^0)$ be the vector space of $\mathbb{C}$-linear endomorphisms which have degree-1 on $\g_-$ and map
\begin{align}\label{hatg1zero}
&\g_{0,\pm2}\to\g_{0,0},
&\g_{0,0}\to 0.
\end{align}
$f\in\hat{\g}_1$ can be visualized by the diagram
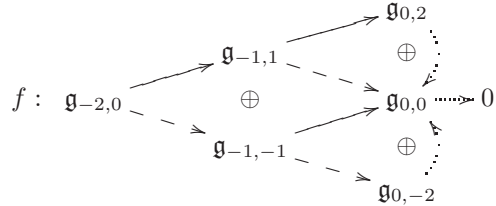
\begin{displaymath}
\xymatrix@=0.25em{
	& & & & &  & & & &\g_{0,2}\ar@/^1pc/@{.>}[dd]& &\\
	& & & & & \g_{-1,1}\ar@{-->}[rrrrd] \ar[rrrru]& & & &\oplus& &\\
	f :& \g_{-2,0}\ar@{-->}[rrrrd]\ar[rrrru]& & & &\oplus& & & &\g_{0,0}\ar@{.>}[rr]& &0\\
	& & & & & \g_{-1,-1}\ar@{-->}[rrrrd] \ar[rrrru]& & & &\oplus& &\\
	& & & & & & & & &\g_{0,-2}\ar@/_1pc/@{.>}[uu]& &\\
}\end{displaymath}
so that the solid and dashed lines have bidegrees $(1,1)$ and $(1,-1)$, respectively. By definition of $\hat{\g}_1$, the image of $\varphi_1\circ f|_{\g_{i,j}}$ is in the kernel of $\theta_{i,j}$ for each $\varphi_1\in\hat{P}^1_\varphi$, so for every $f\in\hat{\g}_1$ we can define a map on the fibers 
\begin{align}\label{hatP1action}
&R_{\mathbbm{1}+f}:\hat{P}^1_\varphi\to\hat{P}^1_\varphi,
&R_{\mathbbm{1}+f}(\varphi_1)|_{\g_{i,j}}=\left\{\begin{array}{ll}
\varphi_1+Q\circ\varphi_1\circ f, &(i,j)=(-2,0), \\
\varphi_1+\varphi_1\circ f,& (i,j)\in I(\g_{-1}\oplus\g_0),\end{array}\right.
\end{align}
where we have used the canonical quotient projection \eqref{Qmodzero}. Further, to each $f\in\hat{\g}_1$ we 
associate the vector field defined point-wise at $\varphi_1\in\hat{P}^1$ by
\begin{align}\label{P1zeta}
\zeta_f(\varphi_1)=\left.\frac{d}{dt}\right|_{t=0}R_{\mathbbm{1}+tf}(\varphi_1).
\end{align}

\begin{prop}\label{hatP1fibers}
For every $\varphi_1$ in the fiber $\hat P^1_{\varphi}$ of the bundle $\hat{\pi}^1:\hat P^1\rightarrow P^0$ over the point  $\varphi\in P^0$, the maps $A_{\varphi_1}: \hat{\g}_1\rightarrow \hat{P}^1_{\varphi}$ and $T_{\varphi_1}:\hat{\g}_1\rightarrow T_{\varphi_1} \hat P^1_{\varphi}$ given by  $A_{\varphi_1}(f)=R_{\mathbbm{1}+f}(\varphi_1)$ and $T_{\varphi_1}(f)=\zeta_f(\varphi_1)$ are bijections. The analogous statement holds for the subbundle $\Re\hat{P}^1\rightarrow P^0$ with $\hat{\g}^1$ replaced by $\Re\hat{\g}_1$ and $A_{\varphi_1}$, $T_{\varphi_1}$ by their restrictions to $\Re\hat{\g}_1$.
\end{prop}

\begin{proof}
Injectivity of $A_{\varphi_1}$ follows by Definition \ref{hatP1def} and the definition \eqref{hatP1action} of $R_{\mathbbm{1}+f}$. To prove surjectivity, let $\tilde{\varphi}_1$ be any other point in the same fiber $\hat P^1_{\varphi}$, and recall that both $\varphi_1$ and $\tilde\varphi_1$ are maps with domain $\g^0$. Consider their difference as maps restricted to each bigraded component of $\g^0$, beginning with $\g_{-2,0}$,
\begin{align*}
0=\mathbbm{1}_{\g_{-2,0}}-\mathbbm{1}_{\g_{-2,0}}=\theta_{-2,0}\circ(\tilde{\varphi}_1-\varphi_1)|_{\g_{-2,0}},
\end{align*}
so $(\tilde{\varphi}_1-\varphi_1)|_{\g_{-2,0}}$ takes values in $\ker\theta_{-2}$, i.e.,
\begin{align*}
(\tilde{\varphi}_1-\varphi_1)|_{\g_{-2,0}}:\g_{-2,0}\to \mathbb{C}T_\varphi^{-1}P^0/\mathbb{C}T_\varphi^{0}P^0.
\end{align*}
We compose with $\theta_{-1}$ to define a linear map
\begin{align*}
 \theta_{-1}\circ (\tilde{\varphi}_1-\varphi_1)|_{\g_{-2,0}}= f^1_{-2,0}:\g_{-2,0}\to\g_{-1}.
\end{align*}
Taking into account the splitting of $\theta_{-1}$ into $\theta_{-1,\pm1}$, we identify bigraded components
\begin{displaymath}
 \xymatrix@=0.5em{ & & &\g_{-1,1}\\ f^1_{-2,0}:  &\g_{-2,0}\ar[rru]^{ f _{-2,0}^{1,1}}\ar@{-->}[rrd]_{ f_{-2,0}^{1,-1}}& &\oplus\\ & & &\g_{-1,-1}}.
\end{displaymath}
Since $\theta_{-1}$ is an isomorphism on $\mathbb{C}T_\varphi^{-1}P^0/\mathbb{C}T_\varphi^{0}P^0$ and $\varphi|_{\g_{-1,\pm1}}$ inverts $\theta_{-1,\pm1}$,
\begin{align*}
 (\tilde{\varphi}_1-\varphi_1)|_{\g_{-2,0}}=Q\circ(\varphi_1|_{\g_{-1,1}}\circ f^{1,1}_{-2,0}+\varphi_1|_{\g_{-1,-1}}\circ f^{1,-1}_{-2,0}),
\end{align*}
or more succinctly,
\begin{align*}
 \tilde{\varphi}_1|_{\g_{-2,0}}=\varphi_1|_{\g_{-2,0}}+Q\circ\varphi_1\circ f_{-2,0}.
\end{align*}

Moving on, we adduce Definition \ref{hatP1def} again to write
\begin{align*}
 0=\theta_{-1,\pm1}((\tilde{\varphi}_1-\varphi_1)|_{\g_{-1,\pm1}})\Longrightarrow (\tilde{\varphi}_1-\varphi_1)|_{\g_{-1,\pm1}}:\g_{-1,\pm1}\to T_\varphi^{0,\pm2}P^0.
\end{align*}
Composing the latter with $\theta_{0,\pm2}$ determines $f^1_{-1,\pm1}:\g_{-1,\pm1}\to \g_{0,\pm2}$ with components
\begin{displaymath}\xymatrix@=0.5em{
& & & & & &\g_{0,2} &\\
\g_{-1,1}\ar[rrrrrru]_{f_{-1,1}^{1,1}}& & & & & &\oplus &\\
\oplus& & & & & & \g_{0,0} &\\
\g_{-1,-1}\ar@{-->}[rrrrrrd]^{f_{-1,-1}^{1,-1}}& & & & & &\oplus &\\
 & & & & & &\g_{0,-2}& \\
}.
\end{displaymath}
By its definition, the map
\begin{align*}
 (\tilde{\varphi}_1-\varphi_1)|_{\g_{-1,\pm1}}-\varphi_1|_{\g_{0,\pm2}}\circ f^1_{-1,\pm1}:\g_{-1,\pm1}\to\mathbb{C}T_\varphi^{-1}P^0
\end{align*}
takes values in the kernel of $\theta_{0,\pm2}$, so evaluating the isomorphism $\theta_{0,0}$ on its image yields
\begin{displaymath}\xymatrix@=0.5em{
 & & & & & &\g_{0,2} &\\
\g_{-1,1}\ar@{-->}[rrrrrrd]^{f_{-1,1}^{1,-1}}& & & & & &\oplus &\\
\oplus& & & & & & \g_{0,0} &\\
\g_{-1,-1}\ar[rrrrrru]_{f_{-1,-1}^{1,1}}& & & & & &\oplus &\\
 & & & & & &\g_{0,-2}& \\
}
\end{displaymath}

Next we have that $(\tilde{\varphi}_1-\varphi_1)|_{\g_{0,\pm2}}$ takes values in $\ker\theta_{0,\pm2}=T_\varphi^{0,0}P^0$, so this difference factors through the isomorphism $\varphi_1|_{\g_{0,0}}$ by a linear map
\begin{align*}
 f^{0,\mp2}_{0,\pm2}:\g_{0,\pm2}\to\g_{0,0}.
\end{align*}
Note that $\tilde{\varphi}_1|_{\g_{0,0}}=\varphi_1|_{\g_{0,0}}=(\theta_{0,0})^{-1}$, so the last component of $f:\g^0\to\g^0$ is the trivial map $f:\g_{0,0}\to 0$ as in \eqref{hatg1zero}. This concludes the proof of surjectivity and therefore bijectivity of $A_{\varphi_1}$, which immediately implies that $T_{\varphi_1}$ is bijective.

Finally, we observe that $\varphi_1,\tilde{\varphi}_1\in\Re\hat{P}^1$ with $\tilde{\varphi}_1=\varphi_1+\varphi_1\circ f$ implies $f(\overline{ y})=\overline{f( y)}$, so $\overline{f}=f$ as defined in \eqref{conjext} and $f\in\Re\hat{\g}_1$.
\end{proof}

\begin{rem}
Note that the maps \eqref{hatP1action} do not determine a group action on $\hat{P}^1$ because in general, elements of $\hat{\g}_1$ do not have degree $1$ by \eqref{hatg1zero}. So, despite the fact that each fiber of $\hat P^1$ is canonically identified with $\hat{\g}_1$ (after fixing one point $\varphi_1$ in it) via the bijection $A_{\varphi_1}$, this fiber is not an affine space modeled on $\hat{\g}_1$, and the bundle $\hat \pi^1:\hat P^1\rightarrow P^0$ is neither affine nor a principal bundle. Nevertheless, the tangent space to the fiber of $\hat P^1$ at every point $\varphi_1$ can be canonically identified with $\hat{\g}_1$ via the map $T_{\varphi_1}$, so the vector fields \eqref{P1zeta} play a role similar to  that of fundamental vector fields on a principal bundle. 
\end{rem}
The bundle $\hat{\pi}^1:\hat{P}^1\to P^0$ inherits a filtration as before,
\begin{align*}
&\mathbb{C}T^i\hat{P}^1=(\hat{\pi}^1_*)^{-1}(\mathbb{C}T^iP^0),
&T^{i,j}\hat{P}^1=(\hat{\pi}^1_*)^{-1}(T^{i,j}P^0),
&&(i,j)\in I(\g^0),
\end{align*}
where each subbundle contains the vertical bundle
\begin{equation}
\label{vert1}
\mathbb{C}T^1\hat{P}^1=\ker\hat{\pi}^1_*.
\end{equation}
$\hat{P}^1$ also admits tautologically defined, $\g^0$-valued soldering forms which may be thought of as extending those on $P^0$. In particular, the forms of a given degree have ``graduated" in terms of the filters on which they are defined.

\begin{prop}\label{hatP1forms}
There exist intrinsically defined forms
\begin{align*}
&\hat{\theta}^1_{-2}:\mathbb{C}T^{-2}\hat{P}^1\to\g_{-2},
&\hat{\theta}^1_{-1}:\mathbb{C}T^{-2}\hat{P}^1\to\g_{-1},
&&\hat{\theta}^1_0:\mathbb{C}T^{-1}\hat{P}^1\to\g_{0},
\end{align*}
with $\g_{i,j}$-valued bigraded components $\hat{\theta}^1_{i,j}$ which restrict to give the pullback of the soldering forms on $P^0$,
\begin{align*}
&\hat{\theta}^1_{0,0}|_{T^{0,0}\hat{P}^1}=(\hat{\pi}^1)^*\theta_{0,0}
&\hat{\theta}^1_{i,j}|_{\mathbb{C}T^i\hat{P}^1}=(\hat{\pi}^1)^*\theta_{i,j}
&&(i,j)\in I(\m).
\end{align*}
Furthermore, if $\varphi_1,\ \tilde{\varphi}_1\in\hat{P}^1$ are two elements in the fiber over $\varphi\in P^0$ which are related by $\tilde{\varphi}_1=R_{\mathbbm{1}+f}(\varphi_1)$ for $f\in\hat{\g}_1$ as in Proposition \ref{hatP1fibers}, then
\begin{align*}
(R_{\mathbbm{1}+f})^*\hat{\theta}^1_{-2}|_{\tilde{\varphi}_1}&=\hat{\theta}^1_{-2}|_{\varphi_1},
&&(R_{\mathbbm{1}+f})^*\hat{\theta}^1_{-1}|_{\tilde{\varphi}_1}=(\hat{\theta}^1_{-1}-f\circ\hat{\theta}^1_{-2})|_{\varphi_1},\\
(R_{\mathbbm{1}+f})^*\hat{\theta}^1_{0,\pm2}|_{\tilde{\varphi}_1}&=(\hat{\theta}^1_{0,\pm2}-f\circ\hat{\theta}^1_{-1})|_{\varphi_1},
&&(R_{\mathbbm{1}+f})^*\hat{\theta}^1_{0,0}|_{\tilde{\varphi}_1}=(\hat{\theta}^1_{0,0}-(f-f|_{\g_0}\circ f)\circ\hat{\theta}^1_{-1}-f\circ \hat{\theta}^1_{0,\pm2})|_{\varphi_1},
\end{align*}
so that the zero-graded form transforms according to
\begin{align*}
(R_{\mathbbm{1}+f})^*\hat{\theta}^1_{0}|_{\tilde{\varphi}_1}&=(\hat{\theta}^1_{0}-(f-f\circ  f)\circ\hat{\theta}^1_{-1}-f\circ\hat{\theta}^1_{0})|_{\varphi_1}.
\end{align*}
\end{prop}
\begin{proof}
The first soldering form we define will satisfy all of the hypotheses immediately,
\begin{align*}
\hat{\theta}^1_{-2}=(\hat{\pi}^1)^*\theta_{-2}.
\end{align*}
Note that the transformation rule follows from $\hat{\pi}\circ R_{\mathbbm{1}+f}=\hat{\pi}\Rightarrow \hat{\pi}_*\circ (R_{\mathbbm{1}+f})_*=\hat{\pi}_*$, a fact we'll tacitly use when addressing the remaining forms.

Let $Y\in\mathbb{C}T_{\varphi_1}\hat{P}^1$ so that $\hat{\pi}^1_*(Y)\in\mathbb{C}T_\varphi^{-2}P^0$ and recall the canonical quotient projection \eqref{Qmodzero}. By Definition \ref{hatP1def},
\begin{align*}
0=\theta_{-2}(Q\circ\hat{\pi}^1_*(Y)-\varphi|_{\g_{-2}}\circ\theta_{-2}\circ\hat{\pi}^1_*(Y)),
\end{align*}
so $Q\circ\hat{\pi}^1_*(Y)-\varphi|_{\g_{-2}}\circ\theta_{-2}\circ\hat{\pi}^1_*(Y)$ lies in the kernel of $\theta_{-2}$ (when restricted to the image of $Q$), which is $\mathbb{C}T_\varphi^{-1}P^0/\mathbb{C}T_\varphi^0P^0$ -- i.e., the bundle on which $\theta_{-1}$ acts as an isomorphism. Therefore, define
\begin{align*}
\hat{\theta}^1_{-1}|_{\varphi_1}(Y)=\theta_{-1}(Q\circ\hat{\pi}^1_*(Y)-\varphi|_{\g_{-2}}\circ\theta_{-2}\circ\hat{\pi}^1_*(Y)).
\end{align*}
In particular, if $Y\in\mathbb{C}T^{-1}_{\varphi_1}\hat{P}^1$ then $\theta_{-2}\circ\hat{\pi}^1_*(Y)=0$ and we have
\begin{align*}
\hat{\theta}^1_{-1}=(\hat{\pi}^1)^*(\theta_{-1}\circ Q)=(\hat{\pi}^1)^*\theta_{-1}.
\end{align*}
To see how $\hat{\theta}^1_{-1}$ transforms under the action of $R_{\mathbbm{1}+f}$, plug in $\tilde{\varphi}_1=\varphi_1+\varphi_1\circ f$,
\begin{align*}
(R_{\mathbbm{1}+f})^*\hat{\theta}^1_{-1}|_{\tilde{\varphi}_1}(Y)&=\theta_{-1}(Q\circ\hat{\pi}^1_*(Y)-\tilde{\varphi}|_{\g_{-2}}\circ\theta_{-2}\circ\hat{\pi}^1_*(Y))\\
&=\theta_{-1}(Q\circ\hat{\pi}^1_*(Y)-(\varphi|_{\g_{-2}}+\varphi|_{\g_{-1}}\circ f)\circ\theta_{-2}\circ\hat{\pi}^1_*(Y))\\
&=\hat{\theta}^1_{-1}|_{\varphi_1}(Y)-f\circ(\hat{\pi}^1)^*\theta_{-2}(Y),
\end{align*}
where we have used Definition \ref{hatP1def} once again.

Now let $Y\in\mathbb{C}T^{-1}_{\varphi_1}\hat{P}^1$ and note
\begin{align*}
\hat{\pi}^1_*(Y)-\varphi_1\circ \theta_{-1}\circ \hat{\pi}^1_*(Y)\in\ker\theta_{-1}=\mathbb{C}T^0_\varphi P^0,
\end{align*}
so we set
\begin{align*}
\hat{\theta}^1_{0,\pm2}|_{\varphi_1}(Y)=\theta_{0,\pm2}(\hat{\pi}^1_*(Y)-\varphi_1\circ \theta_{-1}\circ \hat{\pi}^1_*(Y)).
\end{align*}
If $Y\in \mathbb{C}T^{0}_{\varphi_1}\hat{P}^1$ then $\theta_{-1}\circ \hat{\pi}^1_*(Y)=0$ and we have
\begin{align*}
\hat{\theta}^1_{0,\pm2}|_{\varphi_1}(Y)=(\hat{\pi}^1)^*\theta_{0,\pm2}(Y),
\end{align*}
as claimed. The transformation property follows like before,
\begin{align*}
(R_{\mathbbm{1}+f})^*\hat{\theta}^1_{0,\pm2}|_{\tilde{\varphi}_1}(Y)&=\theta_{0,\pm2}(\hat{\pi}^1_*(Y)-(\varphi_1|_{\g_{-1}}+\varphi_1|_{\g_{0}}\circ f)\circ \theta_{-1}\circ \hat{\pi}^1_*(Y))\\
&=\hat{\theta}^1_{0,\pm2}|_{\varphi_1}(Y)-f\circ(\hat{\pi}^1)^*\theta_{-1}(Y),
\end{align*}
where the latter term is $f\circ\hat{\theta}^1_{-1}(Y)$ by virtue of the fact that $\hat{\theta}^1_{-1}=(\hat{\pi}^1)^*\theta_{-1}$ on $\mathbb{C}T^{-1}\hat{P}^1$. Note that $f\circ\hat{\theta}^1_{-1}$ is understood to refer only to those components of $f$ which map into $\g_{0,\pm2}$,
\begin{align*}
f\circ\hat{\theta}^1_{-1}=f_{-1,\pm1}^{1,\pm1}\hat{\theta}^1_{-1,\pm1},
\end{align*}
as $\theta_{0,\pm2}$ vanishes on the image of $\varphi_1|_{\g_{0,j}}$ for $j=0,\mp2$.

Next we observe
\begin{align*}
\hat{\pi}^1_*(Y)-\varphi_1\circ \theta_{-1}\circ \hat{\pi}^1_*(Y)-\varphi_1\circ\hat{\theta}^1_{0,2}|_{\varphi_1}(Y)-\varphi_1\circ\hat{\theta}^1_{0,-2}|_{\varphi_1}(Y)\in\ker\theta_{0,2}\cap\ker\theta_{0,-2}=T_\varphi^{0,0} P^0,
\end{align*}
and define
\begin{align*}
\hat{\theta}^1_{0,0}|_{\varphi_1}(Y)=\theta_{0,0}(\hat{\pi}^1_*(Y)-\varphi_1\circ \theta_{-1}\circ \hat{\pi}^1_*(Y)-\varphi_1\circ\hat{\theta}^1_{0,2}|_{\varphi_1}(Y)-\varphi_1\circ\hat{\theta}^1_{0,-2}|_{\varphi_1}(Y)).
\end{align*}
When $Y\in T_{\varphi_1}^{0,0}\hat{P}^1$, each term vanishes except $\theta_{0,0}(\hat{\pi}^1_*(Y))$ as claimed. Showing the transformation property for $\hat{\theta}^1_{0,0}$ requires that of the other zero-graded components,
\begin{align*}
(R_{\mathbbm{1}+f})^*\hat{\theta}^1_{0,0}|_{\tilde{\varphi}_1}(Y)&=\theta_{0,0}\Big(\hat{\pi}^1_*(Y)-\tilde{\varphi}_1\circ \theta_{-1}\circ \hat{\pi}^1_*(Y)-\sum_{j=\pm2}\tilde{\varphi}_1\circ(R_{\mathbbm{1}+f})^*\hat{\theta}^1_{0,j}|_{\tilde{\varphi}_1}(Y)\Big)\\
&=\theta_{0,0}\Big(\hat{\pi}^1_*(Y)-(\varphi_1|_{\g_{-1}}+\varphi_1|_{\g_{0}}\circ f)\circ \theta_{-1}\circ \hat{\pi}^1_*(Y)\\
&-\sum_{j=\pm2}(\varphi_1|_{\g_{0,j}}+\varphi_1|_{\g_{0,0}}\circ f|_{\g_{0,j}})(\hat{\theta}^1_{0,j}-f\circ\hat{\theta}^1_{-1})|_{\varphi_1}(Y)\Big)\\
&=\hat{\theta}^1_{0,0}|_{\varphi_1}(Y)-\theta_{0,0}\Big(\varphi_1|_{\g_{0}}\circ f\circ \theta_{-1}\circ \hat{\pi}^1_*(Y)
-\sum_{j=\pm2}\varphi_1|_{\g_{0,j}}\circ f\circ\hat{\theta}^1_{-1}|_{\varphi_1}(Y)\Big)\\
&-\sum_{j=\pm2}f|_{\g_{0,j}}(\hat{\theta}^1_{0,j}-f\circ\hat{\theta}^1_{-1})|_{\varphi_1}(Y)\\
&=\hat{\theta}^1_{0,0}|_{\varphi_1}(Y)-\theta_{0,0}\Big(\varphi_1|_{\g_{0,0}}\circ f\circ\hat{\theta}^1_{-1}|_{\varphi_1}(Y)\Big)-\sum_{j=\pm2}f|_{\g_{0,j}}(\hat{\theta}^1_{0,j}-f\circ\hat{\theta}^1_{-1})|_{\varphi_1}(Y),
\end{align*}
so $\hat{\theta}^1_{0,0}$ transforms as indicated. Putting all of the zero-graded forms together, we obtain
\begin{align*}
&\hat{\theta}^1_0=\sum_{j=0,\pm2}\hat{\theta}^1_{0,j}:\mathbb{C}T^{-1}\hat{P}^1\to\g_0
\end{align*}
whose kernel when restricted to $\mathbb{C}T^0\hat{P}^1$ is the vertical bundle $\mathbb{C}T^1\hat{P}^1$.
\end{proof}

To these soldering forms we now add one which acts isomorphically on the vertical bundle $\mathbb{C}T^1\hat{P}^1$, as in \eqref{vert1}. This form is defined using vector fields $\zeta_f$ as in \eqref{P1zeta} in the same manner that $\theta_{0,0}$ was defined on $P^0$,
\begin{align*}
 \hat{\theta}^1_1:\mathbb{C}T^{1}\hat{P}^1\to\hat{\g}_1,
\end{align*}
by the linear extension of
\begin{align*}
\hat{\theta}^1_1(\zeta_f)=f.
\end{align*}
The transformation properties of the $\g^0$-valued soldering forms described in Proposition \ref{hatP1forms} now provide the following.

\begin{cor}\label{T1_brax}
Let $Z\in\Gamma(\mathbb{C}T^1\hat{P}^1)$ and $i=-2,-1$. For $Y\in\Gamma(\mathbb{C}T^i\hat{P}^1)$ such that $\hat{\theta}^1_{i}(Y)= y\in\g_i$ and $\hat{\theta}^1_{0}(Y)=0$ when $i=-1$,
\begin{eqnarray}
~&\hat{\theta}^1_{i}([Z,Y])=0,\nonumber \\
~&\hat{\theta}^1_{i+1}([Z,Y])=\hat{\theta}^1_1(Z)( y),\label{T1_ZY}
\end{eqnarray}
where the latter indicates the action of $\hat{\theta}^1_1(Z)\in\hat{\g}_1\subset\Hom(\g^0,\g^0)$ on $ y\in\g_-$. Furthermore, for $V\in\Gamma(\mathbb{C}T^{0}\hat{P}^1)$ such that $\hat{\theta}^1_{0}(V)= v\in\g_0$,
\begin{align}
\label{T1_ZV}
&\hat{\theta}^1_{0,\pm2}([Z,V])=0,
&\hat{\theta}^1_{0,0}([Z,V])=\hat{\theta}^1_1(Z)( v).
\end{align}
\end{cor}

\begin{proof}
The proof is analogous to that of Corollary \ref{equiv-reg}, so we will leave out some of the details this time. Once again, it suffices to consider the vector fields $Z=\zeta_f$ for $f\in\hat{\g}_1$ as defined in \eqref{P1zeta}. The definition of the Lie derivative is
\begin{align*}
 \mathcal{L}_{\zeta_f}\hat{\theta}^1_{i}=\left.\frac{d}{dt}\right|_{t=0}R_{\mathbbm{1}+tf}^*\hat{\theta}^1_{i},
\end{align*}
and Cartan's formula shows
\begin{align*}
 (\mathcal{L}_{\zeta_f}\hat{\theta}^1_{i})(Y)=-\hat{\theta}^1_{i}([\zeta_f,Y]).
\end{align*}
By Proposition \ref{hatP1forms}, when $i=-2$,
\begin{align*}
-\hat{\theta}^1_{-2}([\zeta_f,Y])&= \left.\frac{d}{dt}\right|_{t=0}R_{\mathbbm{1}+tf}^*\hat{\theta}^1_{-2}(Y)=\left.\frac{d}{dt}\right|_{t=0}\hat{\theta}^1_{-2}(Y)=\left.\frac{d}{dt}\right|_{t=0} y=0,
\end{align*}
while
\begin{align*}
-\hat{\theta}^1_{-1}([\zeta_f,Y])&= \left.\frac{d}{dt}\right|_{t=0}R_{\mathbbm{1}+tf}^*\hat{\theta}^1_{-1}(Y)\\
&=\left.\frac{d}{dt}\right|_{t=0}(\hat{\theta}^1_{-1}-tf\circ\hat{\theta}^1_{-2})(Y)\\
&=-f( y).
\end{align*}
Similarly, when $i=-1$ and $Y\in\Gamma(\mathbb{C}T^{-1}\hat{P}^1)$ as prescribed,
\begin{align*}
-\hat{\theta}^1_{-1}([\zeta_f,Y])&=\left.\frac{d}{dt}\right|_{t=0}(\hat{\theta}^1_{-1}-tf\circ\hat{\theta}^1_{-2})(Y)=-f\circ\hat{\theta}^1_{-2}(Y)=0,
\end{align*}
while
\begin{align*}
-\hat{\theta}^1_{0}([\zeta_f,Y])&=\left.\frac{d}{dt}\right|_{t=0}(\hat{\theta}^1_{0}-(tf-t^2f\circ  f)\circ\hat{\theta}^1_{-1}-tf\circ\hat{\theta}^1_{0})(Y) \\
&=-f( y).
\end{align*}
The latter equation also proves the result for $V\in\Gamma(\mathbb{C}T^0\hat{P}^1)$, as $f\circ\hat{\theta}^1_{0}$ only takes values in $\g_{0,0}$.
\end{proof}

The regularity conditions expressed in \eqref{g-reg}, \eqref{bigreg}, and Corollary \ref{equiv-reg} dictate how the soldering forms on $P^0$ relate the bracket of vector fields in $\Gamma(\mathbb{C}TP^0)$ to the Lie algebra bracket on $\g^0$. The soldering forms on $\hat{P}^1$ maintain this property when restricted to the subbundles of $\mathbb{C}T\hat{P}^1$ where they are the $\hat{\pi}^1$-pullbacks of the soldering forms on $P^0$, but on their newly extended domains their interaction with the bracket of vector fields is measured by a \emph{torsion tensor} associated to each $\varphi_1\in\hat{P}^1$.

\begin{lem}\label{torsion_1_def}
For $\varphi_1\in\hat{P}^1$, define the \emph{torsion tensor} $\tau_1\in\Hom(\g^0\otimes\g_{-1},\g_-)$ as follows. Let $i_2=-1$, $-2\leq i_1\leq 0$, and $ y_{\ell}\in\g_{i_\ell}$ for $\ell=1,2$. Take local vector fields $Y_{\ell}\in\Gamma(\mathbb{C}T^{i_\ell}\hat{P}^1)$ such that
\begin{align*}
&\hat{\theta}^1_{i_\ell}(Y_\ell)= y_\ell,
&&\hat{\theta}^1_{i_\ell+1}(Y_\ell)=0\hspace{2mm}(\text{if }i_\ell<0),
\end{align*}
and set
\begin{align*}
&\tau_1( y_1, y_2)=\hat{\theta}^1_{m}|_{\varphi_1}([Y_1,Y_2]),
&m=\min\{i_1,-1\}.
\end{align*}
Similarly, define $\tau_\pm\in \Hom(\g_{0,\pm2}\otimes\g_{-1,\mp1},\g_{0,\pm2})$ for $ y_1\in\g_{0,\pm2}$ and $ y_2\in\g_{-1,\mp1}$ by taking $Y_1,Y_2$ as above and setting
\begin{align*}
\tau_\pm( y_1, y_2)=\hat{\theta}^1_{0,\pm2}|_{\varphi_1}([Y_1,Y_2]).
\end{align*}
Then $\tau_1$ and $\tau_\pm$ are well-defined.
\end{lem}

\begin{proof}
We must confirm that the definition is independent of the choices of vector fields, so suppose $\tilde{Y}_\ell\in\Gamma(\mathbb{C}T^{i_\ell}\hat{P}^1)$ is an alternative choice for each $i_\ell$. If $i_\ell<0$,
\begin{align*}
&0= y_\ell- y_\ell=\hat{\theta}^1_{i_\ell}(\tilde{Y}_\ell-Y_\ell),
&0=\hat{\theta}^1_{i_\ell+1}(\tilde{Y}_\ell-Y_\ell),
\end{align*}
which together imply
\begin{align}
&\tilde{Y}_\ell=Y_\ell+Z_\ell,
&Z_\ell\in\Gamma(\mathbb{C}T^{i_\ell+2}\hat{P}^1).
\end{align}
Similarly, when $i_1=0$,
\begin{align}
&\tilde{Y}_1=Y_1+Z_1,
&Z_1\in\Gamma(\mathbb{C}T^1\hat{P}^1).
\end{align}
With these vector fields, the first component of torsion is given by
\begin{equation}\label{t1_wd}
\begin{aligned}
\tau_1( y_1, y_2)&=\hat{\theta}^1_{m}|_{\varphi_1}([\tilde{Y}_1,\tilde{Y}_2])\\
&=\hat{\theta}^1_{m}|_{\varphi_1}([Y_1,Y_2]+[Z_1,Y_2]+[Y_1,Z_2]+[Z_1,Z_2]).
\end{aligned}
\end{equation}
By Corollary \ref{T1_brax}, $[Z_1,Y_2]+[Y_1,Z_2]+[Z_1,Z_2]\in\ker\hat{\theta}^1_{m}$, so $\tau_1$ is well-defined. Using $\tilde{Y}_1\in\Gamma(T^{0,\pm2}\hat{P}^1)$ and $\tilde{Y}_2\in\Gamma(T^{-1,\mp1}\hat{P}^1)$ to define $\tau_\pm$ will also result in an expression of the form \eqref{t1_wd}, except in this case $m=(0,\pm2)$. The result will follow once again by Corollary \ref{T1_brax}, noting that $ y_2\in\g_{-1,\mp1}$ implies $\hat{\theta}^1_1(Z_1)( y_2)\in\g_{0,0}\oplus\g_{0,\mp2}$, which implies $[Z_1,Y_2]\in\ker\hat{\theta}^1_{0,\pm2}$. The rest is immediate from the Corollary.
\end{proof}

\begin{rem}\label{intt}
Let $y_1\in\g_{0,0}\oplus\g_{0,\pm2}$. If $y_2\in\g_{-1,\mp 1}$ (opposite signs) Proposition \ref{hatP1forms} and the regularity conditions given by \eqref{bigreg} and Corollary \ref{equiv-reg} together imply that $\tau_1(y_1,y_2)$ is the same for every $\varphi_1$ in a fiber of $\hat{\pi}^1:\hat{P}^1\to P^0$. Furthermore, integrability of the bundles $T^{-1,\pm1}\hat{P}^1$ implies that if $y_2\in\g_{-1,\pm1}$ (matching signs) then the only nontrivial component of $\tau_1(y_1,y_2)$ takes values in $\g_{-1,\pm1}$. 
\end{rem}

\begin{definition}\label{lacdiff} The modified Lie algebra cohomology differential $\partial_1:\hat{\g}_1\to\Hom(\g^0\otimes\g_{-1},\g^0)$ is 
\begin{align*}
\partial_1 f(y,z)=\left\{\begin{array}{lr}[f(y), z]+[ y,f(z)]-f([y, z]),& y\in\g_-, z\in\g_{-1};\\{}
[f(y),z]+\sum_{j=\pm1}\Big([y_{0,2j},f^{1,j}(z_{-1,-j})]-f^{1,j}([y_{0,2j},z_{-1,-j}])\Big),&y\in\g_{0},z\in\g_{-1},
\end{array}\right.
\end{align*}
where, in the latter formula, we have
\begin{align*}
&y=y_{0,-2}+y_{0,0}+y_{0,2},
&z=z_{-1,-1}+z_{-1,1},
&&y_{i,j}, z_{i,j}\in\g_{i,j};
&&f=f^{1,1}+f^{1,-1},
&&f^{1,\pm1}\in\hat{\g}_{1,\pm1}.
\end{align*}
We emphasize the following facts about $\partial_1f$:
\begin{enumerate}[label=\textbf{(\alph*)}]

\item\label{g-t} When $f\in\hat{\g}_{1,\pm1}$,  $y\in\g_{i_1,j_1}$, and $z\in \g_{i_2,j_2}$, where $(i_\ell,j_\ell)\in I(\g_-)$ for $\ell=1,2$, $\partial_1 f(y,z)$ is necessarily trivial in all but the following cases:
\begin{enumerate}[(1)]
\item $y\in\g_{-1,\mp1}$, $z\in\g_{-1,\pm1}$;

\item $y,z\in\g_{-1,\pm1}$, $f\in\g_{1,\mp1}$;

\item $y\in\g_{-2,0}$, $z\in\g_{-1,\pm1}$, $f\in \g_{1,\mp1}$.
\end{enumerate}

\item\label{g0t} When $y\in\g_0$, $\partial_1f(y,z)$ takes values in multiple bigraded components of $\g^0$; in particular, $[f( y), z]\in\g_{-1}$, in contrast to $[y_{0,\pm2},f^{1,\pm1}(z_{-1,\mp1})]-f^{1,\pm1}([y_{0,\pm2},z_{-1,\mp1}])\in\g_{0,\pm2}$.
\end{enumerate}
\end{definition}

\begin{prop}\label{torsion_variation_1}
Let $\varphi_1,\tilde{\varphi}_1$ in the same fiber of $\hat{P}^1\to P^0$ be related by $\tilde{\varphi}_1=R_{\mathbbm{1}+f}(\varphi_1)$ for $f\in\hat{\g}_1$ according to Proposition \ref{hatP1fibers}. The torsion tensor $\tilde{\tau}_1$  associated to $\tilde{\varphi}_1$ is related to $\tau_1$ of $\varphi_1$ by
\begin{align}
\label{torstan}
\tilde{\tau}_1( y_1, y_2)&=\tau_1( y_1, y_2)+[f( y_1), y_2]+[ y_1,f( y_2)]-f([ y_1, y_2]),
&& y_1\in\g_{-}, y_2\in\g_{-1},
\end{align}
for every case mentioned in Definition \ref{lacdiff}\ref{g-t}; moreover,
\begin{align}
\label{torstan0}
\tilde{\tau}_1( y_1, y_2)&=\tau_1( y_1, y_2)+[f( y_1), y_2],
&& y_1\in\g_0, y_2\in\g_{-1},
\end{align}
and the tensors $\tilde{\tau}_\pm$  associated to $\tilde{\varphi}_1$ are related to $\tau_\pm$ of $\varphi_1$ by
\begin{align}
\label{torsnew}
\tilde{\tau}_\pm( y_1, y_2)&=\tau_\pm( y_1, y_2)+[ y_1,f^{1,\pm1}( y_2)]-f^{1,\pm1}([  y_1, y_2]),
&& y_1\in\g_{0,\pm2}, y_2\in\g_{-1,\mp1}.
\end{align}
\end{prop}
\begin{proof}
Take local vector fields $Y_\ell\in\Gamma(\mathbb{C}T^{i_\ell}\hat{P}^1)$ in a neighborhood of $\varphi_1$ satisfying
\begin{align*}
&\hat{\theta}^1_{i_\ell}(Y_\ell)= y_\ell,
&&\hat{\theta}^1_{i_\ell+1}(Y_\ell)=0\ (\text{if }i_\ell<0),
\end{align*}
so that
\begin{align*}
&\tau_1( y_1, y_2)=\hat{\theta}^1_{m}|_{\varphi_1}([Y_1,Y_2]),
&m=\min\{i_1,-1\}.
\end{align*}
We will produce local vector fields $\tilde{Y}_\ell\in\Gamma(\mathbb{C}T^{i_\ell}\hat{P}^1)$ in a neighborhood of $\tilde{\varphi}_1$ with which to express $\tilde{\tau}_1( y_1, y_2)$. To declutter notation, we abbreviate the diffeomorphism $R_{\mathbbm{1}+f}$ as $R$. For $i_\ell<0$, take $Z_\ell\in\Gamma(\mathbb{C}T^{i_\ell+1}\hat{P}^1)$ in a neighborhood of $\varphi_1$ such that
\begin{align*}
\hat{\theta}^1_{i_\ell+1}(Z_\ell)=f( y_\ell)\in\g_{i_\ell+1},
\end{align*}
and set
\begin{align*}
\tilde{Y}_\ell=R_*Y_\ell+R_*Z_\ell
\end{align*}
in a neighborhood of $\tilde{\varphi}_1$. Applying the transformation properties described in Proposition \ref{hatP1forms}, first note that when $i_1=-2$,
\begin{align*}
\hat{\theta}^1_{-2}(\tilde{Y}_1)=R^*\hat{\theta}^1_{-2}(Y_1+Z_1)=\hat{\theta}^1_{-2}(Y_1)
=y_1,
\end{align*}
while
\begin{align*}
\hat{\theta}^1_{-1}(\tilde{Y}_1)=R^*\hat{\theta}^1_{-1}(Y_1+Z_1)
=(\hat{\theta}^1_{-1}-f\circ\hat{\theta}^1_{-2})(Y_1+Z_1)
=f( y_1)-f( y_1)
=0,
\end{align*}
and when $i_\ell=-1$,
\begin{align*}
\hat{\theta}^1_{-1}(\tilde{Y}_\ell)=R^*\hat{\theta}^1_{-1}(Y_\ell+Z_\ell)
=(\hat{\theta}^1_{-1}-f\circ\hat{\theta}^1_{-2})(Y_\ell)
= y_\ell,
\end{align*}
while
\begin{align*}
\hat{\theta}^1_{0}(\tilde{Y}_\ell)&=R^*\hat{\theta}^1_{0}(Y_\ell+Z_\ell)\\
&=(\hat{\theta}^1_{0}-(f-f\circ  f)\circ\hat{\theta}^1_{-1}-f\circ\hat{\theta}^1_{0})(Y_\ell+Z_\ell)\\
&=f( y_\ell)-(f-f\circ  f)( y_\ell)-f\circ f( y_\ell)\\
&=0.
\end{align*}
Thus we see that when $i_1<0$,
\begin{align*}
\tilde{\tau}_1( y_1, y_2)=\hat{\theta}^1_{i_1}|_{\tilde{\varphi}_1}([\tilde{Y}_1,\tilde{Y}_2]).
\end{align*}

To address the cases mentioned in Definition \ref{lacdiff}\ref{g-t}, suppose $y_2\in\g_{-1,1}$ which implies $Z_2\in\Gamma(T^{0,2}\hat{P}^1)$ (the instances where $y_2\in\g_{-1,-1}$ can be treated \emph{mutatis mutandis}). Cases (1) and (2) correspond to $i_1=-1$, so we have 
\begin{align*}
\tilde{\tau}_1( y_1, y_2)&=\hat{\theta}^1_{-1}|_{\tilde{\varphi}_1}([\tilde{Y}_1,\tilde{Y}_2])\\
&=R^*\hat{\theta}^1_{-1}|_{\tilde{\varphi}_1}([Y_1+Z_1,Y_2+Z_2])\\
&=(\hat{\theta}^1_{-1}-f\circ\hat{\theta}^1_{-2})|_{\varphi_1}([Y_1,Y_2]+[Z_1,Y_2]+[Y_1,Z_2]+[Z_1,Z_2]).
\end{align*}
Among these terms,
\begin{align*}
&[Z_1,Y_2]+[Y_1,Z_2]\in\Gamma(\mathbb{C}T^{-1}\hat{P}^1)=\ker\hat{\theta}^1_{-2},
&[Z_1,Z_2]\in\Gamma(\mathbb{C}T^{0}\hat{P}^1)=\ker\hat{\theta}^1_{-1}\subset\ker\hat{\theta}^1_{-2}.
\end{align*}
We first consider case (1), so assume $y_1\in\g_{-1,-1}$ and note that $Z_1\in\Gamma(T^{0,-2}\hat{P}^1)$. Here and in the sequel, we will use $\stackrel{\text{r}}{=}$ to indicate any equality that follows from one of the regularity conditions expressed in \eqref{g-reg}, \eqref{bigreg}, or Corollary \ref{equiv-reg}. For example, to continue our calculation we abbreviate $\theta_{0,\pm}=\theta_{0,0}+\theta_{0,\pm2}$ and write
\begin{align*}
\tilde{\tau}_1( y_1, y_2)
&=\tau_1( y_1, y_2)+\hat{\theta}^1_{-1}([Z_1,Y_2]+[Y_1,Z_2])-f\circ\hat{\theta}^1_{-2}([Y_1,Y_2])\\
&\stackrel{\text{r}}{=}\tau_1( y_1, y_2)+[\hat{\theta}^1_{0,-}(Z_1),\hat{\theta}^1_{-1,1}(Y_2)]+[\hat{\theta}^1_{-1,-1}(Y_1),\hat{\theta}^1_{0,+}(Z_2)]-f([\hat{\theta}^1_{-1}(Y_1),\hat{\theta}^1_{-1}(Y_2)])\\
&=\tau_1( y_1, y_2)+[ f( y_1), y_2]+[ y_1,f( y_2)]-f([ y_1, y_2]).
\end{align*}
Case (2) is almost the same; let $y_1\in\g_{-1,1}$ and suppose $f\in\g_{1,-1}$ whereby $Z_1,Z_2\in\Gamma(T^{0,0}\hat{P}^1)$ and we calculate
\begin{align*}
\tilde{\tau}_1( y_1, y_2)
&\stackrel{\text{r}}{=}\tau_1( y_1, y_2)+[\hat{\theta}^1_{0,0}(Z_1),\hat{\theta}^1_{-1,1}(Y_2)]+[\hat{\theta}^1_{-1,1}(Y_1),\hat{\theta}^1_{0,0}(Z_2)]-f([\hat{\theta}^1_{-1}(Y_1),\hat{\theta}^1_{-1}(Y_2)])\\
&=\tau_1( y_1, y_2)+[ f( y_1), y_2]+[ y_1,f( y_2)]-f([ y_1, y_2]).
\end{align*}
For Case (3), take $y_1\in\g_{-2,0}$ and assume $f\in\hat{\g}_{1,-1}$ so that and $Z_1\in\Gamma(T^{-1,-1}\hat{P}^1)$ and $Z_2\in\Gamma(T^{0,0}\hat{P}^1)$. Since $i_1=-2$,
\begin{align*}
\tilde{\tau}_1( y_1, y_2)&=\hat{\theta}^1_{-2}|_{\tilde{\varphi}_1}([\tilde{Y}_1,\tilde{Y}_2])\\
&=R^*\hat{\theta}^1_{-2}|_{\tilde{\varphi}_1}([Y_1+Z_1,Y_2+Z_2])\\
&=\hat{\theta}^1_{-2}|_{\varphi_1}([Y_1,Y_2]+[Z_1,Y_2]+[Y_1,Z_2]+[Z_1,Z_2])\\
&=\tau_1( y_1, y_2)+\hat{\theta}^1_{-2}|_{\varphi_1}([Z_1,Y_2]+[Y_1,Z_2]+[Z_1,Z_2]).
\end{align*}
$\hat{\theta}^1_{-2}$ vanishes on $[Z_1,Z_2]\in \Gamma(\mathbb{C}T^{-1}\hat{P}^1)$, and we are left with
\begin{align*}
\tilde{\tau}_1( y_1, y_2)&=\tau_1( y_1, y_2)+\hat{\theta}^1_{-2}|_{\varphi_1}([Z_1,Y_2]+[Y_1,Z_2])\\
&\stackrel{\text{r}}{=}\tau_1( y_1, y_2)+[\hat{\theta}^1_{-1}(Z_1),\hat{\theta}^1_{-1}(Y_2)]+[\hat{\theta}^1_{-2}(Y_1),\hat{\theta}^1_{0,0}(Z_2)]\\
&=\tau_1( y_1, y_2)+[ f( y_1), y_2]+[ y_1,f( y_2)].
\end{align*}
This completes the proof of equation \eqref{torstan}.

Moving on to \eqref{torstan0}, take $Z_1\in\Gamma(T^{0,0}\hat{P}^1)$ in a neighborhood $\varphi_1$ so that
\begin{align*}
\hat{\theta}^1_{0,0}(Z_1)=f( y_1),
\end{align*}
and define in a neighborhood of $\tilde{\varphi}_1$
\begin{align*}
\tilde{Y}_1=R_*Y_1+R_*Z_1,
\end{align*}
which shows
\begin{align*}
\hat{\theta}^1_{0}|_{\tilde{\varphi_1}}(\tilde{Y}_1)&=R^*\hat{\theta}^1_{0}|_{\tilde{\varphi_1}}(Y_1+Z_1)\\
&=(\hat{\theta}^1_{0}-(f-f\circ  f)\circ\hat{\theta}^1_{-1}-f\circ\hat{\theta}^1_{0})|_{\varphi_1}(Y_1+Z_1)\\
&= y_1+f( y_1)-f( y_1),
\end{align*}
since $f|_{\g_{0,0}}=0$ implies
\begin{align*}
f\circ\hat{\theta}^1_{0}=f\circ\hat{\theta}^1_{0,2}+f\circ\hat{\theta}^1_{0,-2}.
\end{align*}
Thus we see that
\begin{align*}
\tilde{\tau}_1( y_1, y_2)=\hat{\theta}^1_{-1}|_{\tilde{\varphi}_1}([\tilde{Y}_1,\tilde{Y}_2]),
\end{align*}
and we compute
\begin{align*}
\hat{\theta}^1_{-1}|_{\tilde{\varphi}_1}([\tilde{Y}_1,\tilde{Y}_2])&=R^*\hat{\theta}^1_{-1}|_{\tilde{\varphi}_1}([Y_1+Z_1,Y_2+Z_2])\\
&=(\hat{\theta}^1_{-1}-f\circ\hat{\theta}^1_{-2})|_{\varphi_1}([Y_1,Y_2]+[Z_1,Y_2]+[Y_1,Z_2]+[Z_1,Z_2])\\
&=\hat{\theta}^1_{-1}|_{\varphi_1}([Y_1,Y_2]+[Z_1,Y_2])\\
&\stackrel{\text{r}}{=}\hat{\theta}^1_{-1}|_{\varphi_1}([Y_1,Y_2])+[\hat{\theta}^1_{0,0}(Z_1),\hat{\theta}^1_{-1}(Y_2)]\\
&=\hat{\theta}^1_{-1}|_{\varphi_1}([Y_1,Y_2])+[ f( y_1), y_2].
\end{align*}
This proves \eqref{torstan0}. The proof of equation \eqref{torsnew} requires similar arguments and will be omitted.
\end{proof}

\begin{rem}
\label{esentcomp_rem}	
For $y_1\in \g_0$ and $y_2\in \g_{-1}$, Remark \ref{intt} and Proposition \ref{torsion_variation_1} together imply that the only components of $\tau_1(y_1, y_2)$ which are not constant on the fibers of $\hat\pi^1: \hat P^1\to P^0$ are those in $\g_{-1,j}$ when $y_1\in \g_{0,i}$, $y_2\in \g_{-1,j}$, where $j=\pm1$ and $ij=2$.
\end{rem} 

The normalization condition discussed in Remark \ref{Chainrem} is best expressed in relation to the differential $\partial_1$ introduced in Definition \ref{lacdiff}.

\begin{definition}\label{N1_P1_def}
Fix $\mathcal{N}_1\subset \Re(\Hom(\g^0\otimes\g_{-1},\g^0))$ such that $\mathbb{C}\mathcal{N}_1=\mathcal{N}_1\otimes_\mathbb{R}\mathbb{C}$ is a subspace complement to the image of $\hat{\g}_1$ under $\partial_1$,
\begin{align}
\label{firsdirect}
\Hom(\g^0\otimes\g_{-1},\g^0)=\partial_1(\hat{\g}_1)\oplus\mathbb{C}\mathcal{N}_1,
\end{align}
and note that $\tau_1,\tau_\pm$ lie in the left-hand side of \eqref{firsdirect} by trivially extending the domains of $\tau_\pm$ to all of $\g^0\otimes\g_{-1}$. We call $\mathcal{N}_1$ the \emph{first normalization condition}. 
\end{definition}

Define $\Re P^1\subset\Re\hat{P}^1$ to be the subbundle of those $\varphi_1 \in \hat{P}^1$ whose torsion tensor 
\begin{align}\label{tau}
\tau=\tau_1+\tau_-+\tau_+\in\Hom(\g^0\otimes\g_{-1},\g^0)
\end{align}
is \emph{normalized}; i.e, $\Re(\tau)\in\mathcal{N}_1$. Complexifying, $P^1\subset\hat{P}^1$ is the subbundle of frames whose torsion tensor lies in $\mathbb{C}\mathcal{N}_1$. Denote $\pi^1=\hat{\pi}^1|_{P^1}$.

The remainder of this section is dedicated to proving that $\ker\partial_1=\g_1$, the first bigraded prolongation of the symbol $\g^0$, as defined in section \ref{universalsec}.

Fix $\varphi\in P^0$ and take $\varphi_1,\tilde{\varphi}_1\in P^1$ in the fiber $(\pi^1)^{-1}(\varphi)$ related by $\tilde{\varphi}_1=\varphi_1+\varphi_1\circ f$ for $f\in\ker\partial_1$. The torsion tensors \eqref{tau} of $\varphi_1,\tilde{\varphi}_1$ will be denoted by $\tau$ and $\tilde{\tau}$, respectively. By Definition \ref{lacdiff}, Proposition \ref{torsion_variation_1} gives 
\begin{align}\label{t0}
&\tilde{\tau}(y_1,y_2)=\tau(y_1,y_2)+\partial_1f(y_1,y_2),
&y_1\in\g_0,y_2\in\g_{-1},
\end{align}
so each of the bigraded components of $\partial_1f(y_1,y_2)$ identified in Definition \ref{lacdiff}\ref{g0t} vanish separately, and in particular $[f(y_1),y_2]=0$ shows 
\begin{equation}
\label{deg1}
f|_{\g_0}=0,
\end{equation}	
hence $f$ is an endomorphism of $\g^0$ that has degree 1 with respect to the first weight. Thus, in contrast to the bundle $\hat \pi^1: \hat P^1\to P^0$, the bundle $\pi^1:P^1\to P^0$ is an affine bundle whose fibers have modeling vector space $\ker\partial_1\subset\hat{\g}_1$. 

By the observation in part \ref{g-t} of Definition \ref{lacdiff}, $\ker\partial_1$ consists of degree $1$ derivations of $\g^0$, whence
\begin{equation*}
\ker\partial_1\subset \widetilde\g_1\cap\hat\g_1= \widetilde\g_{1,-1}\oplus\widetilde\g_{1,1},
\end{equation*}
where $\widetilde\g_1$ is the first standard algebraic Tanaka prolongation of $\g^0$ as in \eqref{firstalgTanaka}. Furthermore, \eqref{t0} shows
$$[ y_1,f( y_2)]-f([y_1, y_2])=0, \quad y_1\in\g_{0,\pm2},\quad y_2\in\g_{-1,\mp1},\quad f\in \ker\partial_1\cap \widetilde \g_{1, \pm 1},$$
and by definition of the Lie brackets between elements of nonnegative degree in the standard Tanaka prolongation, this is equivalent to 
\begin{equation*}
[f, \g_{0, \pm 2}]=0,\quad \forall f\in \ker\partial_1\cap \widetilde \g_{1, \pm 1}.  
\end{equation*}
Hence, if we set
$$\mathfrak u_{1,-1}:=\{f\in \widetilde \g_{1, -1} \ |\  [f, \g_{0, -2}]=0\},\quad \mathfrak u_{1,1}:=\{f\in \widetilde \g_{1, 1}\ |\  [f, \g_{0, 2}]=0\},$$
then 
\begin{equation}
\label{g_1prelim}
\ker\partial_1=\mathfrak u_{1, -1}\oplus \mathfrak u_{1,1}.
\end{equation}
Now \eqref{firstalgbigrad1}-\eqref{firstalgbigrad2} say that in order to prove $\mathfrak u_{1,\pm 1}=\g_{1,\pm1}$, it remains to verify 
\begin{equation}
\label{g_1goal}
[\g_{0, \pm 2},\mathfrak u_{1, \mp 1}]\subset \mathfrak u_{1, \pm 1} .
\end{equation}

Before we can proceed, we must discuss a geometric consequence of \eqref{deg1} for the fibration $\pi:P^0\to M$. Given a subbundle  $\Delta\subset\mathbb{C}TM$, a subbundle $\widehat \Delta\subset\mathbb{C}TP^0$ is called a \emph{lift of $\Delta$ to $P^0$} if it has the same rank as $\Delta$ and $\pi_*\widehat\Delta=\Delta$. Equation \eqref{deg1} tells us that the subspace 
$$L^{0,2}_\varphi:=\varphi_1(\g_{0,2})\subset T^{0,2}P^0$$ 
is independent of the choice of $\varphi_1$ in the fiber. Furthermore, $\pi_* (L^{0,2}_\varphi)=K_{\pi(\varphi)}$, so $L^{0,2}$ is a lift of the Levi kernel $K$ which we call the \emph{canonical lift of the Levi kernel $K$ to $P^0$ subordinated to the normalization condition $\mathcal N_1$}. In the same way, one can also define the canonical lift $L^{0, -2}\subset T^{0,-2}P^0$ of $\overline{K}$ to $P^0$, subordinated to $\mathcal N_1$. 

\begin{rem}
\label{uninormrem}
 Let $$\mathrm{pr}: \Hom(\g^0\otimes\g_{-1},\g^0)\rightarrow \bigoplus _{ j=\pm1, ij=2}\mathrm{Hom}(\g_{0,i}\otimes\g_{-1,j},\g_{-1,j})$$ be the natural  projection. By Remark \eqref{esentcomp_rem}, $\mathrm{pr}(\tau_1)$ is the only nontrivial (i.e., independent of the CR symbol) component of the restriction of  
$\tau_1$ to $\g_0\otimes \g_{-1}$. Thus, the canonical lifts of $K$ and $\overline{K}$ to $P^0$ subordinated to $\mathcal N_1$ are in fact determined by $\mathrm{pr}(\mathcal{N}_1)$. Note: it is not necessarily true that any lift of $K$ is a canonical lift subordinated to some normalization condition. However, if for every $\varphi_1\in (\hat \pi_1)^{-1}(\varphi)$ satisfying 
\begin{equation}
\label{phi1L2}
\varphi_1(\g_{0,\pm 2})=L^{0, \pm 2}_\varphi,
\end{equation}
the condition 
\begin{equation}
\label{univnorm}
\mathrm{pr}(\tau_1|_{\varphi_1})=0
\end{equation}
holds, 
then $L^{0, 2}_\varphi$ and $L^{0, -2}_\varphi$ are the canonical lifts of $K$ and $\overline{K}$, respectively, subordinated to an arbitrary chosen  normalization condition. 
\end{rem}	
\begin{rem} 
	\label{splitrem}
The splitting  
	\begin{equation}
	\label{spliteq}
\mathbb{C}T^0P^0=L^{0,-2}\oplus T^{0,0}P^0\oplus L^{0,2}		
	\end{equation}
allows for a canonical extension of the form $\theta_{0,0}$ from $T^{0,0}P^0$ to all of $\mathbb{C}T^0P^0$ by composing it with the linear projection onto $T^{0,0}P^0$, whence 
	$$\theta_0:=\theta_{0,-2}+\theta_{0,0}+\theta_{0,2}$$
	is a  $\g_0$-valued form acting isomorphically on all of $\mathbb{C}T^0P^0$.
\end{rem}

 Let $\iota_1:P^1\hookrightarrow\hat{P}^1$ be the inclusion map, and denote by $\theta_i^1$ and $\theta_{i,j}^1$ the forms on $P^1$, obtained by the pull-back with respect to $\iota_1$ of the corresponding forms $\hat \theta_i^1$ and $\hat \theta_{i,j}^1$ on $\hat P^1$. The splitting \eqref {g_1prelim} %depicted in Proposition \ref{hatP1fibers} 
determines bigraded components 
\begin{align*}
{\theta}^1_{1,\pm1}:T^{1,\pm1}P^1\to{\mathfrak u}_{1,\pm1},
\end{align*}
where the subbundle $T^{1,\pm1}P^1\subset\mathbb{C}T^1P^1$ is trivialized by $\zeta_f$ for $f\in\mathfrak u_{1,\pm1}$. From \eqref{deg1} and  relations \eqref{T1_ZV} in Corollary \ref{T1_brax} it follows that if $Z\in\Gamma(\mathbb{C}T^1\hat{P}^1)$ and  $V\in\Gamma(\mathbb{C}T^{0}\hat{P}^1)$ such that $\hat{\theta}^1_{0}(V)= v\in\g_0$, then 
$$\theta^1_{0}([Z,V])=0,$$
which shows $[Z, V]\in \Gamma(\mathbb{C}T^1\hat{P}^1)$; i.e., $\theta_1^1([Z, V])$ is well-defined.

To prove \eqref{g_1goal}, we now note that the space $\ker\partial_1$ depends on the CR symbol $\g^0$ only. In other words, it is independent of the choice of a $2$-nondegenerate CR structure with symbol $\g^0$ for which one implements the first prolongation, and of the choice of the first normalization condition. We may as well take the flat, 2-nondegenerate CR structure with symbol $\g^0$ as described in the last paragraph of section 2. In this case, the bundle $\Re P^0$ can be identified with the Lie group $\Re G^0$. Let $L^{0,\pm 2}$ be the left invariant distributions on $G^0$ obtained by left translations of $\g_{0,\pm 2}$. Then by definition of $\g_{0,\pm 2}$, \eqref{univnorm} holds for all $\varphi_1$ satisfying \eqref{phi1L2}. Hence, by Remark \ref{uninormrem} the spaces $L^{0,2}$  and $ L^{0,-2}$ are canonical lifts of $K$ and $\overline{K}$, respectively, subordinated to some normalization condition.

\begin{rem}
\label{Maurerrem}
By construction, the form $\theta_{0}$, defined on $\mathbb{C}T^0P^0$ as in Remark \ref{splitrem} coincides with the restriction of the Maurer-Cartan form of $G^0$ to $\mathbb{C}T^0P^0$.
\end{rem}

Now, in addition to Corollary \ref{T1_brax} we have the following:
\begin{cor}\label{T1_brax_add}
Let $M$ be the flat, 2-nondegenerate CR structure with symbol $\g^0$ as described in the last paragraph of section 2. Let $Z\in\Gamma(\mathbb{C}T^1{P}^1)$ and $i=-2,-1$. Then for $V\in\Gamma(\mathbb{C}T^{0}{P}^1)$ such that ${\theta}^1_{0}(V)= v\in\g_0$ and $Y\in\Gamma(\mathbb{C}T^i{P}^1)$ such that ${\theta}^1_{i}(Y)= y\in\g_i$, we have 
\begin{align}\label{bigreg1}
&{\theta}^1_{1,\pm1}([Z,V])(y)=[{\theta}^1_{1,\mp1}(Z)(y),v]+
{\theta}^1_{1,\mp1}(Z)([v,y])
&\text{when}
&&v\in\g_{0,\pm2},
&&y\in\g_{-1,\mp1}.
\end{align}
\end{cor}
\begin{proof}  Once again, it suffices to consider vector fields $Z=\zeta_f$ for $f\in\hat{\g}_1$. Let $v\in\g_{0,2}$, $y\in\g_{-1,-1}$, and $f\in\mathfrak u_{1,-1}$. Using relation \eqref{T1_ZY} from  Corollary \ref{T1_brax} and the Jacobi identity, we have 
\begin{align*}
\theta^1_1([\zeta_f,V])(y)={\theta}^1_0([[\zeta_f,V],Y])={\theta}^1_0([[\zeta_f,Y],V])+{\theta}^1_0([\zeta_f,[V,Y]])
\end{align*}
Remark \ref{Maurerrem} and relation \eqref{T1_ZY} imply ${\theta}^1_0([[\zeta_f,Y],V])=[f(y),v]$. With the biweights of all vectors in mind, \eqref{bigreg} shows 
\begin{align*}
f([v,y])={\theta}^1_{1,-1}(\zeta_f)({\theta}^1_{-1,1}([V,Y]))=\hat{\theta}^1_{0,0}([\zeta_f,[V,Y]]),
\end{align*}
so the fact that the left-hand side of \eqref{bigreg1} has biweight $(0,0)$ proves the result. The arguments for $v\in\g_{0,-2}$, $y\in\g_{-1,1}$, and $f\in\mathfrak u_{1,1}$ are the same.
\end{proof}

At last, we are ready to prove \eqref{g_1goal}. For $f\in \mathfrak u_{1,\mp 1}$, substituting $Z=\zeta_f$ into \eqref{bigreg1} yields
\begin{align*}
&{\theta}^1_{1,\pm1}([\zeta_f,V])(y)=[f(y),v]+
f([v,y])=[f, v](y)
&\text{when}
&&v\in\g_{0,\pm2},
&&y\in\g_{-1,\mp1}.
\end{align*}	
In order to get the last equality, we used the Jacobi identity in the standard universal Tanaka prolongation $\mathfrak U(\g^0)$. From the nondegeneracy condition for $\mathfrak U(\g^0)$ and the fact that $\g_-$ is generated by $\g_{-1}$ it follows that $[f,v]={\theta}^1_{1,\pm1}([\zeta_f,V])$ and in particular that $[f, v]\in \mathfrak u_{1, \pm 1}$, which completes the proof of \eqref{g_1goal} and therefore the proof that $\ker\partial_1=\g_1$.

\subsubsection{Higher geometric prolongations}

As mentioned in the sentence before formula \eqref{arbalgTanaka}, for $\kappa\geq 2 $ the space $\g_\kappa=\bigoplus_{j\in\mathbb Z}\g_{\kappa,j}$ of all elements in $\mathfrak U_{\text{bigrad}}(\g^0)$ of first weight $\kappa$ is exactly the same as the degree-$\kappa$ component of the standard Tanaka algebraic prolongation of $\g_-\oplus \g_0 \oplus \g_1$. With the bundle $P^1$ having been constructed in the previous subsection, the construction of remaining bundles $P^\kappa$ in the chain \eqref{bundles1} is also the same as in the standard Tanaka theory, and we will not repeat it, instead referring the reader to the Tanaka's original paper \cite{tanaka1} or to the second author's \cite{zeltan}. For completeness, we briefly describe the nature of the bundles $P^\kappa$ and indicate how their construction ultimately produces an absolute parallelism.

We proceed by induction. Fix $\kappa\geq2$ and suppose we have a bundle $\pi^{\kappa-1}:P^{\kappa-1}\to P^{\kappa-2}$ with graded and bigraded filters
\begin{align*}
&\mathbb{C}T^iP^{\kappa-1},
&T^{i,j}P^{\kappa-1},
&&(i,j)\in I(\g^{\kappa-1}),
\end{align*}
and in particular
\begin{align*}
\mathbb{C}T^{\kappa-1} P^{\kappa-1}=\ker\pi^{\kappa-1}_*.
\end{align*}
Suppose further that there are soldering forms
\begin{align*}
&\theta^{\kappa-1}_{i}:\mathbb{C}T^{-2}P^{\kappa-1}\to\g_i,
&&(-2\leq i\leq \kappa-3);\\
&\theta^{\kappa-1}_{\kappa-2}:\mathbb{C}T^{-1}P^{\kappa-1}\to\g_{\kappa-2},
&&\theta^{\kappa-1}_{\kappa-1}:\mathbb{C}T^{\kappa-1} P^{\kappa-1}\to\g_{\kappa-1},
\end{align*}
such that $\theta^{\kappa-1}_{\kappa-1}$ is an isomorphism and the others restrict and descend to isomorphisms
\begin{align*}\tag{$-2\leq i\leq\kappa-2$}
\theta^{\kappa-1}_i|_{\mathbb{C}T^{i}P^{\kappa-1}}:\mathbb{C}T^{i}P^{\kappa-1}/\mathbb{C}T^{i+1}P^{\kappa-1}\to\g_{i}.
\end{align*}

Define the bundle $\hat{\pi}^{\kappa}:\hat{P}^{\kappa}\to P^{\kappa-1}$ whose fiber over $\varphi\in P^{\kappa-1}$ is composed of maps
\begin{align*}
\varphi_\kappa\in\Hom(\g_{-2},\mathbb{C}T^{-2}_{\varphi}P^{\kappa-1}/\mathbb{C}T^{\kappa-1}_{\varphi}P^0)\oplus\bigoplus_{(i,j)\in I(\g_{-1}\oplus\g_0\oplus\dots\oplus\g_{\kappa-1})}\Hom(\g_{i,j},T^{i,j}_{\varphi}P^{\kappa-1}),
\end{align*}
which satisfy
\begin{align*}
&\theta^{\kappa-1}_{i}\circ\varphi_\kappa|_{\g_{i}}=\mathbbm{1}_{\g_{i}};
&&\g_i\subset\g^{\kappa-1},\\
&\theta^{\kappa-1}_{i_1+i_2}\circ\varphi_\kappa|_{\g_{i_1}}=0,
&&-2\leq i_1\leq \kappa-3, &&1\leq i_2\leq \min\{\kappa-1,\kappa-2-i_1\}
\end{align*}
where $\mathbbm{1}_{\g_{i}}$ is the identity map on $\g_{i}$ The subbundle $\Re\hat{P}^\kappa\subset\hat{P}^\kappa$ is that of frames $\varphi_\kappa$ which additionally satisfy $\varphi_\kappa(\overline{ y})=\overline{\varphi_\kappa( y)}$
for $ y\in\g^{\kappa-1}$.

The fibers of $\hat{P}^\kappa$ are affine spaces whose modeling vector space $\hat{\g}_\kappa\subset\Hom(\g^{\kappa-1},\g^{\kappa-1})$ is that of linear maps $f$ given by sums of bigraded maps
\begin{align*}
&f^{m,j_2}_{i_1,j_1}:\g_{i_1,j_1}\to\g_{i_1+m,j_1+j_2},&(i_1,j_1),(i_1+m,j_1+j_2)\in I(\g^{\kappa-1}),
&&m=\min\{\kappa,\kappa-1-i_1\},
\end{align*}
which vanish on $\g_{\kappa-1}$. The fibers of $\Re\hat{P}^\kappa$ are modeled by those $f\in\Re\hat{\g}_\kappa$ with $\overline{f}=f$ as defined by \eqref{conjext}.

The bundle $\hat{\pi}^\kappa:\hat{P}^\kappa\to P^{\kappa-1}$ inherits a filtration as before,
\begin{align*}
&\mathbb{C}T^i\hat{P}^\kappa=(\hat{\pi}^\kappa_*)^{-1}(\mathbb{C}T^iP^{\kappa-1}),
&T^{i,j}\hat{P}^\kappa=(\hat{\pi}^\kappa_*)^{-1}(T^{i,j}P^{\kappa-1}),
&&(i,j)\in I(\g^{\kappa-1}),
\end{align*}
where each subbundle contains the vertical bundle
\begin{equation*}
\mathbb{C}T^\kappa\hat{P}^\kappa=\ker\hat{\pi}^\kappa_*.
\end{equation*}
Mimicking the proof of Proposition \ref{hatP1forms} and the constructions of $\theta_{0,0}$ and $\hat{\theta}^1_1$, we can show that there exist intrinsically defined forms
\begin{equation}\label{thetak}
\begin{aligned}
\hat{\theta}^{\kappa}_{i}&:\mathbb{C}T^{-2}\hat{P}^{\kappa}\to\g_i,
&&(-2\leq i\leq \kappa-2);\\
\hat{\theta}^{\kappa}_{\kappa-1}&:\mathbb{C}T^{-1}\hat{P}^{\kappa}\to\g_{\kappa-1},
&&\hat{\theta}^{\kappa}_{\kappa}:\mathbb{C}T^{\kappa}\hat{P}^{\kappa}\to\hat{\g}_{\kappa},
\end{aligned}
\end{equation}
the first $\kappa+2$ of which restrict to give the pullbacks of the soldering forms on $P^{\kappa-1}$ while the last is an isomorphism.

One can define the torsion tensor $\tau_\kappa\in\Hom(\g^{\kappa-2}\wedge\g_{-1},\g^{\kappa-2})$ associated to $\varphi_\kappa\in\hat{P}^\kappa$ by analogy with $\tau_i$, $i<\kappa$. Note that no analog of $\tau_\pm$ from Lemma \ref{torsion_1_def} appears in higher prolongations. A \emph{normalization condition for the $\kappa^{\text{th}}$ geometric prolongation} is choice of a subspace $\mathcal{N}_\kappa\subset \Re(\Hom(\Lambda^2\g^{\kappa-1},\g^{\kappa-1}))$ whose complexification is complementary to the image of $\hat{\g}_\kappa$ under
\begin{align*}\tag{as defined in \cite{zeltan}}
\partial_\kappa:\Hom(\g^{\kappa-1},\g^{\kappa-1})\to\Hom(\g^{\kappa-1}\wedge\g_{-1},\g^{\kappa-1}),
\end{align*}
so that
\begin{align}
	\label{kdirect}
\Hom(\Lambda^2\g^{\kappa-1},\g^{\kappa-1})=\partial_\kappa(\hat{\g}_\kappa)\oplus\mathbb{C}\mathcal{N}_\kappa.
\end{align}

%\IZi{Do not we need to have $\Hom(\g^{\kappa-1}\wedge\g_{-1},\g^{\kappa-1})$ in the left-hand side here? \tb{-- No, it is correct as written. The superscript notation means the entire algebra up to that graded component, so that $\g^0=\g_-\oplus\g_0$ and $\g^1=\g^0\oplus\g_1$, etc. So in particular, $\g^{\kappa-1}\wedge\g_{-1}\subset\Lambda^2\g^{\kappa-1}$.} }

%\IZi{I agree, but you have the strict inclusion , so why you need to choose the complement in the this larger space $\Hom(\Lambda^2\g^{\kappa-1},\g^{\kappa-1})$ if you can work with a smaller one $\Hom(\g^{\kappa-1}\wedge\g_{-1},\g^{\kappa-1})$} 

Once $\mathcal{N}_\kappa$ is chosen, $\Re P^1\subset\Re\hat{P}^1$ is the subbundle of those $\varphi_\kappa$ whose torsion tensors are \emph{normalized}; i.e., $\tau_\kappa\in\mathcal{N}_\kappa$, while $P^\kappa\subset\hat{P}^\kappa$ is the subbundle whose torsion tensors lie in $\mathbb{C}\mathcal{N}_\kappa$. Set $\pi^\kappa=\hat{\pi}^\kappa|_{P^\kappa}$, pull back $\hat{\theta}^\kappa$ along the inclusion $P^\kappa\hookrightarrow\hat{P}^\kappa$ to get $\theta^\kappa$ on $P^\kappa$, and iterate.

Note that the fibers of $\pi^\kappa:P^\kappa\to P^{\kappa-1}$ are isomorphic to $\g_\kappa$. Hence, the fibers are trivial for $\kappa\geq l+1$, and $\pi^\kappa$ is a diffeomorphism. By \eqref{thetak}, all of the nontrivial soldering forms $\theta^\kappa_i$ for $i\leq l$ are true one forms -- defined on all of $\mathbb{C}TP^\kappa$ -- as of $\kappa=l+2$. Thus, they may be assembled into a $\mathfrak{U}_{\text{bigrad}}(\g^0)$-valued parallelism on $P^{l+2}$, or by restriction, a $\Re\mathfrak{U}_{\text{bigrad}}(\g^0)$-valued parallelism on $\Re P^{l+2}$.

\subsection{Sketch of the proof of part (2) of Theorem \ref{maintheor}}
Consider the manifold $M_0$ with the flat CR structure of type $\g^0$ as described in the paragraph before Definition \ref{flatdef}.
In this section we show that the Lie algebra $\mathcal A$ of (germs of) infinitesimal symmetries of $M_0$ is isomorphic to $\Re\g$ where $\g=\mathfrak{U}_{\text{bigrad}}(\g^0)$. We hew to the arguments and notation of \cite[section 6]{tanaka1} (see also \cite[subsections 2.2-2.3]{yamag}), emphasizing those aspects of the construction that require modification.

Recall that $G^0$ and its closed subgroup $G_{0,0}$ have Lie algebras $\g^0$ and $\g_{0,0}$, respectively, whose vectors we interpret as left-invariant vector fields on $G^0$. Name the homogeneous spaces
\begin{align*}
&M_0^\mathbb{C}=G^0/G_{0,0},
&M_0=\Re G^0/\Re G_{0,0},
\end{align*}
and observe $TM_0^\mathbb{C}\cong \mathbb{C}TM_0$ has tangent spaces $\m$ as in \eqref{m}. Via the projection $\pi:G^0\to M_0^\mathbb{C}$, we have subbundles,
\begin{align*}
&D=\pi_*\Re(\g_{-1}\oplus\g_0)\subset TM_0,
&H,\overline{H}=\pi_*(\g_{-1,\pm 1}\oplus\g_{0,\pm 2})\subset\mathbb{C}TM_0,
&&K,\overline{K}=\pi_*(\g_{0,\pm 2})\subset\mathbb{C}TM_0,
\end{align*}
and we reiterate that $K$ is exactly the subbundle of $H$ satisfying $[K,\overline{H}]\subset H\oplus\overline{H}$.

Let $\xi\in\Omega^1(G^0,\g^0)$ be the left-invariant Maurer-Cartan form of $G^0$ which satisfies the equations
\begin{align}\label{MC}
&\dd\xi_{i,j}(\hat{X},\hat{Y})=-\sum_{\begin{smallmatrix}i_1+i_2=i\\j_1+j_2=j\end{smallmatrix}}[\xi_{i_1,j_1}(\hat{X}),\xi_{i_2,j_2}(\hat{Y})],&\hat{X},\hat{Y}\in\Gamma(TG^0).
\end{align}
Here and in what follows, we implicitly use the convention that a quantity indexed by $(i,j)\notin (\g^0)$ -- and later, $(i,j)\notin I(\g)$ -- is trivial. Note that $\pi_*:TG^0\to TM_0^{\mathbb{C}}$ is a linear isomorphism on each fiber of the subbundle $\ker\xi_{0,0}\subset TG^0$, so for any $X\in\Gamma(\mathbb{C}TM_0)$ there is a unique $\hat{X}\in\Gamma(TG^0)$ such that $\xi_{0,0}(\hat{X})=0$ and $\pi_*\hat{X}=X$. In particular, the value of $\hat{f}_X=\xi(\hat{X})\in\m$ is constant on the fibers of $\pi$, hence descends to a well-defined function
\begin{align*}
f_X:M_0\to\m,
\end{align*}
whose bigraded components are denoted $f^{i,j}_X$. For $X,Y\in\Gamma(\mathbb{C}TM_0)$, \eqref{MC} implies
\begin{align*}
\xi([\hat{X},\hat{Y}])&=\hat{X}\xi(\hat{Y})-\hat{Y}\xi(\hat{X})-\dd\xi(\hat{X},\hat{Y})\\
%&=\hat{X}\pi^*f_Y-\hat{Y}\pi^*f_X+[f_X,f_Y]\\
&=\dd f_Y(X)-\dd f_X(Y)+[f_X,f_Y],
\end{align*}
and if $\xi_{i,j}(\hat{Y})=\xi_{i,j}([\hat{X},\hat{Y}])=0$, then we can say
\begin{align*}
&\dd f^{i,j}_X(Y)=\sum_{\begin{smallmatrix}i_1+i_2=i\\j_1+j_2=j\end{smallmatrix}}[f^{i_1,j_1}_X,f^{i_2,j_2}_Y],
&(i,j)\in I(\m).
\end{align*}

The algebra $\mathcal{A}$ describes the sheaf of \emph{infinitesimal CR symmetries} of $M_0$; i.e., vector fields $X\in\Gamma(TM_0)$ which satisfy
\begin{align*}
&[X,D]\subset D,
&[X,H]\subset H,
&&(\Rightarrow [X,\overline{H}]\subset\overline{H}).
\end{align*}
The Jacobi identity shows $[X,[K,\overline{H}]]\subset H\oplus\overline{H}$, whence infinitesimal CR symmetries additionally satisfy $[X,K]\subset K$ and its conjugate. For a symmetry $X\in\mathcal{A}$, we have the following instances where it is necessarily true that $\xi_{i,j}(\hat{Y})=\xi_{i,j}([\hat{X},\hat{Y}])=0$:
\begin{itemize}
	\item $Y\in H\oplus\overline{H}$ and $(i,j)=(-2,0)$,
	
	\item $Y\in H$ and $(i,j)\neq(-1,1),(0,2)$,
	
	\item $Y\in\overline{H}$ and $(i,j)\neq(-1,-1),(0,-2)$,
	
	\item $Y\in K$ and $(i,j)\neq(0,2)$,
	
	\item $Y\in \overline{K}$ and $(i,j)\neq(0,-2)$.
\end{itemize}
Consequently, for any of these $Y$ we have
\begin{align*}
\dd f^{-2,0}_X(Y)&=[f_X^{-1,1},f^{-1,-1}_Y]+[f_X^{-1,-1},f^{-1,1}_Y],
\end{align*}
or more generally
\begin{align}\label{df-2}
\dd \hat{f}^{-2,0}_X\equiv[\hat{f}_X^{-1,1},\xi_{-1,-1}]+[\hat{f}_X^{-1,-1},\xi_{-1,1}]\mod\{\xi_{-2,0}\}.
\end{align}
Likewise,
\begin{equation}\label{df-10}
\begin{aligned}
\dd \hat{f}^{-1,1}_X&\equiv[\hat{f}_X^{0,2},\xi_{-1,-1}]+[\hat{f}_X^{-1,-1},\xi_{0,2}]&\mod\{\xi_{-2,0},\xi_{-1,1}\},\\
\dd \hat{f}^{0,2}_X&\equiv0&\mod\{\xi_{-2,0},\xi_{-1,1},\xi_{0,2}\},
\end{aligned}
\end{equation}
with the corresponding statements for $\dd \hat{f}^{-1,-1}_X,\dd \hat{f}^{0,-2}_X$ given by changing the signs of the second indices.

Define $f_X^{0,0}:M_0\to\m\otimes\m^*$ to be the restriction of $\dd f^{i,j}_X$ to those $Y\in\Gamma(\mathbb{C}TM_0)$ that take constant values $\xi(\hat{Y})\in\g_{i,j}$ for $(i,j)\in I(\m)$, and set $\hat{f}^{0,0}_X=\pi^*f^{0,0}_X$. Now \eqref{df-2} becomes an equality
\begin{align*}
\dd \hat{f}^{-2,0}_X=[\hat{f}_X^{0,0},\xi_{-2,0}]+[\hat{f}_X^{-1,1},\xi_{-1,-1}]+[\hat{f}_X^{-1,-1},\xi_{-1,1}],
\end{align*}
and we strengthen \eqref{df-10},
\begin{equation}\label{2df-10}
\begin{aligned}
\dd \hat{f}^{-1,1}_X&\equiv[\hat{f}_X^{0,2},\xi_{-1,-1}]+[\hat{f}_X^{0,0},\xi_{-1,1}]+[\hat{f}_X^{-1,-1},\xi_{0,2}]&\mod\{\xi_{-2,0}\},\\
\dd \hat{f}^{0,2}_X&\equiv[\hat{f}_X^{0,0},\xi_{0,2}]&\mod\{\xi_{-2,0},\xi_{-1,1}\}.
\end{aligned}
\end{equation}
That $f^{0,0}_X$ takes values in $\mathfrak{der}(\g_-)$ is straightforward to confirm; e.g., for $Y_1\in\g_{-1,1}$ and $Y_2\in\g_{-1,-1}$, we compute
\begin{align*}
f_X^{0,0}[Y_1,Y_2]%&=[f_X^{0,0},\xi_{-2,0}([Y_1,Y_2])]\\
&=\dd f^{-2,0}_X([Y_1,Y_2])\\
&=Y_1\dd f_X^{-1,-1}(Y_2)-Y_2\dd f_X^{-1,1}(Y_1)\\
%&=[Y_1,[f_X^{0,0},\xi_{-1,-1}(Y_2)]]-[Y_2,[f_X^{0,0},\xi_{-1,1}(Y_1)]]\\
&=[Y_1,f_X^{0,0}(Y_2)]+[f_X^{0,0}(Y_1),Y_2].
\end{align*}
Furthermore, $[f^{0,0}_X,\g_{0,\pm2}]\subset\g_{0,\pm2}$ holds by virtue of $\dd f^{0,\pm2}_X\in\Omega^1(M_0,\g_{0,\pm2})$, so we see that $f^{0,0}_X$ takes values in $\g_{0,0}$.

In direct analogy to \cite{tanaka1} and \cite{yamag}, higher-order derivatives $f^{i,j}_X:M_0\to\g_{i,j}$ ($i\geq1$) of $f_X$ (and their lifts $\hat{f}^{i,j}_X=\pi^*f^{i,j}_X$) are defined so that we arrive at the fully determined equations
\begin{align}\label{df}
&\dd \hat{f}^{i,j}_X=\sum_{\begin{smallmatrix}i_1+i_2=i\\j_1+j_2=j\end{smallmatrix}}[\hat{f}^{i_1,j_1}_X,\xi_{i_2,j_2}]
&(i,j)\in I(\g),
&&(i_2,j_2)\in I(\m).
\end{align}
In particular, \eqref{2df-10} becomes
\begin{equation}\label{3df-10}
\begin{aligned}
\dd \hat{f}^{-1,1}_X&=[\hat{f}_X^{1,1},\xi_{-2,0}]+[\hat{f}_X^{0,2},\xi_{-1,-1}]+[\hat{f}_X^{0,0},\xi_{-1,1}]+[\hat{f}_X^{-1,-1},\xi_{0,2}],\\
\dd \hat{f}^{0,2}_X&=[\hat{f}_X^{2,2},\xi_{-2,0}]+[\hat{f}_X^{1,1},\xi_{-1,1}]+[\hat{f}_X^{0,0},\xi_{0,2}],
\end{aligned}
\end{equation}
with the corresponding equations for $\dd \hat{f}^{-1,-1}_X,\dd \hat{f}^{0,-2}_X$ given by changing the signs of the second indices.
Since  $f^{1,\pm3}_X$ does not appear in \eqref{3df-10} and  in the corresponding equations for $\dd \hat{f}^{-1,-1}_X,\dd \hat{f}^{0,-2}_X$, from \eqref{df} it follows that  $f^{1,\pm3}_X=0$.

To summarize, we can associate to each infinitesimal CR symmetry $X\in\mathcal{A}$ a collection of $\g$-valued functions $f^{i,j}_X:M_0\to\g_{i,j}$ for $(i,j)\in I(\g)$. Conversely, the conditions $\xi_{i,j}(\hat{X})=f^{i,j}_X$ for $(i,j)\in I(\m)$ along with $\xi_{0,0}(\hat{X})=0$ and \eqref{df} characterize the ``Taylor series" expansion of an infinitesimal CR symmetry $X\in\Gamma(TM_0)$, and it remains to show that such a symmetry exists for arbitrary values $u=f_X(o)\in\g$ at each $o\in M_0$. Once again following \cite{tanaka1} and \cite{yamag}, we let $u^{i,j}$ be $\g_{i,j}$-valued coordinates for the graph space $G^0\times\g$ and consider the exterior differential system  generated by the one-forms
\begin{align*}
&\alpha^{i,j}=\dd u^{i,j}-\sum_{\begin{smallmatrix}i_1+i_2=i\\j_1+j_2=j\end{smallmatrix}}[u^{i_1,j_1},\xi_{i_2,j_2}],
&(i_2,j_2)\in I(\m).
\end{align*}
Via the Maurer-Cartan equations, differentiating shows
\begin{align*}
\dd\alpha^{i,j}+\sum[\alpha^{i_1,j_1},\xi_{i_2,j_2}]\equiv0\mod\{\xi_{0,0}\}.
\end{align*}
Hence, around any $o\in M_0$ there is a neighborhood $U_o\subset M_0$ and a section $s:U_o\to G^0\times\g$, and pulling back $s^*\alpha^{i,j}$ to the graph space $M_0\times\g$ defines an integrable Pfaffian system by the Frobenius theorem. Integral manifolds of this system determine infinitesimal CR symmetries.

\subsection{Sketch of the proof of part (3) of Theorem \ref{maintheor}}
The proof is a standard consequence of the fact that the real part of the universal bigraded algebraic prolongation $\mathfrak U_{\text{bigrad}}(\g^0)$ contains a grading element $E$ with respect to the first weight, i.e. such that $[E,  y]=i y$ if $ y \in \Re\g_i$; e.g., \eqref{conf_alg} and \eqref{ESE} (where the grading element is actually $-\widehat{E}$). If a germ at a point $o$ of a $2$-nondegenerate CR structure with CR symbol $\g^0$ has infinitesimal symmetry algebra $\g$ of dimension $\dim\mathfrak U_{\text{bigrad}}(\g^0)$, then the canonical absolute parallelism assigned to it by part (1) of Theorem \ref{maintheor} has constant structure functions; i.e., the vector fields dual to the parallelism form a Lie algebra. This Lie algebra is isomorphic to the Lie algebra of infinitesimal symmetries.  Name the latter $\mathcal A$.
%the Lie algebra of germs of infinitesimal symmetries of this CR structure at the point $o$ corresponding to the coset of the identity in the group $\Re G^0$ (in fact, the choice of this point is not important here). 
The algebra $\mathcal A$ has the following natural, decreasing filtration $\{A_\kappa\}_{\kappa \in\mathbb Z, \kappa\geq -2}$ with $\mathcal A_{-2}=\mathcal A$,
$$\mathcal A_{-1}=\{X\in\mathcal A: X(o)\in D_o\}, \quad \mathcal A_{0}=\{X\in\mathcal A: X(o)\in \Re(K_o\oplus \overline K_o)\},$$
and $\mathcal A_{\kappa}$ with $\kappa>0$ are defined recursively by
$$\mathcal A_{\kappa} = \{X\in \mathcal A_{\kappa-1} : [X, \mathcal A_{-1}] \subset  \mathcal A_ {\kappa-1} \}.$$
By construction of subsection \ref{geom_prolong_part1}, the associated graded Lie algebra is isomorphic to $\mathfrak U_{\text{bigrad}}(\g^0)$. It is well-known (see \cite[Lemma 3.3]{doubkom}) that the existence of the grading element in $\mathfrak U_{\text{bigrad}}(\g^0)$ implies $\g$ and $\mathfrak U_{\text{bigrad}}(\g^0)$ are isomorphic as filtered Lie algebras, which proves the claim.


\begin{thebibliography}{99}
\footnotesize\itemsep=0pt

\bibitem{aleks}
D.V. Alekseevsky, A. Spiro,\emph{
Prolongations of Tanaka structures and regular CR structures}, Selected topics in Cauchy-Riemann geometry, 1-37,
Quad. Mat., {\bf 9}, Dept. Math., Seconda Univ. Napoli, Caserta, 2001.

\bibitem{BER99}
M.S. Baouendi, P. Ebenfelt, L. P. Rothschild, 
\emph{Real submanifolds in complex space and their mappings.} 
Princeton Mathematical Series, 47. Princeton University Press, Princeton, NJ, 1999. xii+404 pp. 

\bibitem{capshichl}
A. $\rm{\check{C}}$ap and H. Schichl, \emph{Parabolic geometries and canonical Cartan connection}, Hokkaido
Math. J. 29 (2000), 453--505.

\bibitem{capslovak}
A. $\rm{\check{C}}$ap and J. Slov${\rm \acute{a}}$k, \emph{Parabolic Geometries I:  Background and general theory}, Mathematical
Surveys and Monographs, vol. 154, American
Mathematical Society, Providence, RI, 2009.

\bibitem{cartanCR}
E. Cartan, \emph{ Sur la g\'eom\'etrie pseudo-conforme des hypersurfaces de l'espace de deux variables complexes}, Ann.
Mat. Pura Appl. 11 (1933), no. 1, 17--90.


\bibitem{chernmoserCR}
S. S. Chern and J. K. Moser, \emph{Real hypersurfaces in complex manifolds}, Acta Math. 133 (1974), 219–271.

\bibitem{doubkom}
B. Doubrov, B. Komrakov, \emph{Contact Lie algebras of vector fields on the plane}, Geometry $\&$ Topology,
Volume 3 (1999) 1–20

\bibitem{DPZ} B. Doubrov, C. Porter, I. Zelenko, \emph{Hypersurface realization and symmetry algebras of $2$-nondegenerate flat CR structures with one-dimensional Levi kernel and nilpotent regular symbol,} in preparation


\bibitem{ebenfelt}
P. Ebenfelt, \emph{Uniformly Levi degenerate CR manifolds: the 5-dimensional case}, Duke
Math. J. 110 (2001), 37–80; correction, Duke Math. J. 131 (2006), 589–591

\bibitem{felskaup1}
G. Fels,  and W. Kaup,\emph{ CR-manifolds of dimension 5: a Lie algebra approach}, J.
Reine Angew. Math. 604 (2007), 4–71.

\bibitem{felskaup2}
G. Fels and W. Kaup, \emph{Classification of Levi degenerate homogeneous CR-manifolds
in dimension 5}, Acta Math. 201 (2008), 1–82.

\bibitem{freemanfol}
M. Freeman,
\emph{Local complex foliations of real submanifolds}, Math. Ann., 209(1974), 1-30.

\bibitem{freeman}
M. Freeman,
\emph{Local biholomorphic straightening of real submanifolds}, Ann. of Math. (2) 106 (1977), no. 2, 319–352.

\bibitem{fulton}
W. Fulton, J. Harris, \emph{Representation theory. A first course.} Graduate Texts in Mathematics, 129. Readings in Mathematics. Springer-Verlag, New York, 1991. xvi+551 pp


\bibitem{Greg} J. Gregorovi\v{c}, \emph{On  equivalence problem for $2$--nondegenerate CR geometries with simple models}, arXiv preprint	arXiv:1906.00848 [math.DG], 28 pages.

\bibitem{Greg2}  J. Gregorovi\v{c}, \emph{Fundamental invariants of 2--nondegenerate CR geometries with simple models}, arxiv preprint arXiv:2007.03971 [math.DG], 27 pages.

\bibitem{gohberg}
I. Gohberg, P. Lancaster, L. Rodman, \emph{Indefinite linear algebra and applications.} Birkh\"auser Verlag, Basel, 2005. xii+357 pp.

%\bibitem{vuji}
%F. Herbut, P. Loncke, and M. Vujicic,\emph{ Canonical form for matrices under unitary
%congruence transformations II: Congruence-normal matrices}, SIAM J. Appl.
%Math. 24:794--805 (1973).

\bibitem{isaev}
A. Isaev, D. Zaitsev, \emph {Reduction of five-dimensional uniformly Levi degenerate CR structures to absolute parallelisms.} J. Geom. Anal. 23 (2013), no. 3, 1571--1605.

\bibitem{kaupzaitsev}
W Kaup and D. Zaitsev, \emph{On local CR-transformation of Levi-degenerate group
orbits in compact Hermitian symmetric spaces}, J. Eur. Math. Soc. 8 (2006), 465–490.

\bibitem{medori}
C. Medori , A. Spiro,
\emph{The equivalence problem for 5-dimensional Levi degenerate CR manifolds},  Int. Math. Res. Not. IMRN 2014, no. 20, 5602--5647.

\bibitem{medori2}
C. Medori , A. Spiro,
\emph{Structure equations of Levi degenerate CR hypersurfaces of uniform type},  Rend. Semin. Mat. Univ. Politec. Torino 73 (2015), no. 1-2, 127--150.

\bibitem{merker}
J.  Merker, \emph{Lie symmetries and CR geometry}, J. Math. Sci. 154 (2008), 817–922.


\bibitem{pocchiola}
J. Merker, S. Pocchiola \emph{Explicit absolute parallelism for $2$-nondegenerate real hypersurfaces
	$M^5\subset\mathbb C^3$
	of constant Levi rank 1}, The Journal of Geometric Analysis, 30, (2020), no. 3, 2689--2730.

\bibitem{mori} T. Morimoto, {\sl Geometric structures on
filtered manifolds}, Hokkaido Math. J.,{\bf 22}(1993), pp.
263-347.




\bibitem{porter}
C. Porter, \emph{The Local Equivalence Problem for 7-Dimensional, $2$-Nondegenerate CR Manifolds whose Cubic Form is of Conformal Unitary Type},  Comm. Anal. Geom. 27 (2019), no. 7, 1583--1638.

\bibitem{santi}
A. Santi, \emph{Homogeneous models for Levi-degenerate CR manifolds},  Kyoto J. Math. 60 (2020), no. 1, 291--334..

\bibitem{satake}
I. Satake, \emph{Algebraic Structures of Symmetric Domains}, Kan\^o Memorial Lectures 4,
Princeton Univercity Press, 1980.

\bibitem{sukhov}
A. Sukhov
\emph {Segre varieties and Lie symmetries.} 
Math.Z. 238 (2001), no. 3, 483--492.

\bibitem {Sykes} D. Sykes,\emph{Homogeneous $2$-nondegenerate CR manifolds of hypersurface type 
	%\tb{with nonregular symbol} 
	in low dimensions}, in preparation.


\bibitem{SZ20} D. Sykes, I. Zelenko, \emph{A canonical form for pairs consisting of a Hermitian form and a self-adjoint antilinear operator,} Linear Algebra Appl. 590 (2020), 
32--61.




\bibitem {SZ1} D. Sykes, I. Zelenko, \emph{On geometry of $2$-nondegenerate CR structures of hypersurface type with $1$-dimensional Levi kernel  and flag structures on leaf spaces of Levi foliation}, arXiv preprint 	arXiv:2010.02770 [math.CV], 35 pages.

\bibitem{SZ2} D. Sykes, I. Zelenko, \emph{Maximally symmetric homogeneous models for  $2$-nondegenerate CR structures of hypersurface type with 1-dimensional kernel}, arXiv preprint 	arXiv:2102.08599 [math.CV], 37 pages



\bibitem{tanakaCR}
N.~ Tanaka, \emph{On the pseudo-conformal geometry of hypersurfaces of the space of $n$ complex variables}, J. Math.
Soc. Japan 14 (1962), 397--429

\bibitem{tanaka1} N.~Tanaka, \emph{On differential systems, graded Lie
  algebras and pseudo-groups}, J.\ Math.\ Kyoto.\ Univ., \textbf{10}
  (1970), pp.~1--82.

\bibitem{tanaka2} N.~Tanaka, \emph{On the equivalence problems associated with simple graded Lie
algebras}, Hokkaido Math. J.,{\bf 6}(1979), pp 23-84.

\bibitem{uhl} A. Uhlmann, \emph{Anti-(Conjugate) Linearity}. Science China Physics, Mechanics $\&$ Astronomy 59(3) (2016) 630301. arxiv: 1507.06545v2[quant-ph]

\bibitem{yamag} K. Yamaguchi, \emph{ Differential Systems Associated
with Simple Graded Lie Algebras}, Adv. Studies in Pure
Math.,{\bf 22}(1993), pp.413-494.

\bibitem{zeltan}
I. Zelenko, {\sl On Tanaka's prolongation procedure for filtered structures of constant type} , Symmetry, Integrability and Geometry: Methods and Applications (SIGMA), Special Issue "Elie Cartan and Differential Geometry", v. 5, 2009, doi:10.3842/SIGMA.2009.094, 0906.0560 v3 [math.DG], 21 pages

\end{thebibliography}
\end{document}